\documentclass[a4paper, 12pt]{article}
\usepackage{amssymb,amsmath,amsthm,mathrsfs} 

\numberwithin{equation}{section}

\renewcommand\footnotemark{}

\newtheorem{defn}{Definition}
\newtheorem{lemma}{Lemma}
\newtheorem{thm}{Theorem}
\newtheorem{cor}{Corollary}
\newtheorem{rmk}{Remark}
\newcommand{\op}[1]{\operatorname{\text{\rm #1}}}

\begin{document}

\title{Existence and regularity of multivalued solutions to elliptic equations and systems}
\author{Brian Krummel}

\maketitle

\begin{abstract} 
We extend the work of~\cite{SW1}, in which a large class of $C^{1,\mu}$ multivalued solutions to the minimal surface equation were constructed, to produce $C^{1,\mu}$ multivalued solutions to more general classes of elliptic equations and systems, including the minimal surface system with small boundary data and the Laplace equation.  We use methods for differential equations, which are more general than the specific minimal submanifold approach adopted in~\cite{SW1}.  We also prove the branch set of the graphs of the solutions constructed~\cite{SW1} are real analytic submanifolds by inductively using Schauder estimates.
\end{abstract}


\section{Introduction} 

In~\cite{SW1}, Simon and Wickramasekera constructed a rich class of $C^{1,\mu}$ $q$-valued solutions to the Dirichlet problem for the minimal surface equation on the cylinder $\mathcal{C} = B^2_1(0) \times \mathbb{R}^{n-2}$.  However, the method of Simon and Wickramasekera was specific to the minimal surface equation and does not readily generalize to other elliptic equations or to elliptic systems.  We extend the results of~\cite{SW1} by establishing in Theorem \ref{theorem1} the existence of $C^{1,\mu}$ $q$-valued solutions with small boundary data to the Dirichlet problem for a large class of elliptic systems and in Theorem \ref{theorem2} and Corollary \ref{corollary2} the existence of $C^{1,\mu}$ $q$-valued solutions (possibly without small data) to the Dirichlet problem for large classes of elliptic equations.  In particular, we extend the results of~\cite{SW1} by giving examples of $q$-valued harmonic functions and branched minimal submanifolds with codimension greater than one.  The boundary data of these solutions satisfy a $k$-fold symmetry condition as in~\cite{SW1}.  Our approach uses techniques for differential equations, which have the advantage applying in a more general context than codimension one minimal surfaces. 
\let\thefootnote\relax\footnotetext{2010 \textit{Mathematics Subject Classification.} 35J47}

We also study the regularity of the branch set of minimal immersions.  The singular set of minimal submanifolds is known to have Hausdorff dimension at most $n-2$ in the case of area minimizing $n$-dimensional integral currents due Almgren~\cite{Almgren} and stationary graphs of $C^{1,\mu}$ two-valued functions due to Simon and Wickramasekera~\cite{SW2}.  The branch set of the minimal surfaces constructed in~\cite{SW1} and this paper are obviously $C^{1,\mu}$ $(n-2)$-dimensional submanifolds.  We extend these results by showing that the branch sets of minimal immersions constructed in~\cite{SW1} are locally real analytic $(n-2)$-dimensional submanifolds.  

The methods of differential equations require adding and multiplying functions.  However, it is not generally possible to add or multiply $q$-valued functions to obtain a $q$-valued sum or product.  To handle this difficulty we consider $q$-valued functions $\tilde u$ on an open set $\Omega$ in $\mathbb{R}^n$ each associated with a map $u = (u_1,u_2,\ldots,u_q) : \Omega \setminus [0,\infty) \times \{0\} \times \mathbb{R}^{n-2} \rightarrow (\mathbb{R}^m)^q$ such that $\tilde u(X) = \{u_1(X),u_2(X),\ldots,u_q(X)\}$ as an unordered $q$-tuple for each $X \in \Omega \setminus [0,\infty) \times \{0\} \times \mathbb{R}^{n-2}$, as we can then add and multiply the corresponding maps $u$. 

To construct $q$-valued solutions to elliptic equations and systems, we first prove Theorem \ref{poisson_thm}, which establishes the existence of $q$-valued solutions to the Dirichlet problem for a class of Poisson equations.  Using a change of variable $\xi_1+i\xi_2 = (x_1+ix_2)^{1/q}$, we transform the Poisson equation of $q$-valued functions into a singular differential equation of single-valued functions, which we can solve using Fourier analysis and standard elliptic theory.  Using the average-free and $k$-fold symmetry properties of the solution, we obtain a bound on how the solution decays at points on the axis $\{0\} \times \mathbb{R}^{n-2}$ of $\mathcal{C}$, which implies the H\"{o}lder continuity of the gradient of the solution.  The existence result for elliptic systems, Theorem \ref{theorem1}, then follows from the contraction mapping principle and the existence results for elliptic equations, Theorem \ref{theorem2} and Corollary \ref{corollary2}, follow from the Leray-Schauder theory.

The branch set of the graphs of the $q$-valued solutions $\tilde u$ constructed in~\cite{SW1} is the graph of $\tilde u$ over $\{0\} \times \mathbb{R}^{n-2}$.  Thus the real analyticity of the branch set follows in Theorem \ref{theorem3}, which establishes that a $q$-valued solution $\tilde u(x,y)$ (where $x \in \mathbb{R}^2$ and $y \in \mathbb{R}^{n-2}$) to a elliptic equation with real analytic data is real analytic with respect to $y$ in the sense that $\tilde u$ locally satisfies bounds of the form $|D^{\gamma}_y \tilde u(x,y)| \leq |\gamma|! C^{|\gamma|}$ for some constant $C \in (0,\infty)$.  Note that an analogous regularity result, Theorem \ref{theorem3_systems}, also holds for elliptic systems.  Rather than proving Theorem \ref{theorem3} by extending an approach of Morrey in~\cite{Morrey} using integral kernels, we inductively apply the Schauder estimates.  This argument readily yields $C^{1,\mu}$ estimates on derivatives $D^{\gamma}_y \tilde u$ for every multi-index $\gamma$, where $\mu \in (0,1/q)$.  More care is needed to obtain the particular type of bound on $D^{\gamma}_y \tilde u(x,y)$ required for real analyticity with respect to $y$.  We obtain such bounds using a modified version of a technique due to Friedman~\cite{Friedman} involving majorants.

\section{Preliminaries and statement of main results} \label{sec:preliminaries}

We adopt the following notation and conventions. 
\begin{enumerate}
\item[] $n \geq 3$, $m \geq 1$, and $q \geq 2$ are fixed integers.
\item[] $B_R^l(X_0)$ denotes the open ball of radius $R$ centered at $X_0$ in $\mathbb{R}^l$ and $B_R(X_0) = B_R^n(X_0)$.
\item[] $\mathcal{C} = B_1^2(0) \times \mathbb{R}^{n-2}$ denotes an open cylinder in $\mathbb{R}^n$.
\item[] $X = (x,y)$ denotes a point in $\mathbb{R}^n$, where $x \in \mathbb{R}^2$ and $y \in \mathbb{R}^{n-2}$.  We identify $x$ with the point $re^{i\theta}$ in $\mathbb{C}$, where $r \in [0,\infty)$ and $\theta \in \mathbb{R}$.
\end{enumerate}

Let $\mathcal{A}_q(\mathbb{R}^m)$ denote the space of unordered $q$-tuples $\tilde u = \{u_1,u_2,\ldots,u_q\}$, where $u_1,u_2,\ldots,u_q \in \mathbb{R}^m$ and we allow $u_i = u_j$ for $i \neq j$.  We define a metric $\mathcal{G}$ on $\mathcal{A}_q(\mathbb{R}^m)$ by 
\begin{equation*}
	\mathcal{G}(\tilde u, \tilde v) = \min_{\sigma} \left( \sum_{l=1}^q |u_l - v_{\sigma(l)}|^2 \right)^{1/2}
\end{equation*}
for all unordered $q$-tuple $\tilde u = \{u_1,\ldots,u_q\}$ and $\tilde v = \{v_1,\ldots,v_q\}$, where the minimum is taken over all permutations $\sigma$ of $\{1,\ldots,q\}$.  A $q$-valued function $\tilde u$ on a set $\Omega \subseteq \mathbb{R}^n$ is a map $\tilde u : \Omega \rightarrow \mathcal{A}_q(\mathbb{R}^m)$ (note that this definition of $q$-valued functions is equivalent to the definition of Almgren~\cite{Almgren}).  A $q$-valued function $\tilde u : \Omega \rightarrow \mathcal{A}_q(\mathbb{R}^m)$ is continuous at $X_0 \in \Omega$ if either $X_0$ is an isolated point of $\Omega$ or 
\begin{equation*}
	\lim_{X \rightarrow X_0} \mathcal{G}(\tilde u(X),\tilde u(X_0)) = 0, 
\end{equation*} 
where the limit is taken over $X \in \Omega$.  $C^0(\Omega;\mathcal{A}_q(\mathbb{R}^m))$ denotes the space of continuous $q$-valued functions $\tilde u : \Omega \rightarrow \mathcal{A}_q(\mathbb{R}^m)$.  A $q$-valued function $\tilde u : \Omega \rightarrow \mathcal{A}_q(\mathbb{R}^m)$ is H\"{o}lder continuous with exponent $\mu \in (0,1]$ if 
\begin{equation*}
	[\tilde u]_{\mu,\Omega} \equiv \sup_{X,Y \in \Omega, \, X \neq Y} \frac{\mathcal{G}(\tilde u(X), \tilde u(Y))}{|X-Y|^{\mu}} < \infty.
\end{equation*} 
A $q$-valued function $\tilde u : \Omega \rightarrow \mathcal{A}_q(\mathbb{R}^m)$ is differentiable at a point $X$ in the interior of $\Omega$ if for some $m \times n$ matrices $A_1,\ldots,A_q$, 
\begin{equation} \label{defn_derivative}
	\lim_{h \rightarrow 0} \frac{\mathcal{G}(\tilde u(X+h), \{u_l(X) + A_l h\})}{|h|} = 0, 
\end{equation} 
in which case we say $D\tilde u(X) \equiv \{A_1,\ldots,A_q\}$ is the derivative of $\tilde u$ at $X$.  For each open set $\Omega \subseteq \mathbb{R}^n$, $C^1(\Omega;\mathcal{A}_q(\mathbb{R}^m))$ denotes the space of $q$-valued functions $\tilde u : \Omega \rightarrow \mathcal{A}_q(\mathbb{R}^m)$ such that $D\tilde u$ exists at each point in $\Omega$ and $\tilde u$ and $Du$ are continuous on $\Omega$.  For each $\mu \in (0,1]$ and open set $\Omega \subseteq \mathbb{R}^n$, $C^{1,\mu}(\Omega;\mathcal{A}_q(\mathbb{R}^m))$ denotes the space of $q$-valued functions $\tilde u \in C^1(\Omega;\mathcal{A}_q(\mathbb{R}^m))$ such that $[D\tilde u]_{\mu;\Omega'} < \infty$ for every open set $\Omega' \subset \subset \Omega$. 

Let $\Omega$ be an open set in $\mathbb{R}^n$ and $\tilde u \in C^1(\Omega;\mathcal{A}_q(\mathbb{R}^m))$.  We let $\mathcal{B}_{\tilde u}$ denote the set of points $X_0 \in \Omega$ such that there is no ball $B_R(X_0) \subseteq \Omega$ on which $\tilde u = \{u_1,u_2,\ldots,u_q\}$ for some single-valued functions $u_1,u_2,\ldots,u_q \in C^1(B_R(X_0);\mathbb{R}^m)$.  We say $\tilde u$ satisfies 
\begin{equation*}
	D_i (A^i(X,\tilde u,D\tilde u)) + B(X,\tilde u,D\tilde u) = 0 \text{ weakly in } \Omega \setminus \mathcal{B}_{\tilde u} 
\end{equation*}
for continuous single-valued functions $A^i, B : \Omega \times \mathbb{R}^m \times \mathbb{R}^{mn} \rightarrow \mathbb{R}$ if for every ball $B_R(X_0) \subseteq \Omega \setminus \mathcal{B}_{\tilde u}$, $\tilde u = \{u_1,u_2,\ldots,u_q\}$ on $B_R(X_0)$ for single-valued functions $u_l \in C^1(B_R(X_0);\mathbb{R}^m)$ such that 
\begin{equation*}
	D_i (A^i(X,u_l,Du_l)) + B(X,u_l,Du_l) = 0 \text{ weakly in } B_R(X_0)
\end{equation*}
for $l = 1,2,\ldots,q$. 

Observe that we cannot in general add or multiply multivalued functions.  Given two $q$-valued functions $\tilde u, \tilde v : \Omega \rightarrow \mathcal{A}_q(\mathbb{R}^m)$, there is no canonical way to pair the elements $u_i(X)$ and $v_i(X)$ of the unordered $q$-tuples $\tilde u(X) = \{u_1(X),u_2(X),\ldots,u_q(X)\}$ and $\tilde v(X) = \{v_1(X),v_2(X),\ldots,v_q(X)\}$ to obtain a sum $(\tilde u + \tilde v)(X) = \{u_1(X) + v_1(X), u_2(X) + v_2(X), \ldots, u_q(X) + v_q(X)\}$ or product $(\tilde u \tilde v)(X) = \{u_1(X) v_1(X), u_2(X) v_2(X), \ldots, u_q(X) v_q(X)\}$.  Moreover, for some $q$-valued functions $\tilde u, \tilde v \in C^1(\Omega;\mathcal{A}_q(\mathbb{R}^m))$, there is no way to pair the elements $u_i(X)$ and $v_i(X)$ to obtain a sum $(\tilde u + \tilde v)(X) = \{u_1(X) + v_1(X), u_2(X) + v_2(X), \ldots, u_q(X) + v_q(X)\}$ that is $C^1$ on $\Omega$; for example, consider $\tilde u, \tilde v \in C^1(\mathbb{R}^2;\mathcal{A}_2(\mathbb{R}))$ given by $\tilde u(x_1,x_2) = \{ \pm \op{Re} (x_1+ix_2-1)^{3/2} \}$ and $\tilde v(x_1,x_2) = \{ \pm \op{Re} (x_1+ix_2+1)^{3/2} \}$.  In what follows, we develop a theory of multivalued solutions to linear and quasilinear elliptic differential equations, which requires adding and multiplying functions.  Rather than working with multivalued functions directly, we will work with functions $u : \Omega \setminus [0,\infty) \times \{0\} \times \mathbb{R}^{n-2} \rightarrow (\mathbb{R}^m)^q$, which we can add and multiply.  The class of functions $u$ that we consider take the form $u(X) = (u_1(X),u_2(X),\ldots,u_q(X))$ at each $X \in \Omega \setminus [0,\infty) \times \{0\} \times \mathbb{R}^{n-2}$, where $u_l : \Omega \setminus [0,\infty) \times \{0\} \times \mathbb{R}^{n-2} \rightarrow \mathbb{R}^m$ for $l = 1,2,\ldots,q$, and, roughly speaking, satisfy $\lim_{x_2 \uparrow 0} u_l(x_1,x_2,y) = \lim_{x_2 \downarrow 0} u_{l+1}(x_1,x_2,y)$ for $l = 1,2,\ldots,q-1$ and $\lim_{x_2 \uparrow 0} u_q(x_1,x_2,y) = \lim_{x_2 \downarrow 0} u_1(x_1,x_2,y)$ whenever $(0,x_2,y) \in \Omega$.  To each such map $u$ we will associate a $q$-valued function $\tilde u : \Omega \rightarrow \mathcal{A}_q(\mathbb{R}^m)$ such that $\tilde u(X) = \{u_1(X),u_2(X),\ldots,u_q(X)\}$ for $X \in \Omega \setminus [0,\infty) \times \{0\} \times \mathbb{R}^{n-2}$. 

\begin{defn} \label{defnCq}
Let $\Omega$ be an open set in $\mathbb{R}^n$ and $k \geq 0$ be an integer.  $C^{k;q}(\Omega;\mathbb{R}^m)$ denotes the set of maps $u = (u_1,u_2,\ldots,u_q) : \Omega \setminus [0,\infty) \times \{0\} \times \mathbb{R}^{n-2} \rightarrow (\mathbb{R}^m)^q$ such that $u_l |_{\Omega \cap \mathbb{R} \times (0,\infty) \times \mathbb{R}^{n-2}}$ extend to $C^k$ functions on $\Omega \cap \mathbb{R} \times [0,\infty) \times \mathbb{R}^{n-2}$ and $u_l |_{\Omega \cap \mathbb{R} \times (-\infty,0) \times \mathbb{R}^{n-2}}$ extend to $C^k$ functions on $\Omega \cap \mathbb{R} \times (-\infty,0] \times \mathbb{R}^{n-2}$ for $l = 1,2,\ldots,q$ and 
\begin{align*}
	\lim_{x_2 \uparrow 0} D^{\alpha} u_l(x_1,x_2,y) &= \lim_{x_2 \downarrow 0} D^{\alpha} u_{l+1}(x_1,x_2,y) \text{ for } l = 1,2,\ldots,q-1, \\
	\lim_{x_2 \uparrow 0} D^{\alpha} u_q(x_1,x_2,y) &= \lim_{x_2 \downarrow 0} D^{\alpha} u_1(x_1,x_2,y), 
\end{align*}
for all $x_1 \geq 0$, $y \in \mathbb{R}^{n-2}$, and $|\alpha| \leq k$.  Given $u \in C^{k;q}(\Omega;\mathbb{R}^m)$, we let 
\begin{equation*}
	D^{\alpha} u(0,y) \equiv \lim_{x \rightarrow 0} D^{\alpha} u_1(x,y)
\end{equation*}
whenever $(0,y) \in \Omega$ and $|\alpha| \leq k$.  We let $C^{\infty;q}(\Omega;\mathbb{R}^m) = \bigcap_{k=0}^{\infty} C^{k;q}(\Omega;\mathbb{R}^m)$.  $C^{k;q}_c(\Omega;\mathbb{R}^m)$ denotes the set of $u \in C^{k;q}(\Omega;\mathbb{R}^m)$ such that for some $\Omega' \subset \subset \Omega$, $u = 0$ on $(\Omega \setminus \Omega') \setminus [0,\infty) \times \{0\} \times \mathbb{R}^{n-2}$.
\end{defn}

\begin{defn}
Let $\Omega$ be an open set in $\mathbb{R}^n$, $k \geq 0$ be an integer, and $\mu \in (0,1]$.  $C^{k,\mu;q}(\Omega;\mathbb{R}^m)$ denotes the set of maps $u \in C^{k;q}(\Omega;\mathbb{R}^m)$ such that $u_l |_{\Omega \cap \mathbb{R} \times (0,\infty) \times \mathbb{R}^{n-2}}$ extend to $C^{k,\mu}$ functions on $\Omega \cap \mathbb{R} \times [0,\infty) \times \mathbb{R}^{n-2}$ and $u_l |_{\Omega \cap \mathbb{R} \times (-\infty,0) \times \mathbb{R}^{n-2}}$ extend to $C^{k,\mu}$ functions on $\Omega \cap \mathbb{R} \times (-\infty,0] \times \mathbb{R}^{n-2}$ for $l = 1,2,\ldots,q$.
\end{defn}

To each $u \in C^{k;q}(\Omega;\mathbb{R}^m)$, where $k \in \{0,1\}$, we associate a unique $q$-valued function $\tilde u \in C^k(\Omega;\mathcal{A}_q(\mathbb{R}^m))$ given by $\tilde u(X) = \{u_1(X),u_2(X),\ldots,u_q(X)\}$ for $X \in \Omega \setminus [0,\infty) \times \{0\} \times \mathbb{R}^{n-2}$.  Of course, more than one $u \in C^{k;q}(\Omega;\mathbb{R}^m)$ may be associated with the same $q$-valued function $\tilde u$. 

Let $\Omega$ be an open set in $\mathbb{R}^n$.  Given a set $S \subseteq \Omega$, we define 
\begin{align} \label{supinf}
	\inf_S u &\equiv \inf_{X \in S \setminus [0,\infty) \times \{0\} \times \mathbb{R}^{n-2}} \min \{ u_1(X), \ldots, u_q(X) \} , \nonumber \\
	\sup_S u &\equiv \sup_{X \in S \setminus [0,\infty) \times \{0\} \times \mathbb{R}^{n-2}} \max \{ u_1(X), \ldots, u_q(X) \} , 
\end{align} 
for each $u \in C^{0;q}(\Omega;\mathbb{R})$ and we define 
\begin{equation} \label{supinf2}
	\sup_S |u| \equiv \sup_{X \in S \setminus [0,\infty) \times \{0\} \times \mathbb{R}^{n-2}} \max \{ |u_1(X)|, \ldots, |u_q(X)| \} . 
\end{equation}
for each $u \in C^{0;q}(\Omega;\mathbb{R}^m)$.  Note that if instead $u : \Omega \setminus [0,\infty) \times \{0\} \times \mathbb{R}^{n-2} \rightarrow (\mathbb{R}^m)^q$ is measurable, we can define $\inf_{\Omega} u$ and $\sup_{\Omega} u$ if $m = 1$ and $\sup_S |u|$ by (\ref{supinf}) and (\ref{supinf2}) by replacing the infimums and supremums with essential infimums and supremums.  We say $u \in C^{0;q}(\Omega;\mathbb{R})$ attains its maximum value at $X_0 \in \Omega$ if either $X_0 \in \Omega \setminus [0,\infty) \times \{0\} \times \mathbb{R}^{n-2}$ and 
\begin{equation*}
	\sup_{\Omega} u = \max \{ u_1(X_0), u_2(X_0), \ldots, u_q(X_0) \} 
\end{equation*} 
or $X_0 \in \Omega \cap [0,\infty) \times \{0\} \times \mathbb{R}^{n-2}$ and 
\begin{equation*}
	\sup_{\Omega} u = \lim_{X \rightarrow X_0} \max \{ u_1(X), u_2(X), \ldots, u_q(X) \}.
\end{equation*} 
For each integer $k \geq 0$,
\begin{equation*}
	\|u\|_{C^{k;q}(\Omega)} \equiv \sum_{|\alpha| \leq k} \sup_{\Omega} |D^{\alpha} u| 
\end{equation*}
for every $u \in C^{k;q}(\Omega;\mathbb{R}^m)$.  For each integer $k \geq 0$ and $\mu \in (0,1]$, 
\begin{equation*}
	[u]_{\mu;q,\Omega} \equiv \sum_{l=1}^q ([u_l]_{\mu;\Omega \cap \mathbb{R} \times (0,\infty) \times \mathbb{R}^{n-2}} 
		+ [u_l]_{\mu;\Omega \cap \mathbb{R} \times (-\infty,0) \times \mathbb{R}^{n-2}}) 
\end{equation*}
for every $u \in C^{0,\mu;q}(\Omega;\mathbb{R}^m)$ and 
\begin{equation*}
	\|u\|_{C^{k,\mu;q}(\Omega)} \equiv \sum_{|\alpha| \leq k} \sup_{\Omega} |D^{\alpha} u| 
		+ \sum_{|\alpha| = k} \sup_{\Omega} [D^{\alpha} u]_{\mu;q,\Omega} 
\end{equation*}
for every $u \in C^{k,\mu;q}(\Omega;\mathbb{R}^m)$.  When $\Omega = B_R(X_0)$ is an open ball, we define  
\begin{equation*}
	\|u\|'_{C^{k;q}(B_R(X_0))} \equiv \sum_{|\alpha| \leq k} R^{|\alpha|} \sup_{B_R(X_0)} |D^{\alpha} u| 
\end{equation*}
for every $u \in C^{k;q}(B_R(X_0);\mathbb{R}^m)$ and 
\begin{equation*}
	\|u\|'_{C^{k,\mu;q}(B_R(X_0))} \equiv \sum_{|\alpha| \leq k} R^{|\alpha|} \sup_{B_R(X_0)} |D^{\alpha} u| 
		+ \sum_{|\alpha| = k} \sup_{\Omega} R^{k+\mu} [D^{\alpha} u]_{\mu;q,\Omega}. 
\end{equation*}
for every $u \in C^{k,\mu;q}(B_R(X_0);\mathbb{R}^m)$.  Given a sequence $\{ u^{(j)} \}_{j=1,2,3,\ldots}$ in $C^{k;q}(\Omega;\mathbb{R}^m)$ and $u \in C^{k;q}(\Omega)$, we say $u^{(j)} \rightarrow u$ in $C^{k;q}(\Omega;\mathbb{R}^m)$ if $\|u^{(j)} - u\|_{C^{k;q}(\Omega)} \rightarrow 0$.  Note that if $\Omega$ is a bounded open set in $\mathbb{R}^n$ and $u^{(j)} = (u^{(j)}_1,u^{(j)}_2,\ldots,u^{(j)}_q)$, $j=1,2,3,\ldots$, is a sequence in $C^{k,\mu;q}(\Omega;\mathbb{R}^m)$ such that $\sup_j \|u^{(j)}\|_{C^{k,\mu;q}(\Omega)} < \infty$, then by Arzela-Ascoli applied using the sequences $\{ u^{(j)}_l |_{\mu;\Omega \cap \mathbb{R} \times (0,\infty) \times \mathbb{R}^{n-2}} \}$ and $\{ u^{(j)}_l |_{\mu;\Omega \cap \mathbb{R} \times (-\infty,0) \times \mathbb{R}^{n-2}} \}$ for $l = 1,2,\ldots,q$, there is a subsequence $\{ u^{(j_i)} \}_{i=1,2,3,\ldots}$ of $\{ u^{(j)} \}_{j=1,2,3,\ldots}$ and $u \in C^{k,\mu;q}(\Omega;\mathbb{R}^m)$ such that $u^{(j_i)} \rightarrow u$ in $C^{k;q}(\Omega;\mathbb{R}^m)$ as $i \rightarrow \infty$. 

Given open sets $\Omega' \subset \subset \Omega \subseteq \mathbb{R}^n$, $h \in \mathbb{R}$ and $\eta \in \mathbb{R}^{n-2}$ such that $0 < |h \eta| < \op{dist}(\Omega',\partial \Omega)$, and $u \in C^{0;q}(\Omega;\mathbb{R}^m)$, we define $\delta_{h,\eta} u = (\delta_{h,\eta} u_1,\delta_{h,\eta} u_2,\ldots,\delta_{h,\eta} u_q) \in C^{0;q}(\Omega';\mathbb{R}^m)$ by 
\begin{equation} \label{diffquot}
	\delta_{h,\eta} u_l(x,y) \equiv \frac{u_l(x,y+h\eta) - u_l(x,y)}{h} 
\end{equation} 
for all $(x,y) \in \Omega' \setminus [0,\infty) \times \{0\} \times \mathbb{R}^{n-2}$ and $l = 1,2,\ldots,q$.

\begin{defn}
Let $\Omega \subseteq \mathbb{R}^n$.  For $1 \leq p < \infty$, $L^{p;q}(\Omega;\mathbb{R}^m)$ denotes the set of Lebesgue measurable functions $u = (u_1,u_2,\ldots,u_q) : \Omega \rightarrow (\mathbb{R}^m)^q$ such that 
\begin{equation*}
	\|u\|_{L^{p;q}(\Omega)} \equiv \left( \int_{\Omega} \sum_{l=1}^q |u_l|^p \right)^{1/p} < \infty. 
\end{equation*}
$L^{\infty;q}(\Omega;\mathbb{R}^m)$ denotes the set of Lebesgue measurable functions $u = (u_1,u_2,\ldots,u_q) : \Omega \rightarrow (\mathbb{R}^m)^q$ such that $\sup_{\Omega} |u_l| < \infty$ for $l = 1,2,\ldots,q$.  
\end{defn}

\begin{defn}
Let $\Omega \subseteq \mathbb{R}^n$ be an open set, $1 \leq p < \infty$ and $k \geq 1$ be an integer.  $W^{k,p;q}(\Omega;\mathbb{R}^m)$ denotes the set of $u = (u_1,u_2,\ldots,u_q) \in L^{p;q}(\Omega;\mathbb{R}^m)$ such that for every $\alpha$ with $|\alpha| \leq k$ there exists a $v = (v_1,v_2,\ldots,v_q) \in L^{p;q}(\Omega;\mathbb{R}^m)$ (depending on $\alpha$) such that 
\begin{equation*}
	\int_{\Omega} \sum_{l=1}^q u_l D^{\alpha} \zeta_l = (-1)^{|\alpha|} \int_{\Omega} \sum_{l=1}^q v_l \zeta_l
\end{equation*}
for every $\zeta = (\zeta_1,\zeta_2,\ldots,\zeta_q) \in C_c^{k;q}(\Omega;\mathbb{R}^m)$.  $D^{\alpha} u \equiv v$ denotes the order $\alpha$ weak derivative of $u$. 
\end{defn}

To each measurable function $u : \Omega \rightarrow (\mathbb{R}^m)^q$ we associate a measurable $q$-valued function $\tilde u : \Omega \rightarrow \mathcal{A}_q(\mathbb{R}^m)$ given by $\tilde u(X) = \{u_1(X),u_2(X),\ldots,u_q(X)\}$ for $X \in \Omega$.  $\tilde u$ is unique up to its values $\mathcal{L}^n$-a.e. on $\Omega$.  

For each integer $k \geq 1$ and $1 \leq p < \infty$, 
\begin{equation*}
	\|u\|_{W^{k,p;q}(\Omega)} \equiv \sum_{|\alpha| \leq k} \sup_{\Omega} \|D^{\alpha} u\|_{L^{p;q}(\Omega)} 
\end{equation*}
for every $u \in W^{k,p;q}(\Omega;\mathbb{R}^m)$.  Given a sequence $u^{(j)} \in L^{p;q}(\Omega)$, $j = 1,2,3,\ldots$, and $u \in L^{p;q}(\Omega)$, we say $u^{(j)} \rightarrow u$ in $L^{p;q}(\Omega;\mathbb{R}^m)$ if $\|u^{(j)} - u\|_{L^{p;q}(\Omega)} \rightarrow 0$.  Given a sequence $\{ u^{(j)} \}_{j=1,2,3,\ldots}$ in $W^{k,p;q}(\Omega)$ and $u \in W^{k,p;q}(\Omega)$, we say $u^{(j)} \rightarrow u$ in $W^{k,p;q}(\Omega;\mathbb{R}^m)$ if $\|u^{(j)} - u\|_{W^{k,p;q}(\Omega)} \rightarrow 0$.  We let $W_0^{k,p;q}(\Omega;\mathbb{R}^m)$ denote the closure of $C_c^{k;q}(\Omega;\mathbb{R}^m)$ in the Banach space $W^{k,p;q}(\Omega;\mathbb{R}^m)$.  Note that if $u^{(j)} = (u^{(j)}_1,u^{(j)}_2,\ldots,u^{(j)}_q)$, $j=1,2,3,\ldots$, is a sequence in $W^{1,2;q}(\Omega;\mathbb{R}^m)$ such that $\sup_j \|u^{(j)}\|_{W^{1,2;q}(\Omega)} < \infty$, then by Rellich's compactness lemma applied to the sequences $\{u^{(j)}_l |_{\mu;\Omega \cap \mathbb{R} \times (0,\infty) \times \mathbb{R}^{n-2}}\}$ and $\{u^{(j)}_l |_{\mu;\Omega \cap \mathbb{R} \times (-\infty,0) \times \mathbb{R}^{n-2}}\}$ for $l = 1,2,\ldots,q$, there is a subsequence $\{ u^{(j_i)} \}_{i=1,2,3,\ldots}$ of $\{ u^{(j)} \}_{j=1,2,3,\ldots}$ and $u \in W^{1,2;q}(\Omega;\mathbb{R}^m)$ such that $u^{(j_i)} \rightarrow u$ strongly in $L^{2;q}(\Omega;\mathbb{R}^m)$ as $i \rightarrow \infty$ and $\|Du\|_{L^{2;q}(\Omega)} \leq \liminf_j \|Du^{(j)}\|_{L^{2;q}(\Omega)}$. 

Given a set $\Omega \subseteq \mathbb{R}^n$ and $u = (u_1,u_2,\ldots,u_q) : \Omega \setminus [0,\infty) \times \{0\} \times \mathbb{R}^{n-2} \rightarrow (\mathbb{R}^m)^q$, there exists $u_a : \Omega \setminus [0,\infty) \times \{0\} \times \mathbb{R}^{n-2} \rightarrow \mathbb{R}^m$ and $u_f = (u_{f,1},u_{f,2},\ldots,u_{f,q}) : \Omega \setminus [0,\infty) \times \{0\} \times \mathbb{R}^{n-2} \rightarrow (\mathbb{R}^m)^q$ such that 
\begin{equation} \label{averageandfree}
	u_l = u_a + u_{f,l} \text{ for } l = 1,2,\ldots,q, \hspace{3mm} \text{where} \hspace{3mm} u_a = \frac{1}{q} \sum_{j=1}^q u_j. 
\end{equation}
We call $u_a$ the average of $u$.  We say $u$ is average-free if $u_a = 0$ on $\Omega$.  $u_f$ is average-free and thus we call $u_f$ the average-free part of $u$. 

The first of our main results concern the existence of solutions to the Dirichlet problem for elliptic differential equations in the cylinder $\mathcal{C} = B^2_1(0) \times \mathbb{R}^{n-2}$.  Fix an integer $k \geq 2$ such that $k$ and $q$ are relatively prime.  We say $u \in C^{0;q}(\overline{\mathcal{C}};\mathbb{R}^m)$ is \textit{$k$-fold symmetric} if 
\begin{align*}
	u_l(re^{i\theta+i2\pi/k}, y) &= u_l(re^{i\theta}, y) \text{ if } 0 < \theta < 2\pi - 2\pi/k, \, l = 1,2,\ldots,q, \\
	u_1(re^{i\theta+i2\pi/k}, y) &= u_q(re^{i\theta}, y) \text{ if } 2\pi - 2\pi/k < \theta < 2\pi, \\
	u_l(re^{i\theta+i2\pi/k}, y) &= u_{l-1}(re^{i\theta}, y) \text{ if } 2\pi - 2\pi/k < \theta < 2\pi, \, l = 2,3,\ldots,q, 
\end{align*} 
for all $(re^{i\theta},y) \in \Omega$.  We will let $\mathbf{R}$ denote the $n \times n$ matrix such that $\mathbf{R}(re^{i\theta},y) = (re^{i\theta+i2\pi/k},y)$.  We write 
\begin{equation*}
	\mathbf{R} = (R^i_j)_{i,j=1,\ldots,n} = \left( \begin{matrix} 
		\cos(2\pi/k) & -\sin(2\pi/k) & 0 & 0 & \cdots & 0 \\
		\sin(2\pi/k) & \cos(2\pi/k) & 0 & 0 & \cdots & 0 \\
		0 & 0 & 1 & 0 & \cdots & 0 \\
		0 & 0 & 0 & 1 & \cdots & 0 \\
		\vdots & \vdots & \vdots & \vdots & \ddots & \vdots \\
		0 & 0 & 0 & 0 & \cdots & 1 \\
	\end{matrix} \right) 
\end{equation*}
where $R_i^j$ denotes the entry in the $i$-th row and $j$-th column of $\mathbf{R}$.  We say $u$ is \textit{periodic} with respect to $y_j$ with period $\rho_j > 0$ for $j = 1,2,\ldots,n-2$ if 
\begin{equation*}
	u_l(x,y + \rho_j e_j) = u_l(x,y) 
\end{equation*}
for all $(x,y) \in \Omega \setminus [0,\infty) \times \{0\} \times \mathbb{R}^{n-2}$, $l = 1,2,\ldots,q$, and $j = 1,2,\ldots,n-2$, where $e_1,e_2,\ldots,e_{n-2}$ denotes the standard basis for $\mathbb{R}^{n-2}$.  

We will be interested in the regularity of multivalued solutions up to the boundary of $\mathcal{C}$.  Recall that the continuity of multivalued functions on $\overline{\mathcal{C}}$ is defined above.  We say $\tilde u : \overline{\mathcal{C}} \rightarrow \mathcal{A}_q(\mathbb{R}^m)$ is differentiable at $X_0 \in \partial \mathcal{C}$ if (\ref{defn_derivative}) holds when the limit is taken over $h$ such that $X+h \in \overline{\mathcal{C}}$.  We let $C^1(\overline{\mathcal{C}};\mathcal{A}_q(\mathbb{R}^m))$ denote the space of $\tilde u : \overline{\mathcal{C}} \rightarrow \mathcal{A}_q(\mathbb{R}^m)$ that are continuously differentiable on $\overline{\mathcal{C}}$.  We say $\tilde u \in C^{1,\mu}(\overline{\mathcal{C}};\mathcal{A}_q(\mathbb{R}^m))$ for $\mu \in (0,1]$ if $\tilde u \in C^1(\overline{\mathcal{C}};\mathcal{A}_q(\mathbb{R}^m))$ and $[D\tilde u]_{\mu;B^2_1(0) \times B^{n-2}_{\rho}(0)} < \infty$ for all $\rho \in (0,\infty)$.  We define $C^{k;q}(\overline{\mathcal{C}};\mathbb{R}^m)$ for integers $k \geq 0$ by Definition \ref{defnCq} with $\overline{\mathcal{C}}$ in place of $\Omega$.  We say $u \in C^{k,\mu;q}(\overline{\mathcal{C}};\mathbb{R}^m)$ for an integer $k \geq 0$ and $\mu \in (0,1)$ if $u \in C^{k;q}(\overline{\mathcal{C}};\mathbb{R}^m)$ and $[D^k u]_{\mu;q,B^2_1(0) \times B^{n-2}_{\rho}(0)} < \infty$ for all $\rho \in (0,\infty)$.  Note that given a set $S \subseteq \overline{\mathcal{C}}$, we define $\inf_S u$ and $\sup_S u$ for $u \in C^0(\overline{\mathcal{C}})$ by (\ref{supinf}) and $\sup_S |u|$ for $u \in C^0(\overline{\mathcal{C}};\mathbb{R}^m)$ by (\ref{supinf2}).

We will first establish the existence of solutions in $C^{0;q}(\overline{\mathcal{C}}) \cap C^{1,\mu;q}(\mathcal{C})$ to weak Poisson equations:

\begin{thm} \label{poisson_thm}
Let $\mu \in (0,1/q)$ and $k > q$ be an integer such that $k$ and $q$ are relatively prime.  Given $f^j = (f^j_1,f^j_2,\ldots,f^j_q) \in C^{0,\mu;q}(\overline{\mathcal{C}})$ and $g, \varphi \in C^{0;q}(\overline{\mathcal{C}})$ such that 
\begin{align} \label{poisson_symmetry}
	f^j_l(re^{i\theta+i2\pi/k}, y) &= \sum_{p=1}^n R^j_p f^p_l(re^{i\theta}, y) \text{ if } 0 < \theta < 2\pi - 2\pi/k, \, l = 1,2,\ldots,q, \nonumber \\
	f^j_1(re^{i\theta+i2\pi/k}, y) &= \sum_{p=1}^n R^j_p f^p_q(re^{i\theta}, y) \text{ if } 2\pi - 2\pi/k < \theta < 2\pi, \nonumber \\
	f^j_l(re^{i\theta+i2\pi/k}, y) &= \sum_{p=1}^n R^j_p f^p_{l-1}(re^{i\theta}, y) \text{ if } 2\pi - 2\pi/k < \theta < 2\pi, \, l = 2,3,\ldots,q, 
\end{align} 
for all $(re^{i\theta},y) \in \mathcal{C}$, $g$ and $\varphi$ are $k$-fold symmetric, and 
\begin{equation*}
	\sup_{\partial \mathcal{C}} |\varphi| + [f]_{\mu;q,\mathcal{C}} + \sup_{\mathcal{C}} |g| < \infty, 
\end{equation*}
there is a $u \in C^{0;q}(\overline{\mathcal{C}}) \cap C^{1,\mu;q}(\mathcal{C})$ such that $u$ is $k$-fold symmetric, 
\begin{align} \label{poisson_thm_eqn}
	\int_{\mathcal{C}} \sum_{l=1}^q D_j u_l D_j \zeta_l &= \int_{\mathcal{C}} \sum_{l=1}^q (f^j_l D_j \zeta_l - g_l \zeta_l) 
		\text{ for all } \zeta \in C_c^{1;q}(\mathcal{C} \setminus \{0\} \times \mathbb{R}^{n-2}), \nonumber \\
	u_l &= \varphi_l \quad \text{ on } \partial \mathcal{C}, 
\end{align}
for $l = 1,2,\ldots,q$ and 
\begin{equation*}
	\sup_{\mathcal{C}} |u| \leq C \left( \sup_{\partial \mathcal{C}} |\varphi| + [f]_{\mu;q,\mathcal{C}} + \sup_{\mathcal{C}} |g| \right) .
\end{equation*}

Moreover, if $f^j$, $g$, and $\varphi$ are periodic with respect to $y_i$ with period $\rho_i > 0$ for $i = 1,2,\ldots,n-2$, then $u$ is the unique solution to (\ref{poisson_thm_eqn}) that is periodic with respect to $y_i$ with period $\rho_i$ for $i = 1,2,\ldots,n-2$. 
\end{thm}

Note that in the special case where $f^j = 0$ and $g = 0$, the $q$-valued function $\tilde u(X) = \{u_1(X),u_2(X),\ldots,u_q(X)\}$ associated with the solution $u$ obtained in Theorem \ref{poisson_thm} is a $q$-valued function in $C^0(\overline{\mathcal{C}};\mathcal{A}_q(\mathbb{R})) \cap C^{1,\mu}(\mathcal{C};\mathcal{A}_q(\mathbb{R}))$ such that $\Delta \tilde u = 0$ weakly in $\mathcal{C} \setminus \mathcal{B}_{\tilde u}$. 

To prove Theorem \ref{poisson_thm}, we first assume $f^j$, $g$, and $\varphi$ are periodic with respect to each $y_i$, as the general result follows by approximation of $f^j$, $g$, and $\varphi$.  We use the change of variable $\xi_1+i\xi_2 = (x_1+ix_2)^{1/q}$ to transform $u(x_1,x_2,y)$ to a single-valued function $u_0(\xi_1,\xi_2,y)$ defined by $u_0(re^{i\theta},y) = u_l(r^{1/q} e^{i\theta/q},y)$ for $r \in [0,1]$, $\theta \in (2(l-1)\pi,2l\pi)$, and $y \in \mathbb{R}^{n-2}$.  The single-valued function $u_0$ satisfies a singular differential equation which we solve using Fourier series with respect to the $y_i$ variables and the existence theory for single-valued solutions to elliptic equations to solve for the Fourier coefficients as functions of $\xi_1$ and $\xi_2$.  By linearity, we can assume that $f^j$, $g$, and $\varphi$ are all average-free and therefore the constructed solution $u$ will be average-free and $k$-fold symmetric.  The average-free and $k$-fold symmetry conditions on $u$ will guarantee that $u(x,y)$ decays sufficiently quickly as $x$ approaches zero to guarantee that $u \in C^{1,\mu;q}(\mathcal{C})$.  

Using Theorem \ref{poisson_thm} and the contraction mapping principle, we can construct solutions to quasilinear elliptic systems with small boundary data $\varphi$ in $C^{1,\mu;q}(\overline{\mathcal{C}};\mathbb{R}^m)$:

\begin{thm} \label{theorem1}
Let $m \geq 1$ be an integer, $\mu \in (0,1/q)$, and $k > q$ be an integer such that $k$ and $q$ are relatively prime.  Let $F^i_{\kappa} \in C^2(\mathbb{R}^{mn})$ and $G_{\kappa} \in C^1(\mathbb{R}^m \times \mathbb{R}^{mn})$ be single-valued functions, where $\mathbb{R}^{mn}$ is the space of $m \times n$ matrices, such that $F^i_{\kappa}(0) = 0$, $DF^i_{\kappa}(0) = 0$, $G_{\kappa}(0,0) = 0$, $DG_{\kappa}(0,0) = 0$, and 
\begin{equation*}
	F^i_{\kappa}(P \mathbf{R}) = R^j_i F^j_{\kappa}(P), \quad G_{\kappa}(Z, P \mathbf{R}) = G_{\kappa}(Z,P) 
\end{equation*}
for all $P \in \mathbb{R}^{mn}$.  For some $\varepsilon > 0$ depending on $m$, $n$, $q$, $\mu$, $F^i_{\kappa}$, and $G_{\kappa}$, if $\varphi \in C^{1,\mu;q}(\overline{\mathcal{C}},\mathbb{R}^m)$ is $k$-fold symmetric with $\|\varphi\|_{C^{1,\mu;q}(\mathcal{C})} \leq \varepsilon$, there exists a solution $u \in C^{1,\mu;q}(\overline{\mathcal{C}};\mathbb{R}^m)$ to 
\begin{align} \label{theorem1_dirprob}
	\int_{\mathcal{C}} \sum_{l=1}^q D_j u_l^{\kappa} D_j \zeta_l^{\kappa} 
		&= \int_{\mathcal{C}} \sum_{l=1}^q (F^i_{\kappa}(Du_l) D_i \zeta^{\kappa} - G_{\kappa}(u_l,Du_l) \zeta^{\kappa}) \text{ for all } 
		\zeta \in C_c^{1;q}(\mathcal{C} \setminus \{0\} \times \mathbb{R}^{n-2};\mathbb{R}^m), \nonumber \\
	u_l &= \varphi_l \text{ on } \partial \mathcal{C} \text{ for } l = 1,2,\ldots,q. 
\end{align}
Moreover, $u$ is $k$-fold symmetric and $\|u\|_{C^{1,\mu;q}(\mathcal{C})} \leq \varepsilon$.  The $q$-valued function $\tilde u(X) = \{u_1(X),$ $u_2(X),\ldots, u_q(X)\}$ associated with $u$ is a $q$-valued solution in $C^{1,\mu}(\overline{\mathcal{C}};\mathcal{A}_q(\mathbb{R}^m))$ to 
\begin{equation*}
	\Delta \tilde u^{\kappa} - D_i F^i_{\kappa}(D\tilde u) - G_{\kappa}(\tilde u,D\tilde u) = 0 \text{ weakly in } \mathcal{C} \setminus \mathcal{B}_{\tilde u}. 
\end{equation*}
\end{thm}

In particular, Theorem \ref{theorem1} yields $q$-valued solutions $\tilde u \in C^{1,\mu}(\mathcal{C};\mathcal{A}_q(\mathbb{R}^m))$ to the minimal surface system in $\mathcal{C} \setminus \mathcal{B}_{\tilde u}$.  For sufficiently small $\varepsilon > 0$, these solutions to the minimal surface system are stable in the sense that 
\begin{equation} \label{stability}
	\int_{\Sigma_{\tilde u}} \left( \sum_{i=1}^n |(D_{\tau_i} \mathbf{X})^{\perp}|^2 - \sum_{i,j=1}^n |\mathbf{X} \cdot A(\tau_i,\tau_j)|^2 \right) \geq 0 
\end{equation}
for all normal vector fields $\mathbf{X} \in C^0_c(\Sigma_{\tilde u};\mathbb{R}^{n+m}) \cap W^{1,2}(\Sigma_{\tilde u}, \mathbb{R}^{n+m})$, where $\Sigma_{\tilde u}$ is the graph of $\tilde u$ regarded as an immersed submanifold.  (\ref{stability}) holds true in the case that $\mathbf{X} = 0$ near $\{0\} \times \mathbb{R}^{n-2+m}$ by the convexity of the area functional.  To prove (\ref{stability}) for general $\mathbf{X}$, for $\delta \in (0,1)$ let $\chi_{\delta} \in C^1([0,\infty))$ be the logarithmic cutoff function given by $\chi_{\delta} = 0$ on $B^2_{\delta^2}(0) \times \mathbb{R}^{n-2-m}$, $\chi_{\delta}(x,y,Z) = -\log(|x|/\delta^2)/\log(\delta)$ if $x \in B^2_{\delta}(0) \setminus B^2_{\delta^2}(0)$, $y \in \mathbb{R}^{n-2}$, and $Z \in \mathbb{R}^m$, and $\chi_{\delta} = 1$ on $\mathbb{R}^{n+m} \setminus B^2_{\delta}(0) \times \mathbb{R}^{n-2+m}$.  Replace $\mathbf{X}$ by $\chi_{\delta} \mathbf{X}$ in (\ref{stability}) and let $\delta \downarrow 0$ to obtain (\ref{stability}) with the original $\mathbf{X}$.  Theorem \ref{theorem1} also yields $q$-valued solutions to the Euler-Lagrange equations for functionals of the form $\int_{\mathcal{C}} (|Du|^2 + f(Du))$ where $f \in C^3(\mathbb{R}^{mn};\mathbb{R})$ is a single-valued function such that $Df(0) = 0$, $D^2 f(0) = 0$, and $f(P \mathbf{R}_{2\pi/k}) = f(P)$ for all $P \in \mathbb{R}^{mn}$.  

Note that Theorem \ref{theorem1} would not be true without the assumption of small boundary data as a consequence of~\cite{3Non}, which showed that for some boundary data there are no $C^1$ single-valued solutions to the Dirichlet problem for the minimal surface system.

We also use Theorem \ref{poisson_thm} and the Leray-Schauder theory to construct solutions to general quasilinear elliptic equations (without assuming small boundary data):

\begin{thm} \label{theorem2}
Let $k > q$ be an integer such that $k$ and $q$ are relatively prime.  Let $A^i \in C^2(\mathbb{R}^n)$, $B \in C^1(\mathbb{R} \times \mathbb{R}^n)$ be single-valued functions such that 
\begin{equation} \label{theorem2_symmetry}
	A^i(P \mathbf{R}) = R^j_i A^j(P), \quad B(Z,P \mathbf{R}) = B(Z,P). 
\end{equation}
Suppose  
\begin{equation*}
	0 < \lambda(P) |\xi|^2 \leq D_{P_j} A^i(P) \xi_i \xi_j \leq \Lambda(P) |\xi|^2 \quad \text{for all } \xi \in \mathbb{R}^n 
\end{equation*}
for some continuous positive functions $\lambda$ and $\Lambda$, the structure conditions  
\begin{align} 
	\label{theorem2_structure1} B(Z,P) \op{sgn} Z/\lambda(P) &\leq \beta_1 |P| + \beta_2, \\
	\label{theorem2_structure2} |\Lambda(P)| + |B(Z,P)| &\leq \beta_3 \lambda(P) |P|^2 \text{ if } |P| \geq 1, 
\end{align}
for some constants $\beta_1, \beta_2, \beta_3 \in (0,\infty)$, and $B(Z,P)$ is non-increasing in $Z$ for fixed $P \in \mathbb{R}^n$.  Let $\varphi \in C^{2;q}(\overline{\mathcal{C}})$ is $k$-fold symmetric with $\|\varphi\|_{C^{2;q}(\overline{\mathcal{C}})} < \infty$.  Then there exists a $u \in C^{1;q}(\overline{\mathcal{C}})$ such that $u \in C^{1,\mu;q}(\overline{\mathcal{C}})$ for every $\mu \in (0,1/q)$ and 
\begin{align} \label{theorem2_dirprob}
	\int_{\mathcal{C}} \sum_{l=1}^q (A^i(Du_l) D_i \zeta_l - B(u_l,Du_l) \zeta_l) &= 0
		\text{ for all } \zeta \in C_c^{1;q}(\mathcal{C} \setminus \{0\} \times \mathbb{R}^{n-2}), \nonumber \\
	u_l &= \varphi_l \text{ on } \partial \mathcal{C} \text{ for } l = 1,2,\ldots,q. 
\end{align}
Moreover, $u$ is $k$-fold symmetric.  The $q$-valued function $\tilde u(X) = \{u_1(X),u_2(X),\ldots,u_q(X)\}$ associated with $u$ satisfies $\tilde u \in C^{1,\mu}(\overline{\mathcal{C}};\mathcal{A}_q(\mathbb{R}))$ for all $\mu \in (0,1/q)$ and 
\begin{equation} \label{theorem2_diffeqn}
	D_i A^i(D\tilde u) + B(\tilde u,D\tilde u) = 0 \text{ weakly in } \mathcal{C} \setminus \mathcal{B}_{\tilde u}. 
\end{equation}
\end{thm}

Note that to prove Theorem \ref{theorem2} we need a new $C^{1,\tau;q}$ Schauder estimate (Lemma \ref{schauder_strong} in Section~\ref{sec:elliptictheory}) in order to construct a compact map to apply the Leray-Schauder theory.

The proof of Theorem \ref{theorem2} uses the maximum principle to obtain a global gradient estimate.  By obtaining interior gradient estimates via~\cite{Simon_interior} and using an approximation argument, we can assume $\varphi \in C^{0;q}(\overline{\mathcal{C}})$.  See Section 4 of~\cite{Simon_interior} for other examples of structural conditions on $A^i$ and $B$ that imply interior gradient estimates. 

\begin{cor} \label{corollary2}
Let $k > q$ be an integer such that $k$ and $q$ are relatively prime.  Let $A^i \in C^2(\mathbb{R}^n)$, $B \in C^1(\mathbb{R} \times \mathbb{R}^n)$ be single-valued functions satisfying (\ref{theorem2_symmetry}).  Let $v(P) = (1+|P|^2)^{1/2}$ and $g^{ij}(P) = \delta^{ij} - P_i P_j/(1+|P|^2)$ and suppose the structure conditions (\ref{theorem2_structure1}), 
\begin{gather*} 
	P_i A^i(P) \geq v(P) - \gamma_1, \hspace{5mm} |A(P)| \leq \gamma_2, \hspace{5mm} |B(Z,P)| \leq \gamma_2/v(P), \\
	v(P) D_{P_j} A^i(P) \xi_j \xi_j \geq g^{ij}(P) \xi_i \xi_j \text{ for all } \xi \in \mathbb{R}^n, \\
	v(P) |D_{P_j} A^i(P) \xi_i \eta_j| \leq \gamma_2 (g^{ij}(P) \xi_i \xi_j)^{1/2} (g^{ij}(P) \eta_i \eta_j)^{1/2} 
		\text{ for all } \xi, \eta \in \mathbb{R}^n, 
\end{gather*}
hold for all $Z \in \mathbb{R}$ and $P \in \mathbb{R}^n$ for some constants $\beta_1, \beta_2, \beta_3 \in (0,\infty)$, $\gamma_1 \in [0,1)$, and $\gamma_2 \in (0,\infty)$.  Also suppose $B(Z,P)$ is non-increasing in $Z$ for fixed $P \in \mathbb{R}^n$.  Let $\varphi \in C^{0;q}(\overline{\mathcal{C}})$ is $k$-fold symmetric with $\sup_{\partial \mathcal{C}} |\varphi| < \infty$.  Then there exists a solution $u \in C^{0;q}(\overline{\mathcal{C}}) \cap C^{1;q}(\mathcal{C})$ to (\ref{theorem2_dirprob}) such that $u$ is $k$-fold symmetric and $u \in C^{1,\mu;q}(\overline{\mathcal{C}})$ for every $\mu \in (0,1/q)$.  The $q$-valued function $\tilde u(X) = \{u_1(X),u_2(X),\ldots,u_q(X)\}$ associated with $u$ is a $q$-valued solution in $C^0(\overline{\mathcal{C}};\mathcal{A}_q(\mathbb{R}))$ and $C^{1,\mu}(\mathcal{C};\mathcal{A}_q(\mathbb{R}))$ for all $\mu \in (0,1/q)$ to (\ref{theorem2_diffeqn}). 
\end{cor}

Finally we consider the interior regularity of $q$-valued solutions to elliptic equations: 

\begin{thm} \label{theorem3}
Let $\tilde u \in C^1(B_1(0);\mathcal{A}_q(\mathbb{R}))$ be a $q$-valued function such that $\mathcal{B}_{\tilde u}$ is nonempty, $\mathcal{B}_{\tilde u} \subseteq \{0\} \times B^{n-2}_1(0)$, and $\|\tilde u\|_{C^1(B_1(0))} \leq 1/2$.  Suppose $\tilde u$ is a solution to 
\begin{equation} \label{theorem3_equation0}
	D_i (A^i(X,\tilde u,D\tilde u)) + B(X,\tilde u,D\tilde u) = 0 \text{ weakly in } B_1(0) \setminus \mathcal{B}_{\tilde u}, 
\end{equation}
for some locally real analytic single-valued functions $A^i, B : B_1(0) \times (-1,1) \times B^n_1(0) \rightarrow \mathbb{R}$ and 
\begin{equation} \label{theorem3_ellipticity}
	(D_j A^i)(X,Z,P) \xi_i \xi_j \geq \lambda |\xi|^2 \text{ for } (X,Z,P) \in B_1(0) \times (-1,1) \times B^n_1(0), \, \xi \in \mathbb{R}^n
\end{equation} 
for some constant $\lambda > 0$.  Then $\tilde u(x,y)$ is real analytic in $y$ in the sense that for $B_R(x_0,y_0) \subset \subset B_1(0)$, 
\begin{equation*}
	\sup_{(x,y) \in B_{R/2}(x_0,y_0)} |D_y^{\gamma} \tilde u(x,y)| \leq p! C^p R^{-p} \text{ for } p = |\gamma| \geq 1
\end{equation*}
for some constant $C \in (0,\infty)$ depending on $n$, $q$, $\mu$, $A^i$, $B$, and $\mathcal{B}_{\tilde u}$.  Consequently, the branch set of the graph of $\tilde u$ is a union of $N \leq q/2$ real analytic, $(n-2)$-dimensional submanifolds. 
\end{thm}

In particular, (\ref{theorem3}) establishes that the branch sets of the minimal hypersurfaces constructed in~\cite{SW1} are locally real analytic $(n-2)$-dimensional submanifolds.  

Suppose $\tilde u$ is as in the statement of Theorem \ref{theorem3}.  Then $\tilde u = \{u_1,u_2,\ldots,u_q\}$ on $B_1(0) \setminus (0,\infty) \times \{0\} \times \mathbb{R}^{n-2}$ for some $C^1$ single-valued $u_1, u_2,\ldots, u_q$ such that $D_i A^i(X,u_l,Du_l) + B(X,u_l,Du_l) = 0$ weakly in $B_1(0) \setminus [0,\infty) \times \{0\} \times \mathbb{R}^{n-2}$ and $\tilde u = \{v_1,v_2,\ldots,v_q\}$ on $B_1(0) \setminus (-\infty,0) \times \{0\} \times \mathbb{R}^{n-2}$ for some $C^1$ single-valued $v_1,v_2,\ldots,v_q$ such that $D_i A^i(X,v_l,Dv_l) + B(X,v_l,Dv_l) = 0$ weakly in $B_1(0) \setminus (-\infty,0] \times \{0\} \times \mathbb{R}^{n-2}$.  By unique continuation, we can order $v_1,v_2,\ldots,v_q$ so that $u_l = v_l$ on $B_1(0) \cap \mathbb{R} \times (-\infty,0) \times \mathbb{R}^{n-2}$.  Moreover, there exists a permutation $\sigma$ of $\{1,2,\ldots,q\}$ such that $u_{\sigma(l)} = v_l$ on $B_1(0) \cap \mathbb{R} \times (0,\infty) \times \mathbb{R}^{n-2}$.  After reordering $u_1,u_2,\ldots,u_q$, we may assume that $\sigma = (1,2,\ldots,i_1) (i_1+1,i_1+2,\ldots,i_2) \cdots (i_{N-1},i_{N-1}+1,\ldots,q)$ for some integers $i_j$ so that $(u_{i_{j-1}+1},u_{i_{j-1}+2},\ldots,u_{i_j}) \in C^{1;i_j-i_{j-1}}(B_1(0))$ for $j = 1,2,\ldots,N$, where $i_0 = 0$ and $i_N = q$.  To prove Theorem \ref{theorem3}, it suffices to assume $N = 1$ so that $u = (u_1,u_2,\ldots,u_q) \in C^{1;q}(B_1(0))$ and show that 
\begin{equation} \label{theorem3_conclusion}
	\sup_{(x,y) \in B_{R/2}(x_0,y_0)} |D_y^{\gamma} u(x,y)| \leq p! C^p R^{-p} \text{ for } p = |\gamma| \geq 1
\end{equation}
for some constant $C \in (0,\infty)$ depending on $n$, $q$, $\mu$, $A^i$, and $B$.  The branch set of the graph of $\tilde u$ is $\{ (0,y,u_1(0,y)) : y \in B^{n-2}_1(0) \}$, which is real analytic if $u$ satisfies (\ref{theorem3_conclusion}). 

We can regard Theorem \ref{theorem3} as analogous to the result that single-valued solutions to (\ref{theorem3_equation0}) are real analytic.  One approach to proving such theorems for single-valued functions due to Morrey (see~\cite[Sections 5.8 and 6.7]{Morrey} or~\cite{Morrey2}) is to use integral kernels to show the single-valued solution extends to a holomorphic function on some domain in $\mathbb{C}^n$.  However, we cannot use integral kernels for $q$-valued functions, so instead we take another approach of inductively using Schauder estimates.  To prove Theorem \ref{theorem3}, we first show that $D_y^{\gamma} u \in C^{1,\mu;q}(B_1(0))$ for all $\gamma$ by an inductive argument involving difference quotients and Schauder estimates.  For Theorem \ref{theorem3} we need estimates on $D_y^{\gamma} u$ of the particular form (\ref{theorem3_conclusion}), which requires obtaining precise estimates on terms appearing in the Schauder estimates using a modified version of a technique used by Friedman in~\cite{Friedman} involving majorants.

By replacing a Schauder estimate for equations (Lemma \ref{schauder_div} in Section~\ref{sec:elliptictheory}) with a Schauder estimate for elliptic systems (Lemma \ref{schauder_system} in Section~\ref{sec:elliptictheory}), we obtain a similar result for elliptic systems:

\begin{thm} \label{theorem3_systems}
Let $\mu \in (0,1/q)$.  There is a $\varepsilon = \varepsilon(n,m,\mu,\nu) > 0$ such that the following is true.  Let $\tilde u \in C^{1,\mu}(B_1(0);\mathcal{A}_q(\mathbb{R}^m))$, where $\mu \in (0,1/q)$, such that $\mathcal{B}_{\tilde u}$ is nonempty, $\mathcal{B}_{\tilde u} \subseteq \{0\} \times B^{n-2}_1(0)$, and $\|\tilde u\|_{C^1(B_1(0))} \leq 1/2$.  Suppose $\tilde u$ is a solution to the non-linear elliptic differential equation 
\begin{equation*}
	D_i A^i_{\kappa}(X,\tilde u_l,D\tilde u_l) + B_{\kappa}(X,\tilde u_l,D\tilde u_l) = 0 \text{ weakly in } B_1(0) \setminus \mathcal{B}_{\tilde u}, 
\end{equation*}
where $A^i_{\kappa}, B_{\kappa} : B_1(0) \times B^m_1(0) \times B^{mn}_1(0) \rightarrow \mathbb{R}$ are locally real analytic single-valued functions such that 
\begin{equation*} 
	|(D_{P^{\lambda}_j} A^i_{\kappa})(X,Z,P) - \delta^{ij} \delta_{\kappa \lambda}| < \varepsilon 
	\text{ for } (X,Z,P) \in B_1(0) \times B^m_1(0) \times B^{mn}_1(0). 
\end{equation*} 
Then $\tilde u(x,y)$ is real analytic in $y$ in the sense that for $B_R(x_0,y_0) \subset \subset B_1(0)$, 
\begin{equation*}
	\sup_{(x,y) \in B_{R/2}(x_0,y_0)} |D_y^{\gamma} \tilde u(x,y)| \leq p! C^p R^{-p} \text{ for } p = |\gamma| \geq 1, 
\end{equation*}
for some constant $C \in (0,\infty)$ depending on $n$, $q$, $\mu$, $A^i$, $B$, and $\mathcal{B}_{\tilde u}$.  Consequently, the branch set of the graph of $\tilde u$ is a union of at most $q/2$ real analytic, $(n-2)$-dimensional submanifolds. 
\end{thm} 

In particular, (\ref{theorem3}) establishes that the branch sets of the minimal submanifolds constructed in Theorem \ref{theorem1} are locally real analytic $(n-2)$-dimensional submanifolds provided $\varepsilon$ is sufficiently small.


\section{Elliptic theory for multivalued functions} \label{sec:elliptictheory}

The proof of the main results use standard theorems for elliptic differential equation such as the maximum principle and the Schauder estimates.  This chapter is concerned with extending those theorems to solutions in the spaces $C^{k;q}$ and $W^{1,2;q}$ discussed in Section~\ref{sec:preliminaries}.  We first consider differential equations of the form 
\begin{equation*}
	a_l^{ij} D_{ij} u_l + b_l^i D_i u_l + c_l u_l = (\geq , \leq) \, f_l \text{ in } \Omega \setminus [0,\infty) \times \{0\} \times \mathbb{R}^{n-2} 
\end{equation*}
for $l = 1,2,\ldots,q$, where $\Omega$ is an open set in $\mathbb{R}^n$, $u = (u_1,u_2,\ldots,u_q) \in C^{2;q}(\Omega \setminus \{0\} \times \mathbb{R}^{n-2})$, and $a^{ij} = (a^{ij}_1,a^{ij}_2,\ldots,a^{ij}_q), b^i = (b^i_1,b^i_2,\ldots,b^i_q), c = (c_1,c_2,\ldots,c_q), f = (f_1,f_2,\ldots,f_q) \in C^{0;q}(\Omega)$.  We assume the ellipticity condition 
\begin{equation} \label{ellipticity}
	a_l^{ij}(X) \xi_i \xi_j \geq \lambda |\xi|^2 \text{ for } X \in \Omega, \, \xi \in \mathbb{R}^n, \, l = 1,2,\ldots,q
\end{equation}
for some constant $\lambda > 0$.  Given $(x,y) \in \Omega$, if $x_1 \leq 0$ then each $u_l$ solves $a_l^{ij} D_{ij} u_l + b_l^i D_i u_l + c_l u_l = (\geq , \leq) \, f_l$ in $B_{|x|/2}(x,y) \cap \Omega$.  Similarly if $x_1 > 0$ and $\hat u_l$ are the $C^2$ single-valued functions defined by $\hat u_l = u_l$ on $B_{|x|/2}(x,y) \cap \Omega \cap \mathbb{R} \times (0,\infty) \times \mathbb{R}^{n-2}$ for $l = 1,2,\ldots,q$, $\hat u_1 = u_q$ on $B_{|x|/2}(x,y) \cap \Omega \cap \mathbb{R} \times (-\infty,0) \times \mathbb{R}^{n-2}$, $\hat u_l = u_{l-1}$ on $B_{|x|/2}(x,y) \cap \Omega \cap \mathbb{R} \times (-\infty,0) \times \mathbb{R}^{n-2}$ for $l = 2,3,\ldots,q$, then each $\hat u_l$ satisfies an elliptic differential equation on $B_{|x|/2}(x,y) \cap \Omega$.  Thus $u$ satisfies standard elliptic estimates on $B_{|x|/2}(x,y) \cap \Omega$.  Our first result is a strong maximum principle:

\begin{lemma} \label{maxp_strong}
Let $u \in C^{0;q}(\overline{\mathcal{C}}) \cap C^{1;q}(\mathcal{C}) \cap C^{2;q}(\mathcal{C} \setminus \{0\} \times \mathbb{R}^{n-2})$, $a^{ij},$ $b^i, c \in C^{0;q}(\mathcal{C})$ satisfy 
\begin{equation} \label{maxp_strong_eqn1}
	a_l^{ij} D_{ij} u_l + b_l^i D_i u_l + c_l u_l \geq 0 \text{ in } \mathcal{C} \setminus [0,1) \times \{0\} \times \mathbb{R}^{n-2} 
\end{equation}
for $l = 1,2,\ldots,q$.  Assume (\ref{ellipticity}) holds true for some constant $\lambda > 0$ and $c_l \leq 0$ in $\mathcal{C} \setminus [0,1) \times \{0\} \times \mathbb{R}^{n-2}$ for $l = 1,2,\ldots,q$.  Then $u$ does not attain its maximum value in the interior of $\mathcal{C}$ unless $u_l$ all equal the same constant function. 
\end{lemma}
\begin{proof}
Assume $u_l$ do not all equal the same constant function.  By the strong maximum principle~\cite[Theorem 3.5]{GT} applied locally in $\mathcal{C} \setminus \{0\} \times \mathbb{R}^{n-2}$, $u$ does not attain its maximum value in $\mathcal{C} \setminus \{0\} \times \mathbb{R}^{n-2}$.  Suppose $u$ attained its maximum value at $(0,y_0)$ for some $y_0 \in \mathbb{R}^{n-2}$.  Then $u_1$ extends to a $C^1$ function on $\overline{B_{1/4}(0,1/4,y_0)}$ that attains its maximum value at $(0,y_0)$, $Du_1(0,y_0) = 0$, and satisfies (\ref{maxp_strong_eqn1}), contradicting the Hopf boundary point lemma~\cite[Lemma 3.4]{GT}. 
\end{proof}

Next we prove a Schauder estimate that will be needed for the proof of Theorem \ref{theorem2}. 

\begin{lemma} \label{schauder_strong}
Let $0 < \mu < \tau < 1/q$ and $B_R(X_0) \subseteq \mathbb{R}^n$.  Suppose $u \in C^{1,\tau;q}(\overline{B_R(X_0)})$, $a^{ij}, f \in C^{0,\mu;q}(\overline{B_R(X_0)})$ satisfy 
\begin{equation} \label{schauder_strong_eqn1}
	a_l^{ij} D_{ij} u_l = f_l \text{ in } B_R(X_0) \setminus \{0\} \times \mathbb{R}^{n-2}.  
\end{equation}
Assume (\ref{ellipticity}) holds true for some constant $\lambda > 0$ and $\|a^{ij}\|'_{C^{0,\mu;q}(B_R(X_0))} \leq \Lambda$ for some constant $\Lambda > 0$.  Then 
\begin{equation} \label{schauder_strong_eqn2}
	\|u\|'_{C^{1,\tau;q}(B_{R/2}(X_0))} \leq C \left( R^{-n/2} \|u\|_{L^{2;q}(B_R(X_0))} + R^2 \|f\|'_{C^{0,\mu;q}(B_R(X_0))} \right)
\end{equation}
for some constant $C = C(n,q,\mu,\tau,\lambda,\Lambda) \in (0,\infty)$.
\end{lemma}

The proof of Lemma \ref{schauder_strong} extending Liouville-type result~\cite[Corollary 2.6]{SW2}. 

\begin{lemma} \label{gap_lemma}
Let $\mu \in (0,1/q)$ and suppose $u \in C^{1,\mu;q}(\mathbb{R}^n) \cap C^{\infty;q}(\mathbb{R}^n)$ such that 
\begin{equation} \label{gap_lemma_eqn1}
	\Delta u_l = 0 \text{ in } \mathbb{R}^n \setminus [0,\infty) \times \mathbb{R}^{n-2}
\end{equation}
and $[Du]_{\mu;q,\mathbb{R}^n} < \infty$.  Then $u_l(X) = a + b \cdot X$ for all $X \in \mathbb{R}^n \setminus [0,\infty) \times \{0\} \times \mathbb{R}^{n-2}$ and $l = 1,2,\ldots,q$ for some $a \in \mathbb{R}$ and $b \in \mathbb{R}^n$ independent of $l$. 
\end{lemma}
\begin{proof}
Let $u$ be as in the statement of Lemma \ref{gap_lemma}.  $u = u_a + u_f$ where $u_a = \frac{1}{q} \sum_{j=1}^q u_j$ and $u_f$ are as by (\ref{averageandfree}) and $u_a$ is an affine function by the Liouville theorem.  Thus it suffices to suppose that $u$ is average-free and show that $u = 0$. 

For $u \in C^{1;q}(\mathbb{R}^n)$ that is non-zero, average-free, and satisfies (\ref{gap_lemma_eqn1}) and $y_0 \in \mathbb{R}^{n-2}$, we define the frequency function of $u$ at $(0,y_0)$ by 
\begin{equation*}
	N_{u,(0,y_0)}(\rho) = \frac{\rho^{2-n} \int_{B_{\rho}(0,y_0)} \sum_{l=1}^q |Du_l|^2}{\rho^{1-n} \int_{\partial B_{\rho}(0,y_0)} \sum_{l=1}^q |u_l|^2}
\end{equation*}
for $\rho \in (0,\infty)$.  We extend the two identities in~\cite[Remark 2.3(2)]{SW2} by either the argument in~\cite{SW2} using the fact that $u$ and $Du$ vanish on $\{0\} \times \mathbb{R}^{n-2}$ or by using a cutoff function argument.  We then can extend Lemma 2.2, Remark 2.3(1)(3)(4), and Remark 2.4 of~\cite{SW2} to establish monotonicity and other standard properties of frequency functions for $N_{u,(0,y_0)}$.

Next we extend~\cite[Lemma 2.5]{SW2} by showing that for some $\delta = \delta(n,q) \in (0,1)$, there are no $u \in C^{1;q}(\mathbb{R}^n)$ that are non-zero, are average-free, satisfy (\ref{gap_lemma_eqn1}), and are homogeneous degree $\sigma$ for $\sigma \in [1,1+\delta)$.  Arguing as in~\cite{SW2} using the fact that $Du$ vanishes on $\{0\} \times \mathbb{R}^{n-2}$, if $u \in C^{1;q}(\mathbb{R}^n)$ is average-free, satisfies (\ref{gap_lemma_eqn1}), and is homogeneous degree one then $Du_l$ all equal the same constant function on $S^{n-1} \setminus [0,\infty) \times \{0\} \times \mathbb{R}^{n-2}$.  Since $u$ is average-free, this implies $u_l = 0$ on $\mathbb{R}^n \setminus [0,\infty) \times \{0\} \times \mathbb{R}^{n-2}$ for all $l = 1,2,\ldots,q$.  By arguing as in~\cite{SW2} we also conclude that if $u^{(j)} \in C^{1;q}(\mathbb{R}^n)$ that are non-zero, are average-free, satisfy (\ref{gap_lemma_eqn1}) with $u^{(j)}$ in place of $u$, and are homogeneous degree $\sigma_j$ for $\sigma_j \downarrow 1$, then after passing to a subsequence $u^{(j)}$ converges strongly in $L^{2;q}(B_1(0))$ to $u \in W^{1,2;q}(B_1(0))$ such that $u$ is average free, $u$ satisfies (\ref{gap_lemma_eqn1}), $Du_l$ all equal the same constant function, and $\|u\|_{L^{2;q}(B_1(0))} = 1$.  But $u$ being average-free implies that $Du_l$ must all equal the zero function and thus $u_l$ all equal the zero function, contradicting $\|u\|_{L^{2;q}(B_1(0))} = 1$.

By extending the proof of~\cite[Corollary 2.6]{SW2} there is no nonzero $u \in C^{1,\mu;q}(\mathbb{R}^n)$ such that $u$ is average-free, $u$ satisfies (\ref{gap_lemma_eqn1}), and $[Du]_{\mu;q,\mathbb{R}^n} < \infty$ for some $\mu \in (0,\delta)$.  Having established Lemma \ref{gap_lemma} in the case that $\mu \in (0,\delta)$, we can prove Lemma \ref{schauder_strong} in the special case that $0 < \mu < \tau < \delta$.  Using the dimension reduction argument in the proof of~\cite[Theorem 4.1]{SW2} and the fact that the homogeneous, average-free $u \in C^{1;q}(\mathbb{R}^2)$ satisfying (\ref{gap_lemma_eqn1}) are given by $u_l(re^{i\theta}) = \op{Re}(cr^{1+k/q} e^{i k/q (\theta + 2(l-1)\pi)})$ for $r \geq 0$, $\theta \in [0,2\pi)$, and $l = 1,2,\ldots,q$ for some constant $c \in \mathbb{C}$ and integer $k \geq q+1$, we conclude that there are no non-zero, average-free $u \in C^{1;q}(\mathbb{R}^n)$ that satisfies (\ref{gap_lemma_eqn1}) is homogeneous degree $\sigma \in [1,1+1/q)$ and thus Lemma \ref{gap_lemma} holds for all $\mu \in (0,1/q)$. 
\end{proof}

\begin{proof}[Proof of Lemma \ref{schauder_strong}]
We adapt the proof of~\cite[Lemma 3.2]{SW2}.  We in fact assume $R = 1$ and prove the weaker inequality that for every $\delta > 0$, 
\begin{equation*}
	[Du]_{\tau;q,B_{1/2}(X_0)} \leq \delta [Du]_{\tau;q,B_1(X_0)} + C \left( \sup_{B_1(0)} |u| 
		+ \sup_{B_1(0)} |Du| + \|f\|_{C^{0,\mu;q}(B_1(X_0))} \right) 
\end{equation*}
for some constant $C = C(n,q,\mu,\lambda,\Lambda,\delta) \in (0,\infty)$.  Then by translating and rescaling, $u$ as in the statement of Lemma \ref{schauder_strong} satisfies 
\begin{equation*}
	\rho^{1+\tau} [Du]_{\tau;q,B_{\rho/2}(Y)} \leq \delta \rho^{1+\tau} [Du]_{\tau;q,B_{\rho}(Y)} 
		+ C \left( \sup_{B_{\rho}(Y)} |u| + \rho \sup_{B_{\rho}(Y)} |Du| + \rho^2 \|f\|'_{C^{0,\mu;q}(B_{\rho}(Y))} \right) 
\end{equation*}
for all $B_{\rho}(Y) \subseteq B_R(X_0)$ and (\ref{schauder_strong_eqn2}) follows by standard interpolation inequalities.

Suppose instead that for some $\delta > 0$ and every positive integer $k$, there is a ball $B_1(X_k)$ and $u_k = (u_{k,1},u_{k,2},\ldots,u_{k,q}) \in C^{2;q}(B_R(X_l) \setminus \{0\} \times \mathbb{R}^{n-2}) \cap C^{1,\tau;q}(B_R(X_k))$ and $a_k^{ij} = (a_{k,1}^{ij},a_{k,2}^{ij},\ldots,a_{k,q}^{ij}),$ $f_k = (f_{k,1},f_{k,2},\ldots,f_{k,q}) \in C^{0,\mu;q}(B_R(X_l))$ such that (\ref{ellipticity}) and (\ref{schauder_strong_eqn1}) hold with $u_{k,l}$, $a_{k,l}^{ij}$, and $f_{k,l}$ in place of $u_l$, $a_l^{ij}$, and $f_l$ and $\|a_k^{ij}\|_{C^{0,\mu;q}(B_R(X_0))} \leq \Lambda$ but 
\begin{equation} \label{schauder_strong_eqn4}
	[Du_k]_{\tau;q,B_{1/2}(X_k)} > \delta [Du_k]_{\tau;q,B_1(X_k)} + k \left( \sup_{B_1(0)} |u_k| 
		+ \sup_{B_1(0)} |Du_k| + \|f_k\|'_{C^{0,\mu;q}(B_1(X_0))} \right) 
\end{equation}
Assume $[Du_k]_{\tau;q,B_{1/2}(X_k)} \leq 2q [Du_{k,1}]_{\mu;B_{1/2}(X_k) \cap \mathbb{R} \times (0,\infty) \times \mathbb{R}^{n-2}}$.  Select $Y_k, Y'_k  \in B_{1/2}(X_k) \cap \mathbb{R} \times (0,\infty) \times \mathbb{R}^{n-2}$ such that 
\begin{equation} \label{schauder_strong_eqn5}
	\frac{|Du_{k,1}(Y_k) - Du_{k,1}(Y'_k)|}{|Y_k - Y'_k|^{\tau}} \geq \frac{1}{4q} [Du_k]_{\tau;B_{1/2}(X_k)} 
\end{equation}
and let $\rho_k = |Y_k - Y'_k|$.  By (\ref{schauder_strong_eqn4}) and (\ref{schauder_strong_eqn5}),  
\begin{equation*}
	\frac{1}{4q} [Du_k]_{\tau;B_{1/2}(X_k)} \leq \frac{2}{\rho_k^{\tau}} \sup_{B_1(X_k)} |Du_k| < \frac{2}{k \rho_k^{\tau}} [Du_{k,1}]_{\tau;B_{1/2}(X_l)}, 
\end{equation*}
so $\rho_k^{\tau} \leq 8/qk$ for all $k$ and thus $\rho_k \rightarrow 0$ as $k \rightarrow \infty$. 

Suppose $\text{dist}(\{Y_k,Y'_k\}, \{0\} \times \mathbb{R}^{n-2})/\rho_k \leq c$ for some constant $c \in [1,\infty)$.  Then for some $Z_k \in \{0\} \times \mathbb{R}^{n-2}$, $|Y_k - Z_k| \leq 2c \rho_k$.  By translating assume $Z_k = 0$.  Let $R_k = 1/2\rho_k - 2c > 0$ for $k$ sufficiently large.  Rescale letting $\zeta_k = Y_k/\rho_k$ and $\zeta'_k = Y'_k/\rho_k$ and $\hat u_k = (\hat u_{k,1}, \hat u_{k,2},\ldots,\hat u_{k,q})$, $\hat a_k^{ij} = (\hat a_{k,1}^{ij}, \hat a_{k,2}^{ij},\ldots,\hat a_{k,q}^{ij})$, and $\hat f_k = (\hat f_{k,1}, \hat f_{k,2},\ldots,\hat f_{k,q})$ where 
\begin{align*}
	\hat u_{k,l}(X) &= \rho_l^{-1-\tau} [Du_k]_{\tau,B_1(X_k)}^{-1} (u_{k,l}(\rho_k X) - u_{k,l}(0) - Du_{k,l}(0) \cdot \rho_k X), \\
	\hat a_{k,l}^{ij}(X) &= a_{k,l}^{ij}(\rho_k X), \\
	\hat f_{k,l}(X) &= \rho_l^{1-\tau} [Du_k]_{\tau,B_1(X_k)}^{-1} f_{k,l}(\rho_k X), 
\end{align*}
for $X \in B_{R_k}(0) \setminus [0,\infty) \times \{0\} \times \mathbb{R}^{n-2}$ and $l = 1,2,\ldots,q$ so that 
\begin{gather*}
	\hat a_{k,l}^{ij} D_{ij} \hat u_{k,l} = \hat f_{k,l} \text{ in } B_{R_k}(0) \setminus \{0\} \times \mathbb{R}^{n-2}, \\
	[D\hat u_k]_{\tau;q,B_{R_k}(0)} \leq 1, \quad |D\hat u_{k,1}(\zeta_k) - D\hat u_{k,1}(\zeta'_k)| \geq \frac{\delta}{4q}. 
\end{gather*}
Since $\{\zeta_k\}$ and $\{\zeta'_k\}$ are bounded, after passing to a subsequence, $\zeta_k \rightarrow \zeta$ and $\zeta'_k \rightarrow \zeta'$ for some points $\zeta,\zeta' \in \mathbb{R}^n$.  Since 
\begin{equation*}
	\sup_{B_{R_k}(0)} |\hat a_k^{ij}| + \rho_k^{-\mu} [a_k^{ij}]_{\mu;q,B_{R_l}(0)} \leq C \Lambda, 
\end{equation*}
after passing to a subsequence $\{\hat a_{k,l}^{ij}\}$ converges to some constant $\hat a^{ij}$ (independent of $l$) uniformly on $K \setminus [0,\infty) \times \{0\} \times \mathbb{R}^{n-2}$ as $k \rightarrow \infty$ for every compact subset $K$ of $\mathbb{R}^n$.  By (\ref{schauder_strong_eqn4}), 
\begin{equation*}
	\sup_{B_{R_l}(0)} |\hat f_k| + \rho^{-\mu} [\hat f_k]_{\mu, B_{R_l}(0)} \leq C \rho_k^{1-\tau}/k 
\end{equation*} 
for some constant $C = C(m,n) \in (0,\infty)$, so after passing to a subsequence $\{\hat f_k\}$ converges to zero in $C^{0;q}(\overline{\Omega})$ as $k \rightarrow \infty$ for all $\Omega \subset \subset \mathbb{R}^n$.  Since $[D\hat u_k]_{\tau;q,B_{R_k}(0)} \leq 1$, after passing to a subsequence, $\{u_k\}$ converges in $C^{1;q}(\overline{\Omega})$ to some $\hat u = (\hat u_1,\hat u_2,\ldots,\hat u_q) \in C^{1,\tau;q}(\mathbb{R}^n)$ for all $\Omega \subset \subset \mathbb{R}^n$.  Moreover, by the interior $C^{2,\mu}$ Schauder estimates for single-valued functions~\cite[Corollary 6.3]{GT} applied locally on $\mathbb{R}^n \setminus \{0\} \times \mathbb{R}^{n-2}$, after passing to a subsequence $\hat u_k \rightarrow \hat u$ in $C^{2;q}$ on compact subsets of $\mathbb{R}^n \setminus \{0\} \times \mathbb{R}^{n-2}$.  Hence $\hat u$ satisfies the differential equation $\hat a^{ij} D_{ij} \hat u = 0$ on $\mathbb{R}^n \setminus \{0\} \times \mathbb{R}^{n-2}$.  However, $[D\hat u]_{\tau, \mathbb{R}^n} \leq 1$ and $|D\hat u_1(\zeta) - D\hat u_1(\zeta')) \geq \delta/4q$, which after an affine change of variables contradicts Lemma \ref{gap_lemma}. 

Suppose instead that $\text{dist}(\{Y_k,Y'_k\}, \{0\} \times \mathbb{R}^{n-2})/\rho_k$ is unbounded.  Assume $\text{dist}(\{Y_k,Y'_k\}, \{0\} \times \mathbb{R}^{n-2})/\rho_k$ tends to infinity and $Y'_k \in B_{1/2}(0) \cap \{0\} \times (0,\infty) \times \mathbb{R}^{n-2}$.  For some $R_k \rightarrow \infty$, $R_k < \text{dist}(\{Y_k,Y'_k\}, \{0\} \times \mathbb{R}^{n-2})/\rho_k$ and $R_k < 1/2\rho_k$.  Rescale letting $\zeta_k = (Y_k-Y'_k)/\rho_k$ and letting $\hat u_k$, $\hat a_k^{ij}$, and $\hat f_k$ be the single-valued functions defined by 
\begin{align*}
	\hat u_k(X) &= \rho_k^{-1-\tau} [Du_{k,1}]_{\tau,B_1(X_k)}^{-1} (u_k(Y'_k + \rho_k X) - u_k(Y'_k) - Du_k(Y'_k) \cdot \rho_k X), \\
	\hat a_k^{ij}(X) &= a_{k,1}^{ij}(Y'_k + \rho_k X), \\
	\hat f_k(X) &= \rho_k^{-1-\tau} [Du_{k,1}]_{\tau,B_1(X_k)}^{-1} f_k(Y'_k + \rho_k X), 
\end{align*}
for $X \in B_{R_l}(0)$ for large $k$.  Similar to above, after passing to a subsequence, $\{\zeta_k\}$ converges to some $\zeta$, $\{\hat a_k^{ij}\}$ converges uniformly on compact subsets of $\mathbb{R}^n$ to some constant $\hat a^{ij}$, $\{\hat u_k\}$ converges in $C^2$ on compact subsets of $\mathbb{R}^n$ to some single-valued function $\hat u$, and $\{\hat f_k\}$ converges uniformly to zero.  $\hat u$ satisfies $\hat a^{ij} D_{ij} \hat u = 0$ on $\mathbb{R}^n$, $[D\hat u]_{\tau,\mathbb{R}^{n-2}} \leq 1$, and $|D\hat u(\zeta) - D\hat u(0)| \geq \frac{\delta}{4q}$, which after an affine change of variables contradicts the Liouville theorem for single-valued harmonic functions.
\end{proof}

Next we consider equations of the form 
\begin{equation} \label{divform}
	\int_{\Omega} \sum_{l=1}^q \left( \left( a_l^{ij} D_j u_l + b^i u_l \right) D_i \zeta_l - \left( c_l^j D_j u + d_l u_l \right) \zeta_l \right) 
	= (\leq , \geq) \, \int_{\Omega} \sum_{l=1}^q \left( f_l^i D_i \zeta_l - g_l \zeta_l \right)
\end{equation}
for all $\zeta = (\zeta_1,\zeta_2,\ldots,\zeta_q) \in C_c^{1;q}(\Omega \setminus \{0\} \times \mathbb{R}^{n-2})$, where $\Omega$ is an open set in $\mathbb{R}^n$, $u = (u_1,u_2,\ldots,u_q) \in W^{1,2;q}(\Omega)$, $a^{ij} = (a^{ij}_1,a^{ij}_2,\ldots,a^{ij}_q), b^i = (b^i_1,b^i_2,\ldots,b^i_q), c^j = (c^j_1,c^j_2,\ldots,c^j_q), d = (d_1,d_2,\ldots,d_q) \in L^{\infty;q}(\Omega)$, and $f^i = (f^i_1,f^i_2,\ldots,f^i_q), g = (g_1,g_2,\ldots,g_q) \in L^{2;q}(\Omega)$.  We require ellipticity condition (\ref{ellipticity}) to hold for some constant $\lambda > 0$.  We claim that (\ref{divform}) continues to hold if instead $\zeta \in C_c^{1;q}(\Omega)$.  To see this, for every $\delta > 0$ let $\chi_{\delta} \in C^1(\mathbb{R}^n)$ be a single-valued function such that $0 \leq \chi_{\delta} \leq 1$, $\chi_{\delta} = 1$ on $\mathbb{R}^n \setminus B^2_{\delta}(0) \times \mathbb{R}^{n-2}$, $\chi_{\delta} = 0$ on $B^2_{\delta/2}(0) \times \mathbb{R}^{n-2}$, and $|D\chi_{\delta}| \leq 3/\delta$.  Replace $\zeta_l$ with $\zeta_l \chi_{\delta}$ in (\ref{divform}) to get 
\begin{align*}
	&\int_{\Omega} \sum_{l=1}^q \left( \left( a_l^{ij} D_j u_l + b^i u_l \right) D_i \zeta_l - \left( c_l^j D_j u + d_l u_l \right) \zeta_l \right) \chi_{\delta}
	\\&= (\leq , \geq) \, \int_{\Omega} \sum_{l=1}^q \left( f_l^i D_i \zeta_l - g_l \zeta_l \right) \chi_{\delta} 
	- \int_{\Omega} \sum_{l=1}^q \left( a_l^{ij} D_j u_l + b^i u_l - f_l^i \right) \zeta_l D_i \chi_{\delta}
\end{align*}
and let $\delta \downarrow 0$ to get (\ref{divform}) for $\zeta \in C_c^{1;q}(\Omega)$.  Using (\ref{divform}) and Sobolev inequality Lemma \ref{sololevlemma} below, the maximum principle~\cite[Theorem 8.1]{GT} and global supremum estimates~\cite[Theorem 8.16]{GT} readily extend to $u \in W^{1,2;q}(\Omega)$ satisfying (\ref{divform}) with the $\leq$ sign.  Using (\ref{divform}) and Sobolev inequality Lemma \ref{sololevlemma} and Poincar\'{e} inequality Lemma \ref{poincarelemma} below we will extend local H\"{o}lder continuity estimates~\cite[Theorem 8.22]{GT} to solutions to (\ref{divform}) with the $=$ sign.

\begin{lemma} \label{sololevlemma}
Let $1 \leq p < n$.  Suppose $u \in W_0^{1,p;q}(\mathbb{R}^n)$.  Then 
\begin{equation*} 
	\|u\|_{L^{np/(n-p);q}(\mathbb{R}^n)} \leq C \|Du\|_{L^{p;q}(\mathbb{R}^n)} 
\end{equation*} 
for some $C = C(n,q,p) \in (0,\infty)$.  
\end{lemma} 
\begin{proof}
By the Sobolev inequality for the single-valued functions $u_l |_{\mathbb{R} \times (0,\infty) \times \mathbb{R}^{n-2}}$ and $u_l |_{\mathbb{R} \times (-\infty,0) \times \mathbb{R}^{n-2}}$, $l = 1,2,\ldots,q$, 
\begin{align*}
	\|u\|_{L^{np/(n-p);q}(\mathbb{R}^n)} 
	&= \left( \sum_{l=1}^q (\|u_l\|_{L^{np/(n-p);q}(\mathbb{R} \times (0,\infty) \times \mathbb{R}^{n-2})}^{np/(n-p)} 
		+ \|u_l\|_{L^{np/(n-p);q}(\mathbb{R} \times (-\infty,0) \times \mathbb{R}^{n-2})}^{np/(n-p)}) \right)^{(n-p)/np}
	\\&\leq C \left( \sum_{l=1}^q (\|Du_l\|_{L^{p;q}(\mathbb{R} \times (0,\infty) \times \mathbb{R}^{n-2})}^{np/(n-p)}  
		+ \|Du_l\|_{L^{p;q}(\mathbb{R} \times (-\infty,0) \times \mathbb{R}^{n-2})}^{np/(n-p)}) \right)^{(n-p)/np}
	\\&\leq C \|Du\|_{L^{p;q}(\mathbb{R}^n)} 
\end{align*}
for $C = C(n,q,p) \in (0,\infty)$.  
\end{proof}

\begin{lemma} \label{poincarelemma}
For $u \in W^{1,2;q}(B_R(0))$, 
\begin{equation} \label{poincare_ineq}
	\|u - \ell\|_{L^{2;q}(B_R(0))} \leq CR \|Du\|_{L^{2;q}(B_R(0))}, 
\end{equation} 
for some $C = C(n,q) \in (0,\infty)$ where $\ell = \int_{B_R(0)} \frac{1}{q} \sum_{j=1}^q u_j$.
\end{lemma} 

\begin{rmk}
The reason for stating (\ref{poincare_ineq}) in terms of $W^{1,2;q}$ functions on a ball $B_R(0)$ centered at a point on $\{0\} \times \mathbb{R}^{n-2} = \emptyset$ is that (\ref{poincare_ineq}) fails if we replace $B_R(0)$ with a ball $B$ such that $\overline{B} \cap \{0\} \times \mathbb{R}^{n-2} = \emptyset$.  For example, (\ref{poincare_ineq}) fails if $u_1 \equiv -1$, $u_2 \equiv 1$, and $u_l \equiv 0$ for $l \geq 3$ in $B$.
\end{rmk} 

\begin{proof}[Proof of Lemma \ref{poincarelemma}]
By scaling, we may suppose $R = 1$.  Writing $u = u_a + u_f$ for $u_a = \frac{1}{q} \sum_{j=1}^q u_j$ and $u_f$ as by (\ref{averageandfree}), 
\begin{align*}
	\|u - \ell\|_{L^{2;q}(B_1(0))}^2 &= \|u_a - \ell\|_{L^{2;q}(B_1(0))}^2 + \|u_f\|_{L^{2;q}(B_1(0))}^2 \\
	\|Du\|_{L^{2;q}(B_1(0))}^2 &= \|Du_a\|_{L^{2;q}(B_1(0))}^2 + \|Du_f\|_{L^{2;q}(B_1(0))}^2. 
\end{align*}
By the Poincar\'{e} inequality for single-valued functions, $\|u_a - \ell\|_{L^2(B_1(0))} \leq C \|Du_a\|_{L^2(B_1(0))}$ for some $C = C(n) \in (0,\infty)$, so it suffices to suppose $u$ is average-free. 

Suppose that for every integer $j \geq 1$ there are average-free $u^{(j)} \in W^{1,2;q}(B_1(0))$ such that 
\begin{equation*} 
	\|u^{(j)}\|_{L^{2;q}(B_1(0))} > j \|Du^{(j)}\|_{L^{2;q}(B_1(0))}. 
\end{equation*} 
By scaling we may suppose $\|u^{(j)}\|_{L^{2;q}(B_1(0))} = 1$ so that 
\begin{equation*}
	\|u^{(j)}\|_{L^{2;q}(B_1(0))} = 1, \quad \|Du^{(j)}\|_{L^{2;q}(B_1(0))} < 1/j.  
\end{equation*} 
By Rellich's lemma, after passing to a subsequence $u^{(j)}$ converges in $L^{2;q}(B_1(0))$ to some average-free $u = (u_1,u_2,\ldots,u_q) \in W^{1,2;q}(B_1(0))$ such that $\|u\|_{L^{2;q}(B_1(0))} = 1$ and $\|Du\|_{L^{2;q}(B_1(0))} = 0$.  Since $Du_l = 0$ a.e. in $B_1(0)$ for $l = 1,2,\ldots,q$ and $u \in W^{1,2;q}(B_1(0))$, $u_l$ all equal the same constant functions on $B_1(0)$.  Since $u$ is average free, $u_l = 0$ a.e. on $B_1(0)$ for $l = 1,2,\ldots,q$, contradicting $\|u\|_{L^{2;q}(B_1(0))} = 1$. 
\end{proof}

\begin{lemma} \label{DgNM_Holder_thm}
Let $R_0 > 0$.  Let $u \in W^{1,2;q}(B_{R_0}(0))$, $a^{ij}, b^i, c^j, d \in L^{\infty;q}(B_{R_0}(0))$, and $f^j \in L^{s;q}(B_{R_0}(0))$ and $g \in L^{s/2;q}(B_{R_0}(0))$ for $s > n$ such that 
\begin{equation*}
	\int_{B_{R_0}(0)} \sum_{l=1}^q \left( \left( a_l^{ij} D_j u_l + b^i u_l \right) D_i \zeta_l - \left( c_l^j D_j u + d_l u_l \right) \zeta_l \right) 
	= \int_{B_{R_0}(0)} \sum_{l=1}^q \left( f_l^i D_i \zeta_l - g_l \zeta_l \right)
\end{equation*}
for all $\zeta \in C_c^{1;q}(B_{R_0}(0) \setminus \{0\} \times \mathbb{R}^{n-2})$.  Suppose (\ref{ellipticity}) holds for some constant $\lambda > 0$ and 
\begin{equation*}
	\sup_{B_{R_0}(0)} |a^{ij}| \leq \Lambda, \quad R_0 \sup_{B_{R_0}(0)} |b^i| + R_0 \sup_{B_{R_0}(0)} |c^j| + R_0^2 \sup_{B_{R_0}(0)} |d| \leq \nu,
\end{equation*}
for some constants $\Lambda, \nu > 0$.  Then for some constants $\mu \in (0,1/q)$ and $C \in (0,\infty)$ depending on $n$, $q$, $s$, $\lambda$, $\Lambda$, and $\nu$, $u$ is equal to an element of $C^{0,\mu;q}(B_{R_0/2}(0))$ a.e. in $B_{R_0/2}(0)$ and, taking $u$ to be in $C^{0,\mu;q}(B_{R_0/2}(0))$, 
\begin{equation*} 
	R_0^{\mu} [u]_{\mu;q,B_{R_0/2}(0)} \leq C \left( \sup_{B_{R_0}(0)} |u| + R^{1-n/s} \|f\|_{L^{s;q}(B_{R_0}(0))} 
	+ R^{2-2n/s} \|g\|_{L^{s/2;q}(B_{R_0}(0))} \right) . 
\end{equation*}
\end{lemma}
\begin{proof}
First we show that if $B_{4R}(0,y_0) \subset B_{R_0}(0)$ and $\hat u \in W^{1,2;q}(B_{4R}(0,y_0))$, $\hat f^j \in L^{s;q}(B_{4R}(0,y_0))$, and $\hat g \in L^{s/2;q}(B_{4R}(0,y_0))$ for $s > n$ such that $\hat u_l \geq 0$ for $l = 1,2,\ldots,q$ and 
\begin{equation} \label{DgNM_eqn1}
	\int_{B_{4R}(0,y_0)} \sum_{l=1}^q \left( \left( a_l^{ij} D_j \hat u_l + b^i \hat u_l \right) D_i \zeta_l 
		- \left( c_l^j D_j \hat u + d_l \hat u_l \right) \zeta_l \right) 
		\geq \int_{B_{4R}(0,y_0)} \sum_{l=1}^q \left( \hat f_l^i D_i \zeta_l - \hat g_l \zeta_l \right)
\end{equation}
for all $\zeta \in C_c^{1;q}(B_{4R}(0,y_0) \setminus \{0\} \times \mathbb{R}^{n-2})$ such that $\zeta_l \geq 0$ for $l = 1,2,\ldots,q$ then 
\begin{equation} \label{DgNM_eqn2}
	R^{-n} \int_{B_{2R}(0,y_0)} \sum_{l=1}^q \hat u_l 
	\leq C \left( \inf_{B_R(0,y_0)} \hat u + R^{1-n/s} \|\hat f\|_{L^{s;q}(B_{4R}(0,y_0))} + R^{2-2n/s} \|\hat g\|_{L^{s/2;q}(B_{4R}(0,y_0))} \right)
\end{equation}
for some constant $C = C(n,q,s,\Lambda/\lambda, \nu/\lambda) \in (0,\infty)$.  Translate and rescale so that $y_0 = 0$ and $R = 1$.  (\ref{DgNM_eqn2}) follows from standard arguments using (\ref{DgNM_eqn1}) and Sobolev inequality Lemma \ref{sololevlemma} such as the proof of Theorem 8.18 of~\cite{GT} except to prove 
\begin{equation*}
	\left( \int_{B_3(0)} \sum_{l=1}^q \bar u_l^p \right) \left( \int_{B_3(0)} \sum_{l=1}^q \bar u_l^{-p} \right) \leq C 
	\quad \text{where } \bar u_l = \hat u_l + \lambda^{-1} (\|\hat f\|_{L^{s;q}(B_4(0))} + \|\hat g\|_{L^{s/2;q}(C^0(B_4(0))})
\end{equation*}
for $p \in (0,1)$ and some constant $C = C(n,q,p,s,\Lambda/\lambda, \nu/\lambda) \in (0,\infty)$ we use a standard argument that uses (\ref{DgNM_eqn1}), Sobolev inequality Lemma \ref{sololevlemma}, and Poincar\'{e} inequality Lemma \ref{poincarelemma} to bound the integrals of $|w-\ell|^k/k^k$, where $w_l = \log(\bar u_l)$, for large integers $k$ for some $\ell \in \mathbb{R}$ and that avoids using the John-Nirenberg inequality (see the proof of Theorem 4.15 of~\cite{HanLin}). 

Rescale so that $R_0 = 1$.  Arguing as in the proof of Theorem 8.22 of~\cite{GT}, replacing the weak Harnack inequality~\cite[Theorem 8.18]{GT} with (\ref{DgNM_eqn2}), we obtain 
\begin{equation} \label{DgNMH_tv} 
	\op{osc}_{B_R(0,y_0)} u \leq C R^{\mu} \left( \sup_{B_{1/2}(0,y_0)} u + K \right) \leq C R^{\mu} \left( \sup_{B_1(0)} u + K \right) , 
\end{equation}
for all $y_0 \in B^{n-2}_{1/2}(0)$, $R \in (0,1/2]$ and for some constants $\mu \in (0,1/q)$ and $C \in (0,\infty)$ depending on $n$, $q$, $s$, $\Lambda/\lambda$, and $\nu/\lambda$, where 
\begin{gather*}
	\op{osc}_{B_R(0,y_0)} u = \sup_{B_R(0,y_0)} u - \inf_{B_R(0,y_0)} u, \\
	K = \lambda^{-1} (\|f\|_{L^{s;q}(B_1(0))} + \|g\|_{L^{s/2;q}(B_1(0))}). 
\end{gather*}

We want to bound $[u_l]_{\mu;B_{1/2}(0) \cap \mathbb{R} \times (0,\infty) \times \mathbb{R}^{n-2}}$ for $l \in \{1,2,\ldots,q\}$ by showing that if $X_1 = (x_1,y_1)$ and $X_2 = (x_2,y_2)$ are distinct points in $B_{1/2}(0) \cap \mathbb{R} \times (0,\infty) \times \mathbb{R}^{n-2}$ then 
\begin{equation} \label{DgNMH_eqn2} 
	|u_l(X_1) - u_l(X_2)| \leq C |X_1-X_2|^{\mu} \left( \sup_{B_1(0)} |u| + K \right) 
\end{equation}
for some constant $C = C(n,q,s,\Lambda/\lambda,\nu/\lambda) \in (0,\infty)$.  Assume $|x_1| \leq |x_2|$.  We consider four cases: (a) $|X_1-X_2| < |x_2|/2$, (b) $|x_2|/2 \leq |X_1-X_2| \leq |x_2|$, (c) $x_1 = x_2 = 0$, and (d) $|X_1-X_2| > |x_2| > 0$.  In case (a) (\ref{DgNMH_eqn2}) follows by using the H\"{o}lder continuity estimates for single-valued functions~\cite[Theorem 8.22]{GT} to bound $[u]_{\mu;B_{|X_1-X_2|}(X_2)}$, replacing $\mu$ with a smaller value if necessary.  In case (b) (\ref{DgNMH_eqn2}) follows by using (\ref{DgNMH_tv}) to bound $\op{osc}_{B_{2|x_2|}(0,y_2)} u$ if $|x_2| \leq 1/4$ and (\ref{DgNMH_eqn2}) is obvious if $|x_2| > 1/4$.  In case (c) (\ref{DgNMH_eqn2}) follows by using (\ref{DgNMH_tv}) to bound $\op{osc}_{B_{|X_1-X_2|/2}(0,\frac{y_1+y_2}{2})} u$.  In case (d) (\ref{DgNMH_eqn2})  follows from cases (b) and (c) and the triangle inequality: 
\begin{align*}
	|u_l(X_1) - u_l(X_2)| &\leq |u_l(X_1) - u_l(0,y_1)| + |u_l(0,y_1) - u_l(0,y_2)| + |u_l(0,y_2) - u_l(X_2)| \\
	&\leq 3C |X_1-X_2|^{\mu} \left( \sup_{B_1(0)} |u| + K \right) . 
\end{align*}
Similarly we can bound $[u_l]_{\mu,B_{1/2}(0) \cap \mathbb{R} \times (-\infty,0) \times \mathbb{R}^{n-2}}$ by proving (\ref{DgNMH_eqn2}) when $X_1$ and $X_2$ are instead distinct points in $B_{1/2}(0) \cap \mathbb{R} \times (-\infty,0) \times \mathbb{R}^{n-2}$
\end{proof}

\begin{lemma} \label{schauder_div}
Let $\mu \in (0,1/q)$ and $B_R(X_0) \subseteq \mathbb{R}^n$.  Suppose $u \in C^{1,\mu;q}(B_R(X_0))$, $a^{ij}, b^i, f^i \in C^{0,\mu;q}(B_R(X_0))$, and $c^j, d, g \in C^{0;q}(B_R(X_0))$ satisfy 
\begin{equation} \label{sdiv_eqn1}
	\int_{B_R(X_0)} \sum_{l=1}^q \left( \left( a_l^{ij} D_j u_l + b_l^i u_l \right) D_i \zeta_l - \left( c_l^j D_j u + d_l u_l \right) \zeta_l \right) 
	= \int_{B_R(X_0)} \sum_{l=1}^q \left( f_l^i D_i \zeta_l - g_l \zeta_l \right)
\end{equation}
for all $\zeta \in C_c^{1;q}(B_R(X_0) \setminus \{0\} \times \mathbb{R}^{n-2})$.  Suppose (\ref{ellipticity}) holds for some constant $\lambda > 0$ and 
\begin{equation} \label{sdiv_eqn2}
	\|a^{ij}\|'_{C^{0,\mu;q}(B_R(X_0))} \leq \Lambda, 
	\quad R \|b^i\|'_{C^{0,\mu;q}(B_R(X_0))} + R \sup_{B_R(X_0)} |c^j| + R^2 \sup_{B_R(X_0)} |d| \leq \nu,
\end{equation} 
for some constants $\Lambda, \nu > 0$.  Then 
\begin{equation} \label{sdiv_eqn3}
	\|u\|'_{C^{1,\mu;q}(B_{R/2}(X_0))} \leq C \left( \sup_{B_R(X_0)} |u| + R^{1+\mu} [f]_{\mu;q,B_R(X_0)} + R^2 \sup_{B_R(X_0)} |g| \right) 
\end{equation}
for some constant $C = C(n,q,\mu,\lambda,\Lambda,\nu) \in (0,\infty)$.
\end{lemma}
\begin{proof}
First observe that Lemma \ref{schauder_div} holds true in the special case where $a^{ij}_l$ all equal the same constant function, $b^i = 0$, $c^i = 0$, and $d = 0$ by a scaling argument similar to the proof of Lemma \ref{schauder_strong}.  (Note that unlike in the proof of Lemma \ref{schauder_strong}, we do not need to show that after passing to a subsequence $\hat u_k \rightarrow \bar u$ in $C^{2;q}(\Omega)$ for $\Omega \subset \subset \mathbb{R}^n \setminus \{0\} \times \mathbb{R}^{n-2}$.) 

Next we prove Lemma \ref{schauder_div} in general.  Consider any $B_r(X) \subseteq B_R(X_0)$.  Suppose $\{0\} \times \mathbb{R}^{n-2} \cap B_r(X) \neq \emptyset$ and let $Z \in \{0\} \times \mathbb{R}^{n-2} \cap B_r(X)$.  By (\ref{sdiv_eqn1}), 
\begin{equation*}
	a^{ij}(Z) D_{ij} u_l = D_i ((a^{ij}(Z) - a^{ij}_l) D_j u_l - b_l^i u_l + f_l^i) - c_l^j D_j u_l - d_l u_l + g_l \quad \text{in } B_r(X) \setminus \{0\} \times \mathbb{R}^{n-2}. 
\end{equation*}
By Lemma \ref{schauder_div} for the operator $a^{ij}(Z) D_{ij}$ and (\ref{sdiv_eqn2}), 
\begin{align} \label{sdiv_eqn4} 
	r^{1+\mu} [Du]_{\mu;q,B_{r/2}(X)} \leq C \left( \frac{r^{1+2\mu}}{R^{\mu}} [Du]_{\mu, B_r(X)}  + \|u\|'_{C^{1;q}(B_R(X_0))} 
		+ R^{1+\mu} [f]_{\mu,B_R(X_0)} + R^2 \sup_{B_R(X_0)} |g| \right) 
\end{align} 
for some constant $C = C(n,q,\mu,\lambda,\Lambda,\nu) \in (0,\infty)$.  If instead $\{0\} \times \mathbb{R}^{n-2} \cap B_r(X) = \emptyset$, then (\ref{sdiv_eqn4}) holds by the Schauder estimates for single-valued functions.  By (\ref{sdiv_eqn4}) with $r \leq \varepsilon R$ for $\varepsilon = \varepsilon(n,\mu,\lambda,\Lambda,\nu) > 0$ sufficiently small and by interpolation, we obtain (\ref{sdiv_eqn3}). 
\end{proof}

By slightly modifying the proof of Lemma \ref{schauder_div} we obtain the following Schauder estimate for linear systems, whose proof we omit.

\begin{lemma} \label{schauder_system}
Let $\mu \in (0,1/q)$ and $B_R(X_0) \subseteq \mathbb{R}^n$.  Suppose $u = (u_1,u_2,\ldots,u_q) \in C^{1,\mu;q}(B_R(X_0);\mathbb{R}^m)$, $a_{\kappa \lambda}^{ij} = (a^{ij}_{\kappa \lambda,1},a^{ij}_{\kappa \lambda,2},\ldots,a^{ij}_{\kappa \lambda,q}), b_{\kappa \lambda}^i = (b_{\kappa \lambda,1}^i,b_{\kappa \lambda,2}^i,\ldots,b_{\kappa \lambda,q}^i), f_{\kappa}^i = (f_{\kappa,1}^i,f_{\kappa,2}^i,\ldots,f_{\kappa,q}^i) \in C^{0,\mu;q}(B_R(X_0))$, and $c_{\kappa \lambda} = (c^j_{\kappa \lambda,1},c^j_{\kappa \lambda,2},\ldots,c^j_{\kappa \lambda,q}), d_{\kappa \lambda} = (d_{\kappa \lambda,1},d_{\kappa \lambda,2},\ldots,d_{\kappa \lambda,q}), g_{\kappa} = (g_{\kappa,1},g_{\kappa,2},\ldots,g_{\kappa,q}) \in C^{0;q}(B_R(X_0))$ satisfy 
\begin{equation*}
	\int_{B_R(X_0)} \sum_{l=1}^q \left( \left( a_{\kappa \lambda,l}^{ij} D_j u_l^{\lambda} + b_{\kappa \lambda,l}^i u_l^{\lambda} \right) D_i \zeta_l^{\kappa} 
	- \left( c_{\kappa \lambda,l}^j D_j u^{\lambda} + d_{\kappa \lambda,l} u_l^{\lambda} \right) \zeta_l^{\kappa} \right) 
	= \int_{B_R(X_0)} \sum_{l=1}^q \left( f_{\kappa,l}^i D_i \zeta_l^{\kappa} - g_{\kappa,l} \zeta_l^{\kappa} \right)
\end{equation*}
for all $\zeta = (\zeta_1,\zeta_2,\ldots,\zeta_q) \in C_c^{1;q}(B_R(X_0) \setminus \{0\} \times \mathbb{R}^{n-2};\mathbb{R}^m)$.  Suppose 
\begin{equation*}
	R^{\mu} [a_{\kappa \lambda}^{ij}]_{\mu;q,B_R(X_0)} + R \|b_{\kappa \lambda}^i\|'_{C^{0,\mu;q}(B_R(X_0))} 
	+ R \sup_{B_R(X_0)} |c_{\kappa \lambda}^j| + R^2 \sup_{B_R(X_0)} |d_{\kappa \lambda}| \leq \nu,
\end{equation*} 
for some constants $\Lambda, \nu > 0$.  For some $\varepsilon = \varepsilon(n,m,\mu,\nu) > 0$, if 
\begin{equation*}
	\|a_{\kappa \lambda}^{ij} - \delta^{ij} \delta_{\kappa \lambda}\|_{C^0(B_R(X_0))} \leq \varepsilon, 
\end{equation*} 
where $\delta^{ij}$ and $\delta_{\kappa \lambda}$ denote Kronecker deltas, then 
\begin{equation*}
	\|u\|'_{C^{1,\mu;q}(B_{R/2}(X_0))} \leq C \left( \sup_{B_R(X_0)} |u| + R^{1+\mu} [f]_{\mu;q,B_R(X_0)} + R^2 \sup_{B_R(X_0)} |g| \right) 
\end{equation*}
for some constant $C = C(n,m,q,\mu,\nu) \in (0,\infty)$.
\end{lemma}

To prove Theorem \ref{poisson_thm} and Theorem \ref{theorem1} in the case of non-periodic data we need a global estimate supremum estimate on $u$ that is independent of $\rho_1,\ldots,\rho_{n-2}$. 

\begin{lemma} \label{inhomog_poisson}
Let $\mu \in (0,1/q)$.  Suppose $u \in C^{1,\mu;q}(\overline{\mathcal{C}})$, $f^i \in C^{0,\mu;q}(\overline{\mathcal{C}})$, and $g \in C^{0;q}(\overline{\mathcal{C}})$ satisfy 
\begin{equation} \label{inhomog_poisson_eqn}
	\int_{\mathcal{C}} \sum_{l=1}^q D_i u_l D_i \zeta_l = \int_{\mathcal{C}} \sum_{l=1}^q \left( f_l^i D_i \zeta_l - g_l \zeta_l \right)
\end{equation}
for all $\zeta \in C_c^{1;q}(\mathcal{C} \setminus \{0\} \times \mathbb{R}^{n-2})$.  Then 
\begin{equation*}
	\sup_{\mathcal{C}} |u| \leq C \left( \sup_{\partial \mathcal{C}} |u| + [f]_{\mu;q,\overline{\mathcal{C}}} + \sup_{\mathcal{C}} |g| \right) 
\end{equation*}
for some constant $C = C(n,q,\mu) \in (0,\infty)$ independent of $\rho_1,\ldots,\rho_{n-2}$. 
\end{lemma}
\begin{proof} 
Suppose instead for every integer $k \geq 1$ there are $u_k = (u_{k,1},u_{k,2},\ldots,u_{k,q}) \in C^{1,\mu;q}(\overline{\mathcal{C}})$, $f_k^i = (f_{k,1}^i,f_{k,2}^i,\ldots,f_{k,q}^i) \in C^{0,\mu;q}(\overline{\mathcal{C}})$, and $g_k = (g_{k,1},g_{k,2},\ldots,g_{k,q}) \in C^{0;q}(\overline{\mathcal{C}})$ such that (\ref{inhomog_poisson_eqn}) holds with $u_{k,l}$, $f_{k,l}^i$, and $g_{k,l}$ in place of $u_l$, $f_l^i$, and $g_l$ but 
\begin{equation} \label{inhomog_poisson_eqn1}
	\sup_{\mathcal{C}} |u_k| > k \left( \sup_{\partial \mathcal{C}} |u_k| + [f_k]_{\mu;q,\overline{\mathcal{C}}} + \sup_{\mathcal{C}} |g_k| \right). 
\end{equation}
Assume $\sup_{\mathcal{C}} |u_k| = 1$, $|u_{k,1}(\xi_k,0)| = 1$ for some $\xi_k \in B^2_1(0) \setminus [0,1) \times \{0\}$, and $f_{k,l}(0,0) = 0$ for all $k = 1,2,3,\ldots$ and $l = 1,2,\ldots,q$.  Now (\ref{inhomog_poisson_eqn1}) becomes 
\begin{equation} \label{inhomog_poisson_eqn2}
	\sup_{\partial \mathcal{C}} |u_k| + [f_k]_{\mu;q,\overline{\mathcal{C}}} + \sup_{\mathcal{C}} |g_k| < 1/k. 
\end{equation}
After passing to a subsequence $\{\xi_k\}$ converges to some $\hat \xi \in B^2_1(0)$.  By (\ref{inhomog_poisson_eqn2}), $f_{k,l} \rightarrow 0$ uniformly on $K \setminus [0,\infty) \times \{0\} \times \mathbb{R}^{n-2}$ for every compact set $K \subset \overline{\mathcal{C}}$ and $g_{k,l} \rightarrow 0$ uniformly on $\overline{\mathcal{C}} \setminus [0,1] \times \{0\} \times \mathbb{R}^{n-2}$.  By the interior Schauder estimates Lemma \ref{schauder_div} and (\ref{inhomog_poisson_eqn2}), after passing to a subsequence $\{u_k\}$ converges in $C^{1;q}$ on compact subsets of $\mathcal{C}$ to some $\hat u = (\hat u_1,\hat u_2,\ldots,\hat u_q)$.  Given $(x_0,y_0) \in \partial \mathcal{C}$, $u$ decomposes into $q$ single-valued solutions to Poisson equations on $\mathcal{C} \cap B_{1/2}(x_0,y_0)$, so by local estimates~\cite[Theorem 8.27]{GT} and (\ref{inhomog_poisson_eqn2}) $\hat u$ extends to a continuous function on $\overline{\mathcal{C}}$ with $\hat u_l = 0$ on $\partial \mathcal{C}$ for $l = 1,2,\ldots,q$.  Now $\hat u$ satisfies 
\begin{align*}
	\Delta \hat u_l &= 0 \text{ in } \mathcal{C} \setminus [0,1) \times \{0\} \times \mathbb{R}^{n-2} \text{ for } l = 1,2,\ldots,q, \\
	\hat u_l &= 0 \text{ on } \partial \mathcal{C} \text{ for } l = 1,2,\ldots,q. 
\end{align*}
Note that $\hat u \in C^{\infty;q}(\mathcal{C} \setminus \{0\} \times \mathbb{R}^{n-2})$ by elliptic regularity.  Since $\sup_{\mathcal{C}} |\hat u| \leq 1$ and $|\hat u(\hat \xi,0)| = 1$, $|\hat u|$ has attains its maximum value of $1$ at $(\hat \xi,0)$.  However, $\hat u = 0$ on $\partial \mathcal{C}$, so $\hat u$ in fact attains an interior maximum at $(\hat \xi,0)$, contradicting strong maximum principle Lemma \ref{maxp_strong}. 
\end{proof}

Combining the global supremum estimates similar to~\cite[Theorem 8.16]{GT}, Lemma \ref{inhomog_poisson}, Lemma \ref{schauder_div}, and  and the local boundary Schauder estimates for single-valued solutions~\cite[Section 8.11]{GT}, we obtain global Schauder estimates:

\begin{lemma} \label{schauder_wg}
Let $\mu \in (0,1/q)$ and $B_R(X_0) \subseteq \mathbb{R}^n$.  Suppose $u, \varphi \in C^{1,\mu;q}(\overline{\mathcal{C}})$, $a^{ij}, b^i, f^i \in C^{0,\mu;q}(\overline{\mathcal{C}})$, and $c, d, g \in C^{0;q}(\overline{\mathcal{C}})$ satisfy 
\begin{equation*}
	\int_{\mathcal{C}} \sum_{l=1}^q \left( \left( a_l^{ij} D_j u_l + b^i u_l \right) D_i \zeta_l - \left( c_l^j D_j u + d_l u_l \right) \zeta_l \right) 
	= \int_{\mathcal{C}} \sum_{l=1}^q \left( f_l^i D_i \zeta_l - g_l \zeta_l \right)
\end{equation*}
for all $\zeta \in C_c^{1;q}(\mathcal{C} \setminus \{0\} \times \mathbb{R}^{n-2})$ and $u_l = \varphi_l$ on $\partial \mathcal{C}$ for $l = 1,2,\ldots,q$.  Suppose (\ref{ellipticity}) holds for some constant $\lambda > 0$, 
\begin{equation*}
	\|a^{ij}\|_{C^{0,\mu;q}(\mathcal{C})} \leq \Lambda, 
	\quad \|b^i\|_{C^{0,\mu;q}(\mathcal{C})} + \sup_{\mathcal{C}} |c^j| + \sup_{\mathcal{C}} |d| \leq \nu,
\end{equation*} 
for some constants $\Lambda, \nu > 0$, and 
\begin{equation*}
	\int_{\mathcal{C}} \sum_{l=1}^q \left( -b_l^j D_j \zeta_l + d_l \zeta_l \right) \leq 0 
\end{equation*}
for all $\zeta \in C_c^{1;q}(\mathcal{C} \setminus \{0\} \times \mathbb{R}^{n-2})$ such that $\zeta_l \geq 0$ for $l = 1,2,\ldots,q$.  Then 
\begin{equation*}
	\|u\|_{C^{1,\mu;q}(\overline{\mathcal{C}})} \leq C \left( [f]_{\mu;q,\mathcal{C}} + \sup_{\mathcal{C}} |g| 
		+ \|\varphi\|_{C^{1,\mu;q}(\overline{\mathcal{C}})} \right) .
\end{equation*}
for some constant $C = C(n,q,\mu,\lambda,\Lambda,\nu,\rho_1,\ldots,\rho_{n-2} \in (0,\infty)$.  Moreover, in the special case that (\ref{inhomog_poisson_eqn}) holds, the constant $C$ is independent of $\rho_1,\ldots,\rho_{n-2}$. 
\end{lemma}


\section{Existence of solutions to a Poisson equation} \label{sec:poisson}

We now want to prove Theorem \ref{poisson_thm}, which recall involves finding a solution $u = (u_1,u_2,\ldots,u_q) \in C^{0;q}(\overline{\mathcal{C}}) \cap C^{1,\mu;q}(\mathcal{C})$ to (\ref{poisson_thm_eqn}) given $f^i = (f^i_1,f^i_2,\ldots,f^i_q) \in C^{0,\mu;q}(\overline{\mathcal{C}})$ and $g = (g_1,g_2,\ldots,g_q), \varphi = (\varphi_1,\varphi_2,\ldots,\varphi_q) \in C^{0;q}(\overline{\mathcal{C}})$ as in the statement of Theorem \ref{poisson_thm}.  Note that by the weak maximum principle analogous to~\cite[Theorem 8.1]{GT}, in the case that $u$, $f^i$, $g$, and $\varphi$ are periodic with respect to $y_i$ for $i = 1,2,\ldots,n-2$ there is at most one solution $u$ to (\ref{poisson_thm_eqn}).  To solve (\ref{poisson_thm_eqn}) we will first assume that $f^j$, $g$, and $\varphi$ are periodic with respect to $y_i$ with period $\rho_i$ for $i = 1,2,\ldots,n-2$.  Let $u_a$, $f_a$, $g_a$, and $\varphi_a$ denote the average parts of $u$, $f$, $g$, and $\varphi$ respectively and $u_f$, $f_f$, $g_f$, and $\varphi_f$ denote the average-free parts of $u$, $f$, $g$, and $\varphi$ respectively as defined by (\ref{averageandfree}).  By linearity, it suffices to use the existence theory for single-valued functions~\cite[Theorem 8.34]{GT} to solve for $u_a$ such that $\Delta u_a = D_j f_a^j + g_a$ weakly in $\mathcal{C}$ and $u_a = \varphi_a$ on $\partial \mathcal{C}$ and then solve for $u_f$ such that $\Delta u_f = D_j f_f^j + g_f$ weakly in $\mathcal{C}$ and $u_f = \varphi_f$ on $\partial \mathcal{C}$.  Thus we may suppose $f^j$, $g$, and $\varphi$ are average-free and find an average-free solution $u$ to (\ref{poisson_thm_eqn}).

For simplicity, we will first assume $f_l^j = 0$ in $\overline{\mathcal{C}}$ for $l = 1,2,\ldots,q$.  To solve (\ref{poisson_thm_eqn}), we will use the change of variables $\xi_1 + i\xi_2 = (x_1+ix_2)^{1/q}$.  Under this change of variables, $u$, $g$, and $\varphi$ transform into the continuous single-valued function $u_0$, $g_0$, and $\varphi_0$ on $\overline{\mathcal{C}}$ given by 
\begin{equation*}
	u_0(re^{i\theta},y) = u_l(r^q e^{i q\theta},y), \quad 
	g_0(re^{i\theta},y) = g_l(r^q e^{i q\theta},y), \quad 
	\varphi_0(re^{i\theta},y) = \varphi_l(r^q e^{i q\theta},y), 
\end{equation*}
for $r \in [0,1]$, $2(l-1)\pi/q < \theta < 2l\pi/q$, and $y \in \mathbb{R}^{n-2}$, and $l = 1,2,\ldots,q$ and (\ref{poisson_thm_eqn}) transforms into 
\begin{align*}
	\Delta_{\xi} u_0 + q^2 |\xi|^{2q-2} \Delta_y u_0 &= q^2 |\xi|^{2q-2} g_0 \text{ weakly in } \mathcal{C} \setminus \{0\} \times \mathbb{R}^{n-2}, \nonumber \\
	u_0 &= \varphi_0 \text{ on } \partial \mathcal{C}. 
\end{align*}
We will assume that $g_0$ and $\varphi_0$ are smooth on $\overline{\mathcal{C}}$ and $\varphi_0 = 0$ in $B^2_{1/2}(0) \times \mathbb{R}^{n-2}$.  Since $g$ and $\varphi$ are periodic with respect to $y_j$ with period $\rho_j$ for $j = 1,2,\ldots,n-2$, $g_0$ and $\varphi_0$ are periodic with respect to $y_j$ with period $\rho_j$ for $j = 1,2,\ldots,n-2$.  Thus for each $z = (z_1,z_2,\ldots,z_{n-2}) \in \mathbb{Z}^{n-2}$ we can define Fourier coefficients $g_{0,z}$ and $\varphi_{0,z}$ of $u_0$, $g_0$, and $\varphi_0$ with respect to $y$ by 
\begin{equation*}
	g_{0,z}(\xi) = \int_{\mathbb{R}^{n-2}} e^{-i 2\pi \sum_{j=1}^{n-2} z_j y_j/\rho_j} g_0(\xi,y) dy, \quad 
	\varphi_{0,z}(\xi) = \int_{\mathbb{R}^{n-2}} e^{-i 2\pi \sum_{j=1}^{n-2} z_j y_j/\rho_j} \varphi_0(\xi,y) dy 
\end{equation*}
for $\xi \in B^2_1(0)$ and solve for the Fourier coefficient $u_{0,z}$ of $u$ satisfying 
\begin{align} \label{poisson_eqn3}
	\Delta_{\xi} u_{0,z} - q^2 (z_1^2/\rho_1^2 + z_2^2/\rho_2^2 + \cdots + z_{n-2}^2/\rho_{n-2}^2) |\xi|^{2q-2} u_{0,z} 
		&= q^2 |\xi|^{2q-2} g_{0,z} \text{ in } B^2_1(0), \nonumber \\
	u_0 &= \varphi_0 \text{ on } \partial \mathcal{C}. 
\end{align}
for each $z \in \mathbb{Z}^{n-2}$.  By standard elliptic theory~\cite[Theorems 8.14]{GT} there exists a unique solution $u_{0,z} \in C^{\infty}(\overline{B^2_1(0)})$ to (\ref{poisson_eqn3}) for every $z \in \mathbb{Z}^{n-2}$.  

Fix $z \in \mathbb{Z}^{n-2}$ and define $u_z = (u_{1,z},u_{2,z},\ldots,u_{q,z}) : \overline{B^2_1(0)} \setminus [0,1] \times \{0\} \rightarrow \mathbb{R}^q$ by 
\begin{equation*}
	u_{l,z}(r e^{i\theta}) = u_{0,z}(r^{1/q} e^{i(\theta + 2(l-1)\pi)/q})
\end{equation*}
for $r \in [0,1]$, $\theta \in (0,2\pi)$, and $l = 1,2,\ldots,q$.  We will show $u_z \in C^{1,\mu;q}(\overline{B_1(0)})$ using the average-free and $k$-fold symmetry assumptions.  Since $g$ and $\varphi$ are average-free, 
\begin{equation} \label{poisson_transaf}
	\sum_{l=0}^{q-1} g_{0,z}(re^{i(\theta + 2\pi l/q)}) = \sum_{l=0}^{q-1} \varphi_{0,z}(re^{i(\theta + 2\pi l/q)},y) = 0
\end{equation}
for all $r \in [0,1]$ and $\theta \in [0,2\pi)$.  Since $g$ and $\varphi$ have $k$-fold symmetry and $k$ and $q$ are relatively prime, 
\begin{equation} \label{poisson_transsym}
	g_{0,z}(re^{i(\theta+2\pi/k)}) = g_{0,z}(re^{i\theta}), \quad \varphi_{0,z}(re^{i(\theta+2\pi/k)}) = \varphi_{0,z}(re^{i\theta})
\end{equation}
for all $r \in [0,1]$ and $\theta \in [0,2\pi)$.  By the maximum principle, (\ref{poisson_transaf}), and (\ref{poisson_transsym}), 
\begin{equation} \label{poisson_usym} 
	\sum_{l=0}^{q-1} u_{0,z}(re^{i(\theta + 2\pi l/q)}) = 0, \quad u_{0,z}(re^{i(\theta+2\pi/k)},y) = u_{0,z}(re^{i\theta},y), 
\end{equation}
for all $r \in [0,1]$ and $\theta \in [0,2\pi)$.  By (\ref{poisson_eqn3}), the degree $2q-1$ Taylor polynomial of $u_{0,z}$ is harmonic and thus takes the form 
\begin{equation*}
	\op{Re} \left( \sum_{j=0}^{2q-1} c_j (\xi_1+i\xi_2)^j \right)
\end{equation*}
for some $c_j \in \mathbb{C}$.  By (\ref{poisson_usym}), $c_j = 0$ for $j = 0,1,2,\ldots,q$.  By the Schauder estimates, 
\begin{equation*}
	\sum_{j=0}^2 |\xi|^j |D^j u_{0,z}(\xi)| 
	\leq \sup_{B_1(0)} |D^{q+1} u_{0,z}| |\xi|^{q+1} 
	\leq C \left( \|g_{0,z}\|_{C^q(B^2_1(0))} + \|\varphi_{0,z}\|_{C^{q+2}(\partial B^2_1(0)} \right) |\xi|^{q+1} 
\end{equation*}
for $\xi \in \overline{B^2_1(0)}$ for some constant $C = C(n,q,\mu,z,\rho_1,\ldots,\rho_{n-2}) \in (0,\infty)$.  By the change of variable $\xi_1+i\xi_2 = (x_1+ix_2)^{1/q}$, 
\begin{equation*}
	\sum_{j=0}^2 |x|^j |D^j u_{l,z}(x)| 
	\leq C \left( \|g_{0,z}\|_{C^q(B^2_1(0))} + \|\varphi_{0,z}\|_{C^{q+2}(\partial B^2_1(0))} \right) |x|^{1+1/q} 
\end{equation*}
for all $x \in B^2_1(0)$ and thus $u_z \in C^{1,1/q;q}(B^2_1(0))$ for all $z \in \mathbb{Z}^{n-2}$. 

For each integer $\nu \geq 1$, we define partial sums of the Fourier series of $u_0$, $g_0$, and $\varphi_0$ by 
\begin{gather*}
	u_0^{(\nu)}(\xi,y) = \sum_{|z| \leq \nu} u_{0,z}(\xi) e^{i 2\pi \sum_{j=1}^{n-2} z_j y_j/\rho_j}, \quad
	g_0^{(\nu)}(\xi,y) = \sum_{|z| \leq \nu} g_{0,z}(\xi) e^{i 2\pi \sum_{j=1}^{n-2} z_j y_j/\rho_j}, \\
	\varphi_0^{(\nu)}(\xi,y) = \sum_{|z| \leq \nu} \varphi_{0,z}(\xi) e^{i 2\pi \sum_{j=1}^{n-2} z_j y_j/\rho_j}, 
\end{gather*}
and define $u^{(\nu)} = (u^{(\nu)}_1,u^{(\nu)}_2,\ldots,u^{(\nu)}_q)$, $g^{(\nu)} = (g^{(\nu)}_1,g^{(\nu)}_2,\ldots,g^{(\nu)}_q)$, and $\varphi^{(\nu)} = (\varphi^{(\nu)}_1,\varphi^{(\nu)}_2,\ldots,\varphi^{(\nu)}_q)$ in $C^{0;q}(\overline{\mathcal{C}})$ by 
\begin{gather*}
	u^{(\nu)}_l(r e^{i\theta},y) = u^{(\nu)}_0(r^{1/q} e^{i(\theta + 2(l-1)\pi)/q},y), \quad
	g^{(\nu)}_l(r e^{i\theta},y) = g^{(\nu)}_0(r^{1/q} e^{i(\theta + 2(l-1)\pi)/q},y), \\
	\varphi^{(\nu)}_l(r e^{i\theta},y) = \varphi^{(\nu)}_0(r^{1/q} e^{i(\theta + 2(l-1)\pi)/q},y), 
\end{gather*}
for $r \in [0,1]$, $\theta \in (0,2\pi)$, $y \in \mathbb{R}^{n-2}$, and $l = 1,2,\ldots,q$.  Since 
\begin{equation*}
	u^{(\nu)}_l(x,y) = \sum_{z \in \mathbb{Z}^{n-2}} u_{l,z}(x) e^{i 2\pi \sum_{j=1}^{n-2} z_j y_j/\rho_j}, 
\end{equation*}
for all $x \in \overline{B^2_1(0)} \setminus [0,1] \times \{0\}$ and $y \in \mathbb{R}^{n-2}$ and $u_z \in C^{1,1/q;q}(B^2_1(0))$ for all $z \in \mathbb{Z}^{n-2}$, $u^{(\nu)} \in C^{1,1/q;q}(\overline{\mathcal{C}})$.  Since $g_0$ is smooth, 
\begin{equation*}
	\sup_{\xi \in B^2_1(0)} |g_{0,z}(\xi)| \leq C(1+|z_1|)^{-2} (1+|z_2|)^{-2} \cdots (1+|z_{n-2}|)^{-2} \|g_0\|_{C^2(\overline{\mathcal{C}})}. 
\end{equation*} 
Thus $g_0^{(\nu)}$ converges uniformly to $g_0$ on $\overline{\mathcal{C}}$ as $\nu \rightarrow \infty$.  Hence $g^{(\nu)}$ converges to $g$ in $C^{0;q}(\overline{\mathcal{C}})$ as $\nu \rightarrow \infty$.  Similarly $\varphi_0^{(\nu)}$ converges to $\varphi_0$ in $C^2(\overline{\mathcal{C}})$ and thus, since $\varphi_0 = 0$ in $B^2_{1/2}(0) \times \mathbb{R}^{n-2}$, $\varphi^{(\nu)}$ converges to $\varphi$ in $C^{2;q}(\overline{\mathcal{C}})$ as $\nu \rightarrow \infty$.  By the Schauder estimate of Lemma \ref{schauder_wg}, for every $\mu \in (0,1/q)$, 
\begin{equation*}
	\|u^{(\nu)}\|_{C^{1,\mu;q}(\overline{\mathcal{C}})} 
	\leq C \left( \sup_{\mathcal{C}} |g^{(\nu)}| + \|\varphi^{(\nu)}\|_{C^{1,\mu;q}(\overline{\mathcal{C}})} \right) 
	\leq C \left( \sup_{\mathcal{C}} |g| + \|\varphi\|_{C^{1,\mu;q}(\overline{\mathcal{C}})} + 1 \right) 
\end{equation*}
for some constant $C = C(n,q,\mu) \in (0,\infty)$ independent of $\nu$.  After passing to a subsequence, $\{u^{(\nu)}\}$ converges to some $u$ in $C^{1;q}(\overline{\mathcal{C}})$ such that $u \in C^{1,\mu;q}(\overline{\mathcal{C}})$ for all $\mu \in (0,1/q)$ and $u$ satisfies (\ref{poisson_thm_eqn}). 

Consider the case where $f^j \neq 0$ and $g_0$ and $\varphi_0$ are merely continuous.  We will construct functions $f^{(\nu)} \in C^{\infty;q}(\overline{\mathcal{C}};\mathbb{R}^n)$ and $g^{(\nu)}, \varphi^{(\nu)} \in C^{\infty;q}(\overline{\mathcal{C}})$ approximating $f$, $g$, and $\varphi$ as follows.  Extend $f$ to an element of $C^{0,\mu;q}_c(\mathbb{R}^n;\mathbb{R}^n)$ such that $[f]_{\mu;q,\mathbb{R}^n} \leq C [f]_{\mu;q,\mathcal{C}}$ for some $C = C(n,\mu) \in (0,\infty)$.  For each $\delta > 0$, let $\chi_{\delta} \in C^{\infty}(\mathbb{R}^2)$ be a single-valued function such that $0 \leq \chi_{\delta} \leq 1$, $\chi_{\delta} = 1$ on $\mathbb{R}^2 \setminus B^2_{\delta}(0)$, $\chi_{\delta} = 0$ on $B^2_{\delta/2}(0)$, and $|D\chi_{\delta}| \leq 3/\delta$ and extend $\chi_{\delta}$ to a function $\chi_{\delta}(x,y)$ of $x \in \mathbb{R}^2$ and $y \in \mathbb{R}^{n-2}$ that is independent of $y$.  Since $f \in C^{0,\mu;q}_c(\mathbb{R}^n;\mathbb{R}^n)$ and $f_l = 0$ on $\{0\} \times \mathbb{R}^{n-2}$ for $l = 1,2,\ldots,q$ because $f$ is average free, observe that $\chi_{\delta} f = (\chi_{\delta} f_1, \chi_{\delta} f_2, \ldots, \chi_{\delta} f_q)$ is in $C^{0,\mu;q}_c(\mathbb{R}^n;\mathbb{R}^n)$ with 
\begin{equation*}
	\chi_{\delta} f \rightarrow f \text{ in } C^0(B^2_2(0) \times \mathbb{R}^{n-2}) \text{ as } \delta \downarrow 0, \quad 
	[\chi_{\delta} f]_{\mu;q,\mathbb{R}^n} \leq C [f]_{\mu;q,\mathbb{R}^n} \text{ for } C = C(n,\mu) \in (0,\infty).
\end{equation*} 
Select a smooth spherically symmetric mollifier $\psi \in C_c^{\infty}(B_1(0))$ and let $\psi_{\sigma}(X) = \sigma^{-n} \psi(X/\sigma)$ for $X \in B_{\sigma}(0)$ and $\sigma > 0$.  Let $(x_1,x_2,y) \in \mathcal{C} \setminus B^2_{2^{-\nu-2}}(0) \times \mathbb{R}^{n-2}$.  If $x_1 \leq 0$ then we define $f^{(\nu)}_l(x_1,x_2,y)$ to be the value of the convolution of $f_l$ and $\psi_{2^{-\nu-3}}$ at $(x_1,x_2,y)$.  If $x_1 > 0$ and $x_2 > 0$ then we define $f^{(\nu)}_l(x_1,x_2,y)$ to be the value of the convolution of $\hat f_l$ and $\psi_{2^{-\nu-3}}$ at $(x_1,x_2,y)$, where $\hat f_l = f_l$ for $l = 1,2,\ldots,q$ on $\mathcal{C} \cap (0,\infty)^2 \times \mathbb{R}^{n-2}$ and $\hat f_1 = f_q$ and $\hat f_l = f_{l-1}$ for $l = 2,3,\ldots,q$ on $\mathcal{C} \cap (0,\infty) \times (-\infty,0) \times \mathbb{R}^{n-2}$.  If $x_1 > 0$ and $x_2 < 0$ then we define $f^{(\nu)}_l(x_1,x_2,y)$ to be the value of the convolution of $\hat f_l$ and $\psi_{2^{-\nu-3}}$ at $(x_1,x_2,y)$, where now $\hat f_l = f_l$ for $l = 1,2,\ldots,q$ on $\mathcal{C} \cap (0,\infty) \times (-\infty,0) \times \mathbb{R}^{n-2}$ and $\hat f_l = f_{l+1}$ for $l = 1,2,\ldots,q-1$ and $\hat f_q = f_1$ on $\mathcal{C} \cap (0,\infty)^2 \times \mathbb{R}^{n-2}$.  Define $f^{(\nu)}_l = 0$ on $B^2_{2^{-\nu-2}}(0) \times \mathbb{R}^{n-2}$.  Then 
\begin{equation*}
	f^{(\nu)} \rightarrow f \text{ in } C^{0;q}(\overline{\mathcal{C}}), \quad 
	[f^{(\nu)}]_{\mu;q,\mathcal{C}} \leq C [f]_{\mu;q,\mathcal{C}} \text{ for } C = C(n,\mu) \in (0,\infty).
\end{equation*} 
Define $g^{(\nu)}_0$ by convolution of $g_0$ with smooth spherically symmetric mollifiers such that $g^{(\nu)}_0 \rightarrow g_0$ uniformly on $\overline{\mathcal{C}}$ and then define $g^{(\nu)}$ by 
\begin{equation*}
	g^{(\nu)}_0(re^{i\theta},y) = g^{(\nu)}_l(r^q e^{i q\theta},y) \text{ if } 2(l-1)\pi/q < \theta < 2l\pi/q 
\end{equation*}
for $r \in [0,1]$, $y \in \mathbb{R}^{n-2}$, and $l = 1,2,\ldots,q$.  Assume $\varphi = 0$ on $B^2_{1/2}(0) \times \mathbb{R}^{n-2} \setminus [0,1/2) \times \{0\} \times \mathbb{R}^{n-2}$ and define $\varphi^{(\nu)}$ similarly via convolution of $\varphi_0$ with smooth spherically symmetric mollifiers.  Let $u^{(\nu)} = (u^{(\nu)}_1,u^{(\nu)}_2,\ldots,u^{(\nu)}_q) \in C^{1,\mu}(\overline{\mathcal{C}})$ to be the solution to (\ref{poisson_thm_eqn}) with $u^{(\nu)}_l$, $0$, $\op{div} f^{(\nu)}_l + g^{(\nu)}_l$, and $\varphi^{(\nu)}_l$ in place of $u_l$, $f_l$, $g_l$, and $\varphi_l$ respectively.  By global supremum estimates similar to~\cite[Theorem 8.16]{GT}, $\{u^{(\nu)}\}$ is Cauchy in $C^{0;q}(\overline{\mathcal{C}})$ and thus $\{u^{(\nu)}\}$ converges to some $u$ in $C^{0;q}(\overline{\mathcal{C}})$.  By the local Schauder estimates of Lemma \ref{schauder_div} after passing to a subsequence $\{u^{(\nu)}\}$ converges to $u$ in $C^{1;q}(\overline{\Omega})$ for all $\Omega \subset \subset \mathcal{C}$ and $u \in C^{1,\mu;q}(\mathcal{C})$.  Therefore $u$ is a solution to (\ref{poisson_thm_eqn}). 

To solve (\ref{poisson_thm_eqn}) in the case that $f$, $g$, and $\varphi$ are not periodic with respect to each $y_i$, approximate $f$, $g$, and $\varphi$ uniformly on compact subsets of $\overline{\mathcal{C}}$ by $f^{(\nu)} \in C^{0,\mu;q}(\overline{\mathcal{C}})$, $g^{(\nu)}, \varphi^{(\nu)} \in C^{0;q}(\overline{\mathcal{C}})$ such that (\ref{poisson_symmetry}) holds with $f^{(\nu)}$ in place of $f$, $g^{(\nu)}$ and $\varphi^{(\nu)}$ are $k$-fold symmetric, $f^{(\nu)}$, $g^{(\nu)}$, and $\varphi^{(\nu)}$ are periodic with respect to each $y_j$ with period $\rho_{\nu}$ (independent of $j$) such that $\rho_{\nu} \rightarrow \infty$ as $\nu \rightarrow \infty$, and 
\begin{equation} \label{poisson_cor_eqn2}
	[f^{(\nu)}]_{\mu;q,\mathcal{C}} \leq C [f]_{\mu;q,\mathcal{C}}, \quad \sup_{\mathcal{C}} |g^{(\nu)}| \leq \sup_{\mathcal{C}} |g|, \quad 
	\sup_{\partial \mathcal{C}} |\varphi^{(\nu)}| \leq \sup_{\partial \mathcal{C}} |\varphi|, 
\end{equation}
where $C = C(n,\mu) \in (0,\infty)$.  Let $u^{(\nu)} \in C^{0;q}(\overline{\mathcal{C}}) \cap C^{1,\mu;q}(\mathcal{C})$ solve (\ref{poisson_thm_eqn}) with $u^{(\nu)}_l$, $f^{(\nu)}_l$, $g^{(\nu)}_l$, and $\varphi^{(\nu)}_l$ in place of $u_l$, $f_l$, $g_l$, and $\varphi_l$ respectively.  By Lemma \ref{inhomog_poisson}, local Schauder estimates Lemma \ref{schauder_div}, and (\ref{poisson_cor_eqn2}) we have local $C^{1,\mu;q}$ estimates on $u^{(\nu)}$ that are independent of $\nu$ and thus after passing to a subsequence $\{u^{(\nu)}\}$ converges to some $u \in C^{1,\mu;q}(\mathcal{C})$ in $C^{1;q}(\Omega)$ for all $\Omega \subset \subset \mathcal{C}$.  By using~\cite[Theorem 8.27]{GT} and (\ref{poisson_cor_eqn2}), we can establish uniform modulus of continuity estimates on $u^{(\nu)}$ at points on $\partial \mathcal{C}$ that are independent of $\nu$ and thus $u$ extends to an element of $C^{0;q}(\overline{\mathcal{C}})$ such that $u = \varphi$ on $\partial \mathcal{C}$.  Therefore $u$ is a solution to (\ref{poisson_thm_eqn}).


\section{Existence of solutions to nonlinear systems} \label{sec:nonlinearsystems}

In this section we will prove Theorem \ref{theorem1}, which recall involves finding the unique solution $u = (u_1,u_2,\ldots,u_q) \in C^{1,\mu;q}(\overline{\mathcal{C}};\mathbb{R}^m)$ to (\ref{theorem1_dirprob}) given $F^j_{\kappa}$, $G_{\kappa}$, and $\varphi = (\varphi_1,\varphi_2,\ldots,\varphi_q) \in C^{1,\mu;q}(\overline{\mathcal{C}})$ as in the statement of Theorem \ref{theorem1}.  First we consider the case where $u$ and $\varphi$ are periodic with respect to $y_j$ with period $\rho_j > 0$ for $j = 1,2,\ldots,n-2$.  Let $\mathcal{V}$ denote the space of $u \in C^{1,\mu;q}(\overline{\mathcal{C}};\mathbb{R}^m)$ that has $k$-fold symmetry and is periodic with respect to $y_j$ with period $\rho_j$ for $j = 1,2,\ldots,n-2$.  By Theorem \ref{poisson_thm}, we can define $T : \mathcal{V} \rightarrow \mathcal{V}$ by letting $u = Tv$ if $u = (u_1,u_2,\ldots,u_q), v = (v_1,v_2,\ldots,v_q) \in \mathcal{V}$ satisfy 
\begin{align*}
	\int_{\mathcal{C}} \sum_{l=1}^q D_j u_l^{\kappa} D_j \zeta_l^{\kappa} 
		&= \int_{\mathcal{C}} \sum_{l=1}^q (F^j_{\kappa}(Dv_l) D_j \zeta^{\kappa} - G_{\kappa}(v_l,Dv_l) \zeta^{\kappa}) \text{ for all } 
		\zeta \in C_c^{1;q}(\mathcal{C} \setminus \{0\} \times \mathbb{R}^{n-2};\mathbb{R}^m), \nonumber \\
	u_l &= \varphi_l \text{ on } \partial \mathcal{C} \text{ for } l = 1,2,\ldots,q. 
\end{align*}
Let $\varepsilon > 0$ to be determined and choose arbitrary $v = (v_1,v_2,\ldots,v_q), v' = (v'_1,v'_2,\ldots,v'_q) \in \mathcal{V}$ such that $\|v\|_{C^{1,\mu;q}(\overline{\mathcal{C}})} \leq \varepsilon$ and $\|v'\|_{C^{1,\mu;q}(\overline{\mathcal{C}})} \leq \varepsilon$.  Let $u = (u_1,u_2,\ldots,u_q) = Tv$ and $u' = (u'_1,u'_2,\ldots,u'_q) = Tv'$.  By the Schauder estimate Lemma \ref{schauder_wg}, 
\begin{align} \label{theorem1_computation1}
	\|u\|_{C^{1,\mu;q}(\overline{\mathcal{C}})} &\leq C \left( [F(Dv)]_{\mu, \mathcal{C}} + \sup_{\mathcal{C}} |G(v,Dv)| 
		+ \|\varphi\|_{C^{1,\mu;q}(\overline{\mathcal{C}})} \right), \nonumber \\ 
	\|u-u'\|_{C^{1,\mu;q}(\overline{\mathcal{C}})} &\leq C \left( [F(Dv) - F(Dv')]_{\mu;q,\mathcal{C}} + \sup_{\mathcal{C}} |G(v,Dv) - G(v',Dv')| \right) , 
\end{align}
for some constant $C \in (0,\infty)$ depending on $n$, $m$, $q$, and $\mu$ and independent of $\rho_1,\ldots,\rho_{n-2}$, where $F(Dw) = (F(Dw_1),F(Dw_2),\ldots,F(Dw_q))$ and $G(w,Dw) = (G(w_1,Dw_1),G(w_2,Dw_2),\ldots,G(w_q,Dw_q))$ for $w = (w_1,w_2,\ldots,w_q) \in \mathcal{V}$.  Note that $C$ being independent of $\rho_1,\ldots,\rho_{n-2}$ is necessary for later removing the condition that $\varphi$ be periodic with respect to $y_j$.  Since $F \in C^2(\mathbb{R}^{mn})$ with $DF(0) = 0$,
\begin{align} \label{theorem1_computation2}
	[F(Dv)]_{\mu;q,\mathcal{C}} &\leq C \sup_{|P| \leq \varepsilon} |DF(P)| \, [Dv]_{\mu;q,\mathcal{C}} \leq C \varepsilon^2 \nonumber \\
	[F(Dv) - F(Dv')]_{\mu;q,\mathcal{C}}  
	&\leq C \sup_{|P| \leq \varepsilon} |D^2 F(P)| ([Dv]_{\mu;q,\mathcal{C}} + [Dv']_{\mu;q,\mathcal{C}}) \sup_{\mathcal{C}} |Dv - Dv'| \nonumber 
		\\& \hspace{4.4mm} + C \sup_{|P| \leq \varepsilon} |DF(P)| [Dv - Dv']_{\mu;q,\mathcal{C}} \nonumber \\
	&\leq C \varepsilon \|v - v'\|_{C^{1,\mu;q}(\overline{\mathcal{C}})}
\end{align}
for some constants $C \in (0,\infty)$ depending on $n$, $m$, and $\sup_{|Z|+|P| \leq 1} |D^2 F(P)|$.  Since $G \in C^1(\mathbb{R}^m \times \mathbb{R}^{mn})$ with $G(0) = 0$, 
\begin{align} \label{theorem1_computation4}
	\sup_{\mathcal{C}} |G(v,Dv)| &\leq \sup_{|Z|+|P| \leq \varepsilon} |G(P)| \leq C \sup_{|Z|+|P| \leq \varepsilon} |DG(P)| \, \varepsilon, \nonumber \\ 
	\sup_{\mathcal{C}} |G(v,Dv) - G(v',Dv')| &\leq C \sup_{|Z|+|P| \leq \varepsilon} |DG(P)| \, \|v - v'\|_{C^{1;q}(\overline{\mathcal{C}})} 
\end{align}
for some constant $C = C(n,m,q) \in (0,\infty)$.  Combining (\ref{theorem1_computation1}), (\ref{theorem1_computation2}), and (\ref{theorem1_computation4}) and using the fact that $DG(0) = 0$, for some $\varepsilon > 0$ depending on $n$, $m$, $q$, $\mu$, $F$, and $G$, 
\begin{equation*}
	\|u\|_{C^{1,\mu;q}(\overline{\mathcal{C}})} \leq \varepsilon, \quad
	\|u - u'\|_{C^{1,\mu;q}(\overline{\mathcal{C}})} \leq \frac{1}{2} \|v - v'\|_{C^{1,\mu;q}(\overline{\mathcal{C}})}. 
\end{equation*}
Therefore by the contraction mapping principle, there exists a fixed point $u_0 \in \mathcal{V}$ of $T$ with $\|u\|_{C^{1,\mu;q}(\overline{\mathcal{C}})} \leq \varepsilon$.  In other words, $u$ satisfies (\ref{theorem1_dirprob}).

To remove the condition that $\varphi$ is periodic with respect to $y_j$, approximate $\varphi^{(\nu)}$ in $C^{1;q}(\overline{\mathcal{C}};\mathbb{R}^m)$ by $\varphi^{(\nu)} \in C^{1,\mu;q}(\overline{\mathcal{C}};\mathbb{R}^m)$ such that $\varphi^{(\nu)}$ is $k$-fold symmetric and is periodic with respect to each $y_j$ with period $\rho_{\nu}$ where $\rho^{(\nu)} \rightarrow \infty$, $\varphi^{(\nu)} \rightarrow \varphi$ in $C^{1;q}$ on compact subsets of $\overline{\mathcal{C}}$, and $\|\varphi^{(\nu)}\|_{C^{1,\mu;q}(\overline{\mathcal{C}})} \rightarrow \|\varphi\|_{C^{1,\mu;q}(\overline{\mathcal{C}})}$.  Let $u^{(\nu)} \in C^{1,\mu;q}(\overline{\mathcal{C}};\mathbb{R}^m)$ be the unique solution to (\ref{theorem1_dirprob}) with $u^{(\nu)}$ and $\varphi^{(\nu)}$ in place of $u$ and $\varphi$ respectively.  Since $\|u^{(\nu)}\|_{C^{1,\mu;q}(\overline{\mathcal{C}})} \leq \varepsilon$, after passing to a subsequence $\{u^{(\nu)}\}$ converges in $C^{1;q}(B^2_1(0) \times B^{n-2}_{\rho}(0);\mathbb{R}^m)$ for all $\rho \in (0,\infty)$ to $u$ satisfying (\ref{theorem1_dirprob}) with the original $\varphi$.


\section{Existence of solutions to nonlinear equations} \label{sec:nonlineareqns}

In this section we will prove Theorem \ref{theorem2}, which recall involves finding a solution $u = (u_1,u_2,\ldots,u_q) \in C^{1,\tau;q}(\overline{\mathcal{C}};\mathbb{R}^m)$ for every $\tau \in (0,1/q)$ to (\ref{theorem2_dirprob}) given $A^i$, $B$, and $\varphi = (\varphi_1,\varphi_2,\ldots,\varphi_q) \in C^{2;q}(\overline{\mathcal{C}})$ as in the statement of Theorem \ref{theorem2}.  The proof uses the Leray-Schauder theory.  For now assume $\varphi$ is periodic with respect to $y_j$ with period $\rho_j$ for $j = 1,2,\ldots,n-2$.  Rewrite (\ref{theorem2_dirprob}) as 
\begin{align*}
	(D_{P_j} A^i)(Du_l) D_{ij} u_l + B(u_l,Du_l) &= 0 \text{ in } \mathcal{C} \setminus \{0\} \times \mathbb{R}^{n-2} \text{ for } l = 1,2,\ldots,q, \\
	u_l &= \varphi_l \text{ on } \partial \mathcal{C} \text{ for } l = 1,2,\ldots,q. 
\end{align*}
Let $\mathcal{V}$ denote the space of $u \in C^{1,\mu;q}(\overline{\mathcal{C}})$ that are periodic with respect to $y_j$ with period $\rho_j$ for $j = 1,2,\ldots,n-2$ and have $k$-fold symmetry.  Define $T : \mathcal{V} \rightarrow \mathcal{V}$ by $u = Tv$ if $u = (u_1,u_2,\ldots,u_q), v = (v_1,v_2,\ldots,v_q) \in \mathcal{V}$ satisfy 
\begin{align} \label{theorem2_defnT}
	(D_{P_j} A^i)(Dv_l) D_{ij} u_l + B(v_l,Dv_l) &= 0 \text{ in } \mathcal{C} \setminus [0,1) \times \{0\} \times \mathbb{R}^{n-2} 
		\text{ for } l = 1,2,\ldots,q, \nonumber \\
	u_l &= \varphi_l \text{ on } \partial \mathcal{C} \text{ for } l = 1,2,\ldots,q. 
\end{align}
The existence of a unique $u \in \mathcal{V} \cap C^{2;q}(\mathcal{C} \setminus \{0\} \times \mathbb{R}^{n-2})$ satisfying (\ref{theorem2_defnT}) will follow from Lemma \ref{Texistslemma} below.  By Lemma \ref{Texistslemma}, Schauder estimate Lemma \ref{schauder_strong}, and local boundary Schauder estimates for single-valued solutions~\cite[Section 6.2]{GT}, $T$ is in fact a continuous map from $\mathcal{V}$ into $C^{1,\tau;q}(\overline{\mathcal{C}})$ for every $\tau \in (\mu,1/q)$, so by Arzela-Ascoli $T$ is compact.  We will need to show that for some constants $\mu \in (0,1/q)$ and $C \in (0,\infty)$ depending only on $n$, $A^i$, $B$, and $\|\varphi\|_{C^{2;q}(\overline{\mathcal{C}})}$, if $u \in C^{1,\mu;q}(\overline{\mathcal{C}})$ and $\sigma \in [0,1]$ satisfies 
\begin{align} \label{theorem2_evT}
	\int_{\mathcal{C}} \sum_{l=1}^q (A^i(Du_l) D_i \zeta_l - \sigma B(u_l,Du_l) \zeta_l) &= 0
		\text{ for all } \zeta \in C_c^{1;q}(\mathcal{C} \setminus \{0\} \times \mathbb{R}^{n-2}), \nonumber \\
	u_l &= \sigma \varphi_l \text{ on } \partial \mathcal{C} \text{ for } l = 1,2,\ldots,q, 
\end{align}
then $\|u\|_{C^{1,\mu;q}(\overline{\mathcal{C}})} \leq C$.  Then by the Leray-Schauder theory, there exists a fixed point $u \in \mathcal{V}$ of $T$.  In other words, the $u$ solves (\ref{theorem2_dirprob}).  Note that $u = Tu \in C^{1,\tau;q}(\overline{\mathcal{C}})$ for all $\tau \in (0,1/q)$. 

\begin{lemma} \label{Texistslemma}
Let $1 < \mu \leq \tau < 1/q$ and $a^{ij} = (a^{ij}_1,a^{ij}_2,\ldots,a^{ij}_q), g = (g_1,g_2,\ldots,g_q) \in C^{0,\mu;q}(\overline{\mathcal{C}})$ and $\varphi = (\varphi_1,\varphi_2,\ldots,\varphi_q) \in C^{1,\tau;q}(\overline{\mathcal{C}})$.  Suppose $a^{ij}$, $g$, and $\varphi$ are periodic with respect to $y_l$ with period $\rho_l$ for $l = 1,2,\ldots,n-2$, $a^{ij}(\mathbf{R} X) = a^{i' j'}(X) R_{i'}^i R_{j'}^j$ for all $X \in \overline{\mathcal{C}}$, and 
\begin{equation*}
	a^{ij}(X) \xi_i \xi_j \geq \lambda |\xi|^2 \text{ for } X \in \overline{\mathcal{C}}, \, \xi \in \mathbb{R}^n, 
	\quad \|a^{ij}\|_{C^{0,\mu;q}(\overline{\mathcal{C}})} \leq \Lambda, 
\end{equation*} 
for some constants $\lambda, \Lambda > 0$.  Then there exists a unique $u = (u_1,u_2,\ldots,u_q) \in C^{0;q}(\overline{\mathcal{C}}) \cap C^{1,\tau;q}(\mathcal{C}) \cap C^{2;q}(\mathcal{C} \setminus \{0\} \times \mathbb{R}^{n-2})$ such that 
\begin{align} \label{Texists_dirprob1}
	a_l^{ij} D_{ij} u_l &= g_l \text{ in } \mathcal{C} \setminus [0,1) \times \{0\} \times \mathbb{R}^{n-2} \text{ for } l = 1,2,\ldots,q, \nonumber \\
	u_l &= \varphi_l \text{ on } \partial \mathcal{C} \text{ for } l = 1,2,\ldots,q. 
\end{align}
Moreover, $u$ is $k$-fold symmetric and periodic with respect to $y_l$ with period $\rho_l$ for $l = 1,2,\ldots,n-2$. 
\end{lemma}
\begin{proof}
	For now suppose $a^{ij}, \varphi \in C^{\infty;q}(\overline{\mathcal{C}})$.  By replacing $u$ with $u-\varphi$, we may suppose $\varphi = 0$.  Re-write (\ref{Texists_dirprob1}) as 
\begin{align} \label{Texists_dirprob2}
	\int_{\mathcal{C}} \sum_{l=1}^q (a_l^{ij} D_j u_l D_i \zeta_l + D_i a_l^{ij} D_j u_l \zeta_l) &= -\int_{\mathcal{C}} \sum_{l=1}^q g_l \zeta_l
		\text{ for all } \zeta \in C_c^{1;q}(\mathcal{C} \setminus \{0\} \times \mathbb{R}^{n-2}), \nonumber \\
	u_l &= 0 \text{ on } \partial \mathcal{C} \text{ for } l = 1,2,\ldots,q. 
\end{align}
We will solve (\ref{Texists_dirprob2}) using the method of continuity.  Let $\mathcal{W}$ denote the space of $u \in C^{1,\tau;q}(\overline{\mathcal{C}})$ such that $u$ is periodic with respect to $y_j$ with period $\rho_j$ for $j = 1,2,\ldots,n-2$, $u$ has $k$-fold symmetry, and $u_l = 0$ on $\partial \mathcal{C}$ for $l = 1,2,\ldots,q$.  Define the family $\{L_t\}_{t \in [0,1]}$ of weak linear operators on $\mathcal{W}$ by  
\begin{equation*}
	L_t u_l = (1-t) \Delta u_l + t (D_i (a_l^{ij} D_j u_l) - (D_i a_l^{ij}) D_j u_l)
\end{equation*}
and consider 
\begin{align} \label{methodofcont}
	L_t u_l &= D_j f^j_l + g_l \text{ weakly in } \mathcal{C} \setminus \{0\} \times \mathbb{R}^{n-2}, \nonumber \\
	u_l &= 0 \text{ on } \partial \mathcal{C} 
\end{align}
for $f^j = (f^j_1,f^j_2,\ldots,f^j_q) \in C^{0,\tau;q}(\overline{\mathcal{C}})$ and $g = (g_1,g_2,\ldots,g_q) \in C^{0;q}(\overline{\mathcal{C}})$ such that $f^j$ and $g$ are periodic with respect to $y_l$ with period $\rho_l$ for $l = 1,2,\ldots,n-2$, $f^j$ satisfies (\ref{poisson_symmetry}), and $g$ has $k$-fold symmetry.  By Theorem \ref{poisson_thm}, for $t = 0$ we can find a unique weak solution $u \in \mathcal{W}$ to (\ref{methodofcont}).  Suppose we can find a unique solution $u \in \mathcal{W}$ to (\ref{methodofcont}) for $t = s$ for some $s \in [0,1]$.  Then (\ref{methodofcont}) can be rewritten as 
\begin{align*}
	L_s u_l &= D_j f^j_l + g + (L_s - L_t) u_l \text{ in } \mathcal{C} \setminus \{0\} \times \mathbb{R}^{n-2}, \\
	u_l &= 0 \text{ on } \partial \mathcal{C}. 
\end{align*}
Define a map $U : \mathcal{W} \rightarrow \mathcal{W}$ by $u = U(v)$ for $u = (u_1,u_2,\ldots,u_q), v = (v_1,v_2,\ldots,v_q) \in \mathcal{W}$ if 
\begin{align*}
	L_s u_l &= D_j f^j_l + g + (L_s - L_t) v_l \text{ in } \mathcal{C} \setminus \{0\} \times \mathbb{R}^{n-2}, \\
	u_l &= 0 \text{ on } \partial \mathcal{C}. 
\end{align*}
Choose arbitrary $v = (v_1,v_2,\ldots,v_q), v' = (v'_1,v'_2,\ldots,v'_q) \in \mathcal{W}$ and let $u = (u_1,u_2,\ldots,u_q) = U(v)$ and $u' = (u'_1,u'_2,\ldots,u'_q) = U(v')$.  Since 
\begin{align*}
	L_s(u_l - u'_l) =& (L_s - L_t)(v_l - v'_l) 
	\\ ={}& (s-t) (-\Delta (v_l - v'_l) + D_i (a^{ij}_l D_j (v_l - v'_l)) - (D_i a^{ij}_l) D_j (v_l - v'_l)), 
\end{align*} 
by Schauder estimate Lemma \ref{schauder_wg} 
\begin{align*}
	\|u - u'\|_{C^{1,\tau;q}(\overline{\mathcal{C}})} 
	&\leq C |s-t| \left( [Dv - Dv']_{\tau;q,\mathcal{C}} + [a^{ij} D_j (v - v')]_{\tau;q,\mathcal{C}} 
		+ \sup_{\mathcal{C}} |D_i a^{ij}| \, |Dv_l - Dv'_l| \right) 
	\\&\leq C |s-t| \|v - v'\|_{C^{1,\tau;q}(\overline{\mathcal{C}})} ,
\end{align*}
where $C \in (0,\infty)$ depends only on $n$, $q$, $\tau$, $\lambda$, $\Lambda$, $\|a^{ij}\|_{C^{1;q}(\overline{\mathcal{C}})}$, and $\rho_1,\ldots,\rho_{n-2}$.  So if $|s-t| < 1/2C$, then $U$ is a contraction mapping and we can solve (\ref{methodofcont}) for $t$ with $|s-t| < 1/2C$.  By dividing $[0,1]$ into intervals of length less than $1/2C$ and applying this result we conclude that we can solve (\ref{methodofcont}) for all $t \in [0,1]$, in particular for $t = 1$.  This gives us a $u \in \mathcal{W}$ satisfying (\ref{Texists_dirprob2}).  By elliptic regularity, if $g \in C^{0,\mu;q}(\overline{\mathcal{C}})$ then $u \in C^{2,\mu}(\mathcal{C} \setminus \{0\} \times \mathbb{R}^{n-2})$ and thus $u$ satisfies (\ref{Texists_dirprob1}). 

To solve (\ref{Texists_dirprob1}) for general $a^{ij}$ and $\varphi$, approximate $a^{ij}$ and $\varphi$ by approximating their average parts by convolution with smooth, spherically symmetric mollifiers and approximate their average-free parts using the same construction as in the proof of Theorem \ref{poisson_thm} to approximate $f$ by elements of $C^{\infty;q}(\overline{\mathcal{C}})$ to get $a^{ij}_{\nu} \in C^{\infty;q}(\overline{\mathcal{C}})$ converging to $a^{ij}$ uniformly on $\mathcal{C} \setminus [0,1) \times \{0\} \times \mathbb{R}^{n-2}$ and $\varphi^{(\nu)} \in C^{\infty;q}(\overline{\mathcal{C}})$ converging to $\varphi$ uniformly on $\mathcal{C} \setminus [0,1) \times \{0\} \times \mathbb{R}^{n-2}$.  Let $u^{(\nu)} \in C^{1,\tau;q}(\overline{\mathcal{C}})$ solve (\ref{Texists_dirprob1}) with $u^{(\nu)}$, $a^{ij}_{\nu}$, $g$, and $\varphi^{(\nu)}$ in place of $u$, $a^{ij}$, $g$, and $\varphi$.  By an extension of~\cite[Theorem 3.7]{GT} proven using maximum principle Lemma \ref{maxp_strong}, 
\begin{equation*}
	\sup_{\mathcal{C}} |u^{(\nu)}| \leq \sup_{\partial \mathcal{C}} |\varphi| + \frac{C}{\lambda} \sup_{\mathcal{C}} |g| 
\end{equation*}
for some constant $C = C(n) \in (0,\infty)$.  Thus by the Schauder estimate Lemma \ref{schauder_strong}, after passing to a subsequence $\{u^{(\nu)}\}$ converges in $C^{1;q}$ on compact subsets of $\mathcal{C}$ to $u \in C^{1,\tau;q}(\mathcal{C})$.  By the local Schauder estimates~\cite[Corollary 6.3]{GT}, after passing to a subsequence $\{u^{(\nu)}\}$ converges to some $u$ in $C^{2;q}$ on compact subsets of $\mathcal{C} \setminus \{0\} \times \mathbb{R}^{n-2}$ and thus $a_l^{ij} D_{ij} u_l = g_l$ in $\mathcal{C} \setminus [0,1) \times \{0\} \times \mathbb{R}^{n-2}$ for $l = 1,2,\ldots,q$.   By local barriers~\cite[Section 6.3]{GT}, we can establish uniform modulus of continuity estimates for $u^{(\nu)}$ at points on $\partial \mathcal{C}$ that are independent of $\nu$ and thus $u$ extends to a continuous function on $\overline{\mathcal{C}}$ such $u_l = \varphi_l$ on $\partial \mathcal{C}$. 
\end{proof}

Suppose that $u \in C^{1;q}(\overline{\mathcal{C}})$ satisfies (\ref{theorem2_evT}).  We want to bound $\|u\|_{C^{1,\mu;q}(\overline{\mathcal{C}})}$ for some $\mu \in (0,1/q)$.  By extending~\cite[Theorem 10.3]{GT} using maximum principle Lemma \ref{maxp_strong} and by (\ref{theorem2_structure1}), 
\begin{equation} \label{theorem2_sup}
	\sup_{\overline{\mathcal{C}}} |u| \leq M_0 \quad \text{where} \quad M_0 = \sup_{\partial \mathcal{C}} |\varphi| + C\beta_2
\end{equation}
for some constant $C = C(\beta_1) \in (0,\infty)$, where $\beta_1$ and $\beta_2$ as in (\ref{theorem2_structure1}).  By a standard argument involving local barriers~\cite[Corollary 14.3]{GT} along $\partial \mathcal{C}$ using structure condition (\ref{theorem2_structure2}),
\begin{equation} \label{theorem2_bgradient}
	\sup_{\partial \mathcal{C}} |Du| \leq M_1, 
\end{equation}
for some $M_1 \in (0,\infty)$ depends on $n$, $\|\varphi\|_{C^{2;q}(\overline{\mathcal{C}})}$, $\beta_1$, $\beta_2$, and $\beta_3$, where $\beta_3$ is as in (\ref{theorem2_structure2}). 

We want to show $u \in W^{2,2;q}(\Omega)$ for all $\Omega \subset \subset \mathcal{C}$.  By replacing $\zeta_l$ with $D_p \zeta$ for a single-valued function $\zeta \in C^{1;q}_c(\mathcal{C} \setminus \{0\} \times \mathbb{R}^{n-2})$ and $p \in \{1,2,\ldots,n\}$ in the first equation in (\ref{theorem2_dirprob}) and integrating by parts, 
\begin{equation} \label{theorem2_DNeqn}
	\int_{\mathcal{C}} \sum_{l=1}^q (D_{P_j} A^i(Du_l) D_{jp} u_l D_i \zeta - D_{P_j} B(u_l,Du_l) D_{jp} u_l \zeta - D_Z B(u_l,Du_l) D_p u_l \zeta) = 0 
\end{equation}
for all $\zeta \in C^{1;q}_c(\mathcal{C} \setminus \{0\} \times \mathbb{R}^{n-2})$.  Let $\zeta = D_p u \eta^2 \chi_{\delta}^2$ in (\ref{theorem2_DNeqn}), where $\eta \in C^1_c(\mathcal{C})$ is the single-valued cutoff function such that $0 \leq \eta \leq 1$, $\eta = 1$ on $B_{R/2}(X_0)$, $\eta = 0$ on $\mathbb{R}^n \setminus B_R(X_0)$, and $|D\eta| \leq 3/R$ and $\chi_{\delta}$ is the function such that $0 \leq \chi_{\delta} \leq 1$, $\chi_{\delta}(x,y) = 1$ if $|x| \geq \delta$, $\chi_{\delta}(x,y) = 1$ if $|x| \leq \delta/2$, and $|D\chi_{\delta}| \leq 3/\delta$.  By a standard computation using the fact that $D_{P_j} A^i(P) \xi_i \xi_j \geq \lambda(P) |\xi|^2$ for all $P \in \mathbb{R}^n$ and $\xi \in \mathbb{R}^n$ and using Cauchy's inequality, 
\begin{equation*}
	\int_{B_{R/2}(X_0)} \sum_{l=1}^q |D^2 u_l|^2 \chi_{\delta}^2 \leq C \int_{B_R(X_0)} (\chi_{\delta}^2 + |D\chi_{\delta}|^2) \leq C 
\end{equation*}
for some constants $C \in (0,\infty)$ depending on $n$, $R$, $A^i$, $B$, $M_0$, and $\sup_{\mathcal{C}} |Du|$ and independent of $\delta$.  Letting $\delta \downarrow 0$, $\|D^2 u\|_{L^2(B_{R/2}(X_0))} \leq C$ for all $B_R(X_0) \subset \mathcal{C}$.  

Now let $v(X) = (|Du_1(X)|^2,\ldots, |Du_q(X)|^2) \in C^{1;q}(\mathcal{C} \setminus \{0\} \times \mathbb{R}^{n-2})$.  By replacing $\zeta_l$ with $\sum_{p=1}^n D_p (D_p u \zeta_l)$ in the first equation in (\ref{theorem2_dirprob}) and using integration by parts, we get 
\begin{equation*}
	\int_{\mathcal{C}} \sum_{l=1}^q (D_{P_j} A^i(Du_l) D_j v_l D_i \zeta_l + 2 D_{P_j} A^i(Du_l) D_{pj} u_l D_{ip} u_l \zeta_l 
		- D_{P_j} B(u_l,Du_l) D_j v_l \zeta_l - 2D_Z B(u_l,Du_l) v_l \zeta_l) = 0
\end{equation*}
for all $\zeta \in C_c^{1;q}(\mathcal{C} \setminus \{0\} \times \mathbb{R}^{n-2})$.  Since $D_{P_j} A^i(Du_l) D_{ip} u_l D_{jp} u_l \geq 0$, 
\begin{equation*}
	\int_{\mathcal{C}} \sum_{l=1}^q (D_{P_j} A^i(Du_l) D_j v_l D_i \zeta_l - D_{P_j} B(u_l,Du_l) D_j v_l \zeta_l - 2D_Z B(u_l,Du_l) v_l \zeta_l) \leq 0
\end{equation*}
for all $\zeta \in C_c^{1;q}(\mathcal{C} \setminus \{0\} \times \mathbb{R}^{n-2})$ such that $\zeta_l \geq 0$ in $\mathcal{C} \setminus \{0\} \times \mathbb{R}^{n-2}$ for $l = 1,2,\ldots,q$.  By the weak maximum principle similar to~\cite[Theorem 8.1]{GT}
\begin{equation*}
	\sup_{\mathcal{C}} |Du| \leq \sup_{\partial \mathcal{C}} |Du| \leq M_1, 
\end{equation*} 
where $M_1$ is the constant from (\ref{theorem2_bgradient}).  

By the interior H\"{o}lder continuity estimate Lemma \ref{DgNM_Holder_thm} applied to (\ref{theorem2_DNeqn}) and the boundary H\"{o}lder continuity estimates for single-valued functions~\cite[Section 13.1]{GT} we obtain 
\begin{equation*}
	[u]_{\mu;q,\overline{\mathcal{C}}} \leq C
\end{equation*}
for some constant $C \in (0,\infty)$ and $\mu \in (0,1/q)$ depending on $n$, $q$, $\beta_1$, $\beta_2$, $\beta_3$, $\bar \lambda$, $\sup_{B^n_{M_1}(0)} |D_P A^i|$, $\sup_{(-M_0,M_0) \times B^n_{M_1}(0)} |D_P B|$, $\sup_{(-M_0,M_0) \times B^n_{M_1}(0)} |D_Z B|$, and $\|\varphi\|_{C^{2;q}(\overline{\mathcal{C}})}$.  Therefore we have shown that if $u$ satisfies (\ref{theorem2_evT}), then 
\begin{equation} \label{theorem2_finalest}
	\|u\|_{C^{1,\mu;q}(\overline{\mathcal{C}})} \leq C
\end{equation}
for some constant $C \in (0,\infty)$ and $\mu \in (0,1/q)$ depending on $n$, $q$, $\beta_1$, $\beta_2$, $\beta_3$, $\bar \lambda$, $\sup_{B^n_{M_1}(0)} |D_P A^i|$, $\sup_{(-M_0,M_0) \times B^n_{M_1}(0)} |D_P B|$, $\sup_{(-M_0,M_0) \times B^n_{M_1}(0)} |D_Z B|$, and $\|\varphi_0\|_{C^{2;q}(\overline{\mathcal{C}})}$. 

To solve (\ref{theorem2_dirprob}) in the case that $\varphi \in C^{2;q}(\overline{\mathcal{C}})$ with $\|\varphi\|_{C^{2;q}(\overline{\mathcal{C}})} < \infty$ and $\varphi$ is not periodic, approximate $\varphi$ in $C^{2;q}(\overline{\mathcal{C}})$ by $\varphi^{(\nu)}$ that are $k$-fold symmetric and periodic with respect to each $y_j$ with period $\rho_{\nu}$ such that $\rho_{\nu} \rightarrow \infty$ as $\nu \rightarrow \infty$.  Let $u^{(\nu)} \in C^{1,\mu;q}(\overline{\mathcal{C}})$ solve (\ref{theorem2_dirprob}) with $u^{(\nu)}$ and $\varphi^{(\nu)}$ in place of $u$ and $\varphi$ respectively.  By (\ref{theorem2_finalest}), after passing to a subsequence $\{u^{(\nu)}\}$ converges in $C^{1;q}(\mathcal{C})$ to a solution $u$ to (\ref{theorem2_dirprob}) with the original $\varphi$.  By Schauder estimate Lemma \ref{schauder_strong} and local boundary Schauder estimates for single-valued solutions~\cite[Section 6.2]{GT} for every $\tau \in (0,1/q)$ we have uniformly local $C^{1,\tau;q}$ estimates on $u^{(\nu)}$ that are independent of $\nu$ and thus $u \in C^{1,\tau;q}(\overline{\mathcal{C}})$. 

For Corollary \ref{corollary2}, we need to obtain an interior gradient estimate without using $\varphi \in C^{2;q}(\overline{\mathcal{C}})$.  We will do so by extending an interior gradient estimate due to Simon~\cite[Theorem 1]{Simon_interior} to solutions $u \in C^{1;q}(\mathcal{C})$.  This requires using cutoff function arguments to handle of singularity of $u$ along $\{0\} \times \mathbb{R}^{n-2}$.  For example, the analogue of (2.11) of~\cite{Simon_interior} is 
\begin{align} \label{gradlemma_eqn3}
	&\sum_{l=1}^q \int_{\{v_l \geq \tau\}} \left( (1-\tau/v_l) \mathscr{C}_l^2 + v_l D_{P_j} A^i(Du_l) D_i \omega_l D_j \omega_l \right) v_l \chi(v_l) \zeta^2 
		\nonumber \\
	&\leq 8n (1+c_{\chi})^2 \sum_{l=1}^q \int_{\{v_l \geq \tau\}} \left( \Lambda(v) \beta_1^2 \zeta^2 + 2 \bar \mu_l |D\zeta|^2 \right) v_l \chi(v_l) 
\end{align}
for all $\zeta \in C^{1;q}_c(\mathcal{C} \setminus \{0\} \times \mathbb{R}^{n-2})$, where $v_l = (1+|Du_l|^2)^{1/2}$, $g_l^{ij} = \delta_{ij} - D_i u_l D_j u_l/(1+|Du_l|^2)$, $\mathscr{C}_l^2 = v_l^{-1} D_{P_j} A^i(Du_l) g_l^{kk'} D_{ik} u_l D_{jk'} u_l$, and $\bar \mu_l$ satisfies 
\begin{equation*}
	|v_l D_{P_j} A^i(Du_l) \xi_j \eta_i| \leq \left( \bar \mu_l |\eta|^2 \right)^{1/2} \left( v_l D_{P_j} A^i(Du_l) \xi_i \xi_j \right)^{1/2} 
	\text{ on } \mathcal{C} \setminus [0,1] \times \{0\} \times \mathbb{R}^{n-2} \text{ for } \xi, \eta \in \mathbb{R}^n
\end{equation*}
and $\chi$ and $\Lambda$ are single-valued functions as in~\cite{Simon_interior}.  Since $u \in W^{2,2;q}(\Omega)$ for all $\Omega \subset \subset \mathcal{C}$, we can show (\ref{gradlemma_eqn3}) holds for any $\zeta \in C^{0;q}_c(\mathcal{C}) \cap W^{1,2;q}_0(\mathcal{C})$ by replacing $\zeta$ by $\zeta \psi_{\delta}$ in (\ref{gradlemma_eqn3}) for $\delta > 0$, where $\psi_{\delta} \in C^{\infty}(\mathcal{C})$ is the logarithmic cutoff function defined by $\psi_{\delta}(x,y) = 0$ if $|x| \leq \delta^2$, $\psi_{\delta}(x,y) = -\log(|x|/\delta^2)/\log(\delta)$ if $\delta^2 < |x| < \delta$, and $\psi_{\delta}(x,y) = 1$ if $|x| \geq \delta$, and then letting $\delta \downarrow 0$.  The arguments of~\cite{Simon_interior} only require using key integral inequalities, in particular analogues of (2.1) and (2.11), with test functions in $C^{0;q}_c(\mathcal{C}) \cap W^{1,2;q}_0(\mathcal{C})$ and thus~\cite[Theorem 1]{Simon_interior} follows.  Note that the analogue of (2.1) for test functions $h$ vanishing along $\{0\} \times \mathbb{R}^{n-2}$ follows from the first variation of area formula for the closure of the graph of $u$ as an immerse submanifold in $\mathcal{C} \times \mathbb{R} \setminus \{0\} \times \mathbb{R}^{n-1}$ (see the proof of~\cite[Theorem 18.6]{GMT}) and then the analogue of (2.1) holds for any $h \in C^{0;q}_c(\mathcal{C}) \cap W^{1,2;q}_0(\mathcal{C})$ by the logarithmic cutoff function argument. 

Now we will solve (\ref{theorem2_dirprob}) in the case that $\varphi \in C^{0;q}(\overline{\mathcal{C}})$ has $k$-fold symmetry and $\sup_{\mathcal{C}} |\varphi| < \infty$.  Assume $\varphi = 0$ in $B^2_{1/2}(0) \times \mathbb{R}^{n-2} \setminus [0,1/2) \times \{0\} \times \mathbb{R}^{n-2}$.  Approximate $\varphi$ uniformly on compact subsets of $\overline{\mathcal{C}}$ by $\varphi^{(\nu)}$ that are $k$-fold symmetic and periodic with respect to each $y_j$ with period $\rho_{\nu}$ such $\rho_{\nu} \rightarrow \infty$ as $\nu \rightarrow \infty$.  Let $u^{(\nu)} \in C^{1,1/2q;q}(\overline{\mathcal{C}})$ solve (\ref{theorem2_dirprob}) with $u^{(\nu)}$ and $\varphi^{(\nu)}$ in place of $u$ and $\varphi$.  By (\ref{theorem2_sup}), the interior gradient estimate of~\cite[Theorem 1]{Simon_interior}, and Lemma \ref{DgNM_Holder_thm}, $\sup_{\nu} \|u^{(\nu)}\|_{C^{1,\mu;q}(\Omega)} < \infty$ for all $\Omega \subset \subset \mathcal{C}$ for some $\mu \in (0,1/2q]$ depending on $n$, $A^i$, $B$, and $\sup_{\partial \mathcal{C}} |\varphi|$, so after passing to a subsequence $\{u^{(\nu)}\}$ converges in $C^{1;q}$ on compact subsets of $\mathcal{C}$ to some $u \in C^{1,\mu;q}(\mathcal{C})$.  By Schauder estimate Lemma \ref{schauder_strong} for every $\tau \in (0,1/q)$ we have uniformly local interior $C^{1,\tau;q}$ estimates on $u^{(\nu)}$ that are independent of $\nu$ and thus $u \in C^{1,\tau;q}(\mathcal{C})$.  Using local barriers~\cite[Theorem 14.15]{GT}, we obtain uniform modulus of continuity estimates on $u^{(\nu)}$ at points on $\partial \mathcal{C}$ that are independent of $\nu$ and thereby conclude that $u$ extends to an element of $C^{0;q}(\overline{\mathcal{C}})$ such that $u_l = \varphi_l$ on $\partial \mathcal{C}$ for $l = 1,2,\ldots,q$.  Therefore $u$ solves (\ref{theorem2_dirprob}) with the original $\varphi$.


\section{Regularity of $q$-valued solutions} \label{sec:regularity}

Recall from Section~\ref{sec:preliminaries} that to prove Theorem \ref{theorem3}, it suffices to prove (\ref{theorem3_conclusion}) for $u = (u_1,u_2,\ldots,u_q) \in C^{1;q}(B_1(0))$ such that $\|u\|_{C^{1;q}(B_1(0))} \leq 1/2$ and $u$ satisfies the elliptic equation 
\begin{equation} \label{theorem3_equation}
	D_i (A^i(X,u_l,Du_l)) + B(X,u_l,Du_l) = 0 \text{ in } B_1(0) \setminus [0,\infty) \times \{0\} \times B^{n-2}_1(0) 
\end{equation}
for given locally real analytic single-valued functions $A^i(X,Z,P)$ and $B(X,Z,P)$ on $B_1(0) \times (-1,1) \times B^n_1(0)$.  Using arguments similar to those from Section~\ref{sec:nonlineareqns} we can show that $u \in W^{2,2;q}(\Omega)$ for all $\Omega \subset \subset B_1(0)$ and by Lemma \ref{DgNM_Holder_thm} using the fact that $D_k u$ is a satisfies 
\begin{align*}
	&D_i (D_{P_j} A^i(X,u_l,Du_l) D_j D_k u_l + D_Z A^i(X,u_l,Du_l) D_k u_l + D_{X_k} A^i(X,u_l,Du_l)) 
	\\&+ D_{P_j} B(X,u_l,Du_l) D_j D_k u_l + D_Z B(X,u_l,Du_l) D_k u_l + D_{X_k} B(X,u_l,Du_l) = 0 
\end{align*}
in $B_1(0) \setminus [0,\infty) \times \{0\} \times B^{n-2}_1(0)$, we can conclude that $u \in C^{1,\mu;q}(B_1(0))$ for some $\mu \in (0,1/q)$.  The first step to proving Theorem \ref{theorem3} is to establish that $D_y^{\gamma} u \in C^{1,\mu;q}(B_1(0))$ for all $\gamma$.

\begin{lemma} \label{smoothness_thm}
Let $\mu \in (0,1/q)$.  Suppose $u \in C^{1,\mu;q}(B_1(0))$ such that $\|u\|_{C^{1;q}(B_1(0))} \leq 1/2$ and $u$ satisfies (\ref{theorem3_equation}) for given smooth single-valued functions $A^i, B : B_1(0) \times (-1,1) \times B^n_1(0) \rightarrow \mathbb{R}$ and assume (\ref{theorem3_ellipticity}) holds.  Then $D_y^{\gamma} u \in C^{1,\mu;q}(B_1(0))$ for all $\gamma$.
\end{lemma} 

The proofs of Lemma \ref{smoothness_thm} and Theorem \ref{theorem3} require applying $D_y^{\gamma}$ to (\ref{theorem3_equation}).  Observe that 
\begin{align} \label{derivAB}
	D_y^{\gamma} (A^i(X,u_l,Du_l)) &= (D_{P_j} A^i)(X,u_l,Du_l) D_j D_y^{\gamma} u_l + F_{\gamma}^i(X,u_l,\{DD_y^{\beta} u_l\}_{|\beta| \leq |\gamma|-1}), 
		\nonumber \\
	D_y^{\gamma} (B(X,u_l,Du_l)) &= (D_{P_j} B)(X,u_l,Du_l) D_j D_y^{\gamma} u_l + G_{\gamma}(X,u_l,\{DD_y^{\beta} u_l\}_{|\beta| \leq |\gamma|-1}), 
\end{align}
on $B_1(0) \setminus [0,\infty) \times \{0\} \times B^{n-2}_1(0)$ for $l = 1,2,\ldots,q$ for some functions $F_{\gamma}^i$ and $G_{\gamma}$.  To simplify notation, let $a^{ij} = (a^{ij}_1,a^{ij}_2,\ldots,a^{ij}_q)$, $b^j = (b^j_1,b^j_2,\ldots,b^j_q)$, $f_{\gamma}^i = (f_{\gamma,1}^i,f_{\gamma,2}^i,\ldots,f_{\gamma,q}^i)$, and $g_{\gamma} = (g_{\gamma,1},g_{\gamma,2},\ldots,g_{\gamma,q})$ where 
\begin{align} \label{diffeqn_notation} 
	a_l^{ij} &= (D_{P_j} A^i)(X,u_l,Du_l), 
	&f_{\gamma,l}^i &= -F_{\gamma}^i(X,u_l,\{DD_y^{\beta} u_l\}_{|\beta| \leq |\gamma|-1}), \nonumber \\
	b_l^j &= (D_{P_j} B)(X,u_l,Du_l), 
	&g_{\gamma,l} &= -G_{\gamma}(X,u_l,\{DD_y^{\beta} u_l\}_{|\beta| \leq |\gamma|-1}), 
\end{align}
on $B_1(0) \setminus [0,\infty) \times \{0\} \times B^{n-2}_1(0)$.  We can express $f_{\gamma}^i$ as 
\begin{equation} \label{f_gamma}
	f_{\gamma,l}^i = \sum c_{\alpha, j, \beta} (D_{(y,Z,P)}^{\alpha} A^i)(X,u_l,Du_l) 
	\cdot \prod_{k \leq |\alpha_Z|} D_y^{\beta_{Z,k}} u_l \cdot \prod_{k \leq |\alpha_P|} D_y^{\beta_{P,k}} D_{j_k} u_l 
\end{equation}
on $B_1(0) \setminus [0,\infty) \times \{0\} \times B^{n-2}_1(0)$, where $\alpha = (\alpha_y,\alpha_Z,\alpha_P)$ and $D_{(y,Z,P)}^{\alpha} = D_y^{\alpha_y} D^{\alpha_Z} D^{\alpha_P}$ and the sum is taken over all nonzero multi-induces $\alpha$, $\beta_{Z,k}$, and $\beta_{P,l}$ and $1 \leq j_k \leq n$ such that 
\begin{equation} \label{sum_over}
	\alpha_y + \sum_{j \leq |\alpha_Z|} \beta_{Z,j} + \sum_{k \leq |\alpha_P|} \beta_{P,k} = \gamma
\end{equation} 
and $|\beta_{P,k}| < p$ and the coefficients $c_{\alpha,j,\beta}$ are positive integers depending on $\alpha$, $j_1,\ldots,j_{|\alpha|}$, and $\beta_1,\ldots,\beta_{|\alpha|}$.  Note that in (\ref{f_gamma}) and (\ref{sum_over}) assume the convention that sums over $j \leq 0$ equal zero and products over $j \leq 0$ equal one.  We can write a similarly express $g_{\gamma}$ as 
\begin{equation} \label{g_gamma}
	g_{\gamma} = \sum c_{\alpha, j, \beta} (D_{(y,Z,P)}^{\alpha} B)(X,u,Du) 
	\cdot \prod_{k \leq |\alpha_Z|} D_y^{\beta_{Z,k}} u \cdot \prod_{k \leq |\alpha_P|} D_y^{\beta_{P,k}} D_{j_k} u
\end{equation}
on $B_1(0) \setminus [0,\infty) \times \{0\} \times B^{n-2}_1(0)$, where the sum is taken over all nonzero multi-induces $\alpha$, $\beta_{Z,k}$, and $\beta_{P,l}$ and $1 \leq j_k \leq n$ such that (\ref{sum_over}) holds and $|\beta_{P,k}| < p$.  By (\ref{derivAB}), applying $D_y^{\gamma}$ to (\ref{theorem3_equation}) yields 
\begin{equation} \label{theorem3_diffeqn}
	D_i (a^{ij} D_j D_y^{\gamma} u_l) + b^j D_j D_y^{\gamma} u_l = D_i f_{\gamma,l}^i + g_{\gamma,l} 
	\text{ in } B_1(0) \setminus [0,\infty) \times \{0\} \times B^{n-2}_1(0). 
\end{equation}

\begin{proof}[Proof of Lemma \ref{smoothness_thm}]
Given $\eta \in \mathbb{R}^{n-2} \setminus \{0\}$, for each $h \neq 0$ let $\delta_{h,\eta}$ be the difference quotient defined by (\ref{diffquot}) and let $D_{(0,\eta)}$ denote the derivative in the direction $(0,\eta)$.  We will show that $D_{(0,\eta)} D^{\gamma} u \in C^{1,\mu;q}(B_1(0))$ for every $\eta \in \mathbb{R}^{n-2} \setminus \{0\}$ by induction on $|\gamma|$.  This follows by a standard difference quotient argument where we use the Schauder estimates Lemma \ref{schauder_div} to obtain uniform local $C^{1,\mu;q}$ bounds on $\delta_{h,\eta} D_y^{\gamma} u$ that are independent of $h$.  The key to the proof is the fact that such difference quotients $\delta_{h,\eta}$ of $u$ and its derivatives are well-defined.

First we show $D_{(0,\eta)} u \in C^{1,\mu;q}(B_1(0))$ for every $\eta \in \mathbb{R}^{n-2} \setminus \{0\}$.  Let $B_R(x_0,y_0) \subset \subset B_1(0)$ and suppose $|h \eta| \leq R/4$.  By applying $\delta_{h,\eta}$ to (\ref{theorem3_equation}) and using Schauder estimate Lemma \ref{schauder_div}, 
\begin{equation*}
	R \|D \delta_{h,\eta} u\|'_{C^{0,\mu;q}(B_{R/4}(x_0,y_0))} 
	\leq C |\eta| \left( \sup_{B_{R/2}(x_0,y_0)} |\delta_{h,\eta} u| + 1 \right) 
	\leq C |\eta| \left( \sup_{B_R(x_0,y_0)} |D_y u| + 1 \right)
\end{equation*} 
for some constant $C \in (0,\infty)$ depending on $n$, $q$, $\mu$, $A^i$, $B$, $\|u\|_{C^{1,\mu}(B_R(x_0,y_0))}$, and $R$ and independent of $h$.  So given any sequence of $h_j \rightarrow 0$, we can pass to a subsequence $\{\delta_{h_{j'},\eta} u\}$ that converges in $C^{1;q}$ on compact subsets of $B_1(0)$ with a limit in $C^{1,\mu;q}(B_1(0))$.  But $\delta_{h,\eta} u \rightarrow D_{(0,\eta)} u$ pointwise, so in fact $\delta_{h,\eta} u \rightarrow D_{(0,\eta)} u$ in $C^{1;q}$ on compact subsets of $B_1(0)$ as $h \rightarrow 0$ and $D_{(0,\eta)} u \in C^{1,\mu;q}(B_1(0))$. 

Now let $|\gamma| \geq 1$.  We will show that $D_{(0,\eta)} D_y^{\gamma} u \in C^{1,\mu}(B_1(0))$ assuming $D_y^{\beta} u \in C^{1,\mu}(B_1(0))$ whenever $|\beta| < |\gamma|$.  Let $B_R(x_0,y_0) \subset \subset B_1(0)$ and suppose $|h \eta| \leq R/4$.  Recall that applying $D_y^{\gamma}$ to (\ref{theorem3_equation}) yields (\ref{theorem3_diffeqn}).  By applying $\delta_{h,\eta}$ to (\ref{theorem3_diffeqn}) and using the Schauder estimates Lemma \ref{schauder_div}, if $B_R(x_0,y_0) \subset \subset B_1(0)$ and $|h \eta| \leq R/4$, 
\begin{align*}
	R \|D \delta_{h,\eta} D_y^{\gamma} u\|'_{C^{0,\mu;q}(B_{R/4}(x_0,y_0))} 
	&\leq C \left( \sup_{B_{R/2}(x_0,y_0)} |\delta_{h,\eta} D_y^{\gamma} u| + \|DD_y^{\gamma} u\|'_{C^{0,\mu;q}(B_{R/2}(x_0,y_0+h\eta))} \right. 
		\\& \hspace{5mm} \left. + R \|\delta_{h,\eta} f_{\gamma}\|'_{C^{0,\mu;q}(B_{R/2}(x_0,y_0))} 
		+ R^2 \sup_{B_{R/2}(x_0,y_0)} |\delta_{h,\eta} g_{\gamma}| \right) \\
	&\leq C |\eta| \left( \sup_{B_R(x_0,y_0)} |D_y D_y^{\gamma} u| + \|DD_y^{\gamma} u\|'_{C^{0,\mu;q}(B_R(x_0,y_0))} \right. 
		\\& \hspace{5mm} \left. + R \|D_y f_{\gamma}\|'_{C^{0,\mu;q}(B_R(x_0,y_0))} 
		+ R^2 \sup_{B_R(x_0,y_0)} |D_y g_{\gamma}| \right)
\end{align*} 
for some constant $C \in (0,\infty)$ depending on $n$, $q$, $\mu$, $A^i$, $B$, $\|u\|_{C^{1,\mu}(B_R(x_0,y_0))}$, and $R$ and independent of $h$.  So given any sequence of $h_j \rightarrow 0$, we can pass to a subsequence $\{\delta_{h_{j'},\eta} D_y^{\gamma} u\}$ that converges in $C^{1;q}$ on compact subsets of $B_1(0)$ with a limit in $C^{1,\mu;q}(B_1(0))$.  But $\delta_{h,\eta} D_y^{\gamma} u \rightarrow D_{(0,\eta)} D_y^{\gamma} u$ pointwise, so in fact $\delta_{h,\eta} D_y^{\gamma} u \rightarrow D_{(0,\eta)} D_y^{\gamma} u$ in $C^{1;q}$ on compact subsets of $B_1(0)$ as $h \rightarrow 0$ and $D_{(0,\eta)} D_y^{\gamma} u \in C^{1,\mu;q}(B_1(0))$.
\end{proof}

Recall from the beginning of this section that applying $D_y^{\gamma}$ to (\ref{theorem3_equation}) yields (\ref{theorem3_diffeqn}).  By the Schauder estimate Lemma \ref{schauder_div} applied to (\ref{theorem3_diffeqn}), 
\begin{align} \label{theorem3_schauder}
	\|D_y^{\gamma} u\|'_{C^{1,\mu;q}(B_{R/2p}(X_1))} \leq C \left( \sup_{B_{R/p}(X_1)} |D_y^{\gamma} u| 
	+ (R/p)^{1+\mu} [f_{\gamma}]_{\mu;q,B_{R/p}(X_1)} + (R/p)^2 \sup_{B_{R/p}(X_1)} |g_{\gamma}| \right) 
\end{align}
for all $B_R(X_1) \subset \subset B_1(0)$ and some constant $C \in (0,\infty)$ depending on $n$, $q$, $\mu$, $\lambda$, $\sup_{B_R(X_1) \times (-1,1) \times B^n_1(0)} |DA|$, and $\sup_{B_R(X_1) \times (-1,1) \times B^n_1(0)} |DB|$.  Since $f_{\gamma}^i$ and $g_{\gamma}$ can be expressed in terms of $u, Du, DD_y u,\ldots,$ $DD_y^{p-1} u$, we can prove (\ref{theorem3_conclusion}) by inductively computing bounds on $\|DD_y^{\gamma} u\|'_{C^{0,\mu;q}(B_{R/2p}(X_1))}$.  The difficulty is bounding $[f_{\gamma}]_{\mu;q,B_{R/p}(X_1)}$ and $\sup_{B_{R/p}(X_1)} |g_{\gamma}|$ in order to obtain the necessary estimates (\ref{theorem3_conclusion}) on $D_y^{\gamma} u$.  We accomplish this using a modified version of a technique used by Friedman in~\cite{Friedman}.  Since the estimate on $\sup_{B_{R/p}(X_1)} |g_{\gamma}|$ is easier to obtain, we will obtain that estimate first.

\begin{lemma} \label{bound_g_lemma}
Let $p \geq 5$ be a positive integer, $K_0,K,H_0 \geq 1$ be constants, and $B_R(X_1) \subset \subset B_1(0)$.  For some constants $C \in (0,\infty)$ and $H \geq 1$ depending on $n$, $K_0$, $K$, and $H_0$ and independent of $p$ the following holds true.  Suppose $u \in C^{1;q}(B_1(0))$ satisfies $\|u\|_{C^{1;q}(B_1(0))} \leq 1/2$ and (\ref{theorem3_equation}) for given smooth single-valued functions $A^i, B : B_1(0) \times (-1,1) \times B^n_1(0) \rightarrow \mathbb{R}$.  Suppose for any multi-index $\alpha = (\alpha_X,\alpha_Z,\alpha_P)$ with $|\alpha| = k$, where we let $D^{\alpha} = D_X^{\alpha_X} D_Z^{\alpha_Z} D_P^{\alpha_P}$, 
\begin{equation} \label{bound_g_B}
	|D_{(X,Z,P)}^{\alpha} B(X,Z,P)| \leq \left\{ \begin{array}{ll} 
		K_0 K^k R^{-1-|\alpha_X|-|\alpha_Z|} &\text{if } k = 1,2,3, \\
		(k-3)! K_0 K^k R^{-1-|\alpha_X|-|\alpha_Z|} &\text{if } 4 \leq k \leq p, 
	\end{array} \right.
\end{equation}
for $X \in B_R(X_1)$, $|Z| \leq 1/2$, and $|P| \leq 1/2$ and for any multi-index $\beta$ with $|\beta| = s < p$, 
\begin{equation} \label{bound_g_u}
	\frac{s}{R} \sup_{B_{R/p}(X_1)} |D_y^{\beta} u| + \sup_{B_{R/p}(X_1)} |DD_y^{\beta} u| \leq \left\{ \begin{array}{ll} 
		H_0 R^{-s} &\text{if } s = 0,1,2, \\
		(s-2)! H_0 H^{s-3} R^{-s} &\text{if } 3 \leq s < p. 
	\end{array} \right.
\end{equation}
Then $g_{\gamma}$ defined by (\ref{derivAB}) and (\ref{diffeqn_notation}) satisfies 
\begin{equation*}
	R \sup_{B_{R/p}(X_1)} |g_{\gamma}| \leq C (p-2)! H^{p-3} R^{-p}. 
\end{equation*}
\end{lemma}
\begin{proof} 
Suppose we had a function $\Psi(y_1,\ldots,y_{n-2},Z,P_1,\ldots,P_n)$ such that $\Psi(0,0,0) = 0$ and 
\begin{align} \label{g_bound_Psi}
	|D_y^{\alpha_y} D_Z^{\alpha_Z} D_P^{\alpha_P} B(X,Z,P)| \leq D_y^{\alpha_y} D_{z_0}^{\alpha_Z} D_{(z_1,\ldots,z_n)}^{\alpha_P} \Psi(0,0,0) 
\end{align}
for $X \in B_R(X_1)$, $|Z| \leq \sup_{B_R(X_1)} |u|$, and $|P| \leq \sup_{B_R(X_1)} |Du|$ and for all nonnegative integers $\alpha_Z$ and multi-induces $\alpha_y$ and $\alpha_P$ such that $1 \leq |\alpha_y| + \alpha_Z + |\alpha_P| \leq p$.  Further suppose we had functions $v_j(y_1,\ldots,y_{n-2})$, $j = 0,1,\ldots,n$, such that $v_0(0) = 0$ and 
\begin{equation} \label{g_bound_v0} 
	\sup_{B_{R/p}(X_1)} |D_y^{\beta} u| \leq D_y^{\beta} v_0(0) 
\end{equation}
for $0 < |\beta| \leq p$ and for $j = 1,\ldots,n$, $v_j(0) = 0$, $D_y^{\gamma} v_j(0) \geq 0$, and 
\begin{equation} \label{g_bound_vj}
	\sup_{B_{R/p}(X_1)} |D_y^{\beta} D_j u| \leq D_y^{\beta} v_j(0) 
\end{equation}
for $0 < |\beta| < p$.  Recall that $g_{\gamma}$ can be expressed as in (\ref{g_gamma}); that is, 
\begin{equation*}
	g_{\gamma,l} = \sum c_{\alpha, j, \beta} (D_{(y,Z,P)}^{\alpha} B)(X,u_l,Du_l) 
	\cdot \prod_{k \leq |\alpha_Z|} D_y^{\beta_{Z,k}} u_l \cdot \prod_{k \leq |\alpha_P|} D_y^{\beta_{P,k}} D_{j_k} u_l
\end{equation*}
for $l = 1,2,\ldots,q$, where the sum is taken over $\alpha = (\alpha_y,\alpha_Z,\alpha_P)$, $\beta_{Z,k}$, $\beta_{P,l}$, and $1 \leq j_k \leq n$ such that (\ref{sum_over}) holds and $|\beta_{P,l}| < p$ and the $c_{\alpha_y, \alpha_Z, \alpha_P, j, \beta}$ are the positive integers from (\ref{f_gamma}).  We also have 
\begin{equation} \label{deriv_Psi_g}
	D_y^{\gamma} (\Psi(y,v)) = \sum c_{\alpha, j, \beta} (D_{(y,Z,P)}^{\alpha} \Psi)(y,v) 
	\cdot \prod_{k \leq |\alpha_Z|} D_y^{\beta_{Z,k}} v_0 \cdot \prod_{k \leq |\alpha_P|} D_y^{\beta_{P,k}} v_{j_l}
\end{equation}
where the sum is taken over (\ref{sum_over}) and the coefficients $c_{\alpha, j, \beta}$ are the same as above.  Here $v = (v_0,v_1,\ldots,v_n)$ and $\Psi(y,v) = \Psi(y_1,\ldots,y_{n-2},v_0,v_1,\ldots,v_{n-2})$.  Comparing (\ref{g_gamma}) and (\ref{deriv_Psi_g}) using (\ref{g_bound_Psi}), (\ref{g_bound_v0}), and (\ref{g_bound_vj}), 
\begin{equation} \label{g_bound_compare}
	\sup_{B_{R/p}(X_1)} |g_{\gamma}| \leq D_y^{\gamma} (\Psi(y,v(y))) |_{y=0}. 
\end{equation}

To construct $v_0,v_1,\ldots,v_n$ and $\Psi$, first we simplify the setup by letting $v_0(y) = R v(R^{-1} (y_1+\cdots+y_{n-2}))$ and $v_1(y) =\ldots= v_n(y) = v(R^{-1} (y_1+\cdots+y_{n-2}))$ for some function $v(\xi)$ and replacing $\Psi(y_1,\ldots,y_{n-2},Z,P_1,\ldots,P_n)$ with $R^{-1} \Psi(R^{-1} (y_1+\ldots+y_{n-2}), R^{-1} Z+P_1+\cdots+P_n)$ for some function $\Psi(\xi,\zeta)$.  We choose 
\begin{align*}
	&\Psi(\xi,\zeta) = \sum_{k=1}^2 \frac{1}{k!} K_0 K^k (\xi+\zeta)^k + \sum_{k=3}^p \frac{1}{k(k-1)(k-2)} K_0 K^k (\xi+\zeta)^k, \\
	&v(\xi) = H_0 \xi + \frac{1}{2} H_0 \xi^2 + \sum_{s=3}^p \frac{1}{s(s-1)} H_0 H^{s-3} \xi^s.
\end{align*}
It is easy to check (\ref{g_bound_Psi}), (\ref{g_bound_v0}), and (\ref{g_bound_vj}) using (\ref{bound_g_B}) and (\ref{bound_g_u}).

For functions $f(\xi)$ and $g(\xi)$, let $f \ll_p g$ denote that $|D_{\xi}^s f(0,0)| \leq D_{\xi}^s g(0,0)$ for $1 \leq s \leq p$.  We claim that 
\begin{equation} \label{v_power}
	(\xi + (1+n) v(\xi))^k \ll_p c^{k-1} (1+(1+n)H_0)^k \left( \xi^k + \xi^{k+1} + \sum_{s=k+2}^{p^k} \frac{1}{(s-k)^2} H^{s-k-2} \xi^s \right) 
\end{equation}
for $k = 1,2,\ldots,p$ for some constant $c \geq 1$ independent of $k$.  We can see this by induction on $k$.  (\ref{v_power}) obviously holds for $k = 1$.  Let $k \geq 2$.  Assuming (\ref{v_power}) holds with $k-1$ in place of $k$ and multiplying $(\xi + (1+n) v(\xi))$ by $(\xi + (1+n) v(\xi))^{k-1}$, 
\begin{align} \label{v_induction}
	(\xi + (1+n) v(\xi))^k \ll_p {}& 2c^{k-2} (1+(1+n)H_0)^k \left( \xi^k + \xi^{k+1} \right. \nonumber \\
	&\left. + \sum_{s=k+2}^{p^k} \sum_{j=2}^{s-k} \frac{1}{(j-1)^2 (s-j-k+1)^2} H^{s-k-2} \xi^s \right) .
\end{align}
Since
\begin{align} \label{thm5_trick}
	\sum_{j=2}^{s-k} \frac{1}{(j-1)^2 (s-j-k+1)^2} 
	&= \sum_{j=2}^{s-k} \frac{1}{(s-k)^2} \left( \frac{1}{j-1} + \frac{1}{s-j-k+1} \right)^2 \nonumber \\
	&\leq \sum_{j=2}^{s-k} \frac{1}{(s-k)^2} \frac{4}{(j-1)^2} 
	\leq \frac{2\pi^2}{3(s-k)^2}, 
\end{align}
(\ref{v_power}) follows from (\ref{v_induction}) provided we choose $c = 4\pi^2/3$.  Combining the definition of $\Psi$ and (\ref{v_power}), for $p \geq 5$, 
\begin{align*}
	&\Psi(\xi, (n+1) v(\xi)) \ll_p 
	\sum_{k=1}^2 \frac{1}{k!} c^{k-1} (1+(1+n)H_0)^k K_0 K^k \left( \xi^k + \xi^{k+1} + \sum_{s=k+2}^{p^k} \frac{1}{(s-k)^2} H^{s-k-2} \xi^s \right) 
	\\&+ \sum_{k=3}^p \frac{1}{(k-2)^3} c^{k-1} (1+(1+n)H_0)^k K_0 K^k 
		\left( \xi^k + \xi^{k+1} + \sum_{s=k+2}^{p^k} \frac{1}{(s-k)^2} H^{s-k-2} \xi^s \right) .
\end{align*}
It follows that 
\begin{align*}
	&\left. \frac{1}{p!} R \frac{\partial^p}{\partial \xi^p} \Psi(\xi, (n+1) v(\xi)) \right|_{\xi=0} \leq 
	\sum_{k=1}^2 \frac{1}{k! (p-k)^2} c^{k-1} (1+(1+n)H_0)^k K_0 K^k H^{p-k-2} 
	\\&+ \sum_{k=3}^{p-2} \frac{1}{(k-2)^3 (p-k)^2} c^{k-1} (1+(1+n)H_0)^k K_0 K^k H^{p-k-2} 
	\\&+ \frac{1}{(p-3)^3} c^{p-2} (1+(1+n)H_0)^{p-1} K_0 K^{p-1} + \frac{1}{(p-2)^3} c^{p-1} (1+(1+n)H_0)^p K_0 K^p 
\end{align*}
Let $H = \max\{c K (1+(1+n)H_0), 1\}$ so that, using a computation similar to (\ref{thm5_trick}), we have  
\begin{equation*}
	\left. \frac{1}{p!} \frac{\partial^p}{\partial \xi^p} \Psi(\xi, (n+1) v(\xi)) \right|_{\xi=0} 
	\leq \frac{C}{(p-2)^2} K_0 K^3 (1+(1+n)H_0)^3 H^{p-3} 
\end{equation*}
for some constant $C \in (0,\infty)$ independent of $p$.  Thus by (\ref{g_bound_compare}), 
\begin{equation*}
	R \sup_{B_{R/p}(X_1)} |g_{\gamma}| \leq C (p-2)! H^{p-3} R^{-p}
\end{equation*}
for some constant $C = C(n,K_0,K,H_0) \in (0,\infty)$ independent of $p$.
\end{proof}

\begin{lemma} \label{holder_f_lemma}
	Let $p \geq 5$ be a positive integer, $K_0,K,H_0 \geq 1$ be constants, and $B_R(X_1) \subset \subset B_1(0)$.  For some constants $C \in (0,\infty)$ and $H \geq 1$ depending on $n$, $K_0$, $K$, and $H_0$ and independent of $p$ the following holds true.  Suppose $u \in C^{1,\mu;q}(B_1(0))$, where $\mu \in (0,1/q)$, satisfies  $\|u\|_{C^{1;q}(B_1(0))} \leq 1/2$ and (\ref{theorem3_equation}) for given smooth single-valued functions $A^i, B : B_1(0) \times (-1,1) \times B^n_1(0) \rightarrow \mathbb{R}$.  Suppose for any multi-index $\alpha = (\alpha_X,\alpha_Z,\alpha_P)$ with $|\alpha| = k$, 
\begin{equation} \label{holder_f_A}
	|D_{(X,Z,P)}^{\alpha} A(X,Z,P)| \leq \left\{ \begin{array}{ll} 
		K_0 K^k R^{-|\alpha_X|-|\alpha_Z|} &\text{if } k = 1,2,3, \\
		(k-3)! K_0 K^k R^{-|\alpha_X|-|\alpha_Z|} &\text{if } 4 \leq k \leq p, 
	\end{array} \right.
\end{equation}
for $X \in B_R(X_1)$, $|Z| \leq 1/2$, and $|P| \leq 1/2$ and for any multi-index $\beta$ with $|\beta| = s < p$, 
\begin{equation} \label{holder_f_u}
	\frac{s}{R} \sup_{B_{R/p}(X_1)} \hspace{-3mm} |D_y^{\beta} u| + \sup_{B_{R/p}(X_1)} \hspace{-3mm} |DD_y^{\beta} u| 
	+ \left( \frac{R}{s} \right)^{\mu} \hspace{-1mm} [DD_y^{\beta} u]_{\mu;q,B_{R/p}(X_1)} \leq \left\{ \begin{array}{ll} 
		H_0 R^{-s} &\text{if } s = 0,1,2, \\
		(s-2)! H_0 H^{s-3} R^{-s} &\text{if } 3 \leq s < p. 
	\end{array} \right.
\end{equation}
Then $f_{\gamma}^i$ defined by (\ref{derivAB}) and (\ref{diffeqn_notation}) satisfies 
\begin{equation*}
	(R/p)^{\mu} [f_{\gamma}^i]_{\mu, B_{R/p}(X_1)} \leq C (p-2)! H^{p-3} R^{1-p}.
\end{equation*}
\end{lemma}
\begin{proof} 
We will use a similar argument as for Lemma \ref{bound_g_lemma}, except now we need to compute a H\"{o}lder coefficient.  To this we will introduce an auxiliary parameter $t$ such that derivatives of $\Psi$ and $v$ with respect to $t$ bound to H\"{o}lder coefficients of expressions involving $A^i$ and $u$.  The basic idea is to use the fact that the sum, product, and chain rules for computing derivatives with respect to $t$ are similar to sum, product, and composition rules for computing H\"{o}lder coefficients. 

Suppose we had a function $\Psi(y_1,\ldots,y_{n-2},t,Z,P_1,\ldots,P_n)$ such that $\Psi(0,0,0,0) = 0$ and 
\begin{align} \label{holder_f_Psi}
	|D_y^{\alpha_y} D_Z^{\alpha_Z} D_P^{\alpha_P} A^i(X,Z,P)| &\leq D_y^{\alpha_y} D_{z_0}^{\alpha_Z} D_{(z_1,\ldots,z_n)}^{\alpha_P} \Psi(0,0,0,0), \nonumber \\
	(2R/p)^{1-\mu} |D_y D_y^{\alpha_y} D_Z^{\alpha_Z} D_P^{\alpha_P} A^i(X,Z,P)| &\leq D_t D_y^{\alpha_y} D_{z_0}^{\alpha_Z} D_{(z_1,\ldots,z_n)}^{\alpha_P} 
		\Psi(0,0,0,0), 
\end{align}
for $X \in B_R(X_1)$, $|Z| \leq \sup_{B_R(X_1)} |u|$, $|P| \leq \sup_{B_R(X_1)} |Du|$, and $1 \leq |\alpha_y| + \alpha_Z + |\alpha_P| \leq p$ and we had functions $v_j(y_1,\ldots,y_{n-2},t)$, $j = 0,1,\ldots,n$, such that $v_0(0,0) = 0$, 
\begin{align} \label{holder_f_v0}
	\hspace{-6mm}\sup_{B_{R/p}(X_1)} |D_y^{\beta} u| &\leq D_y^{\beta} v_0(0,0) \hspace{9.7mm} \text{for } 0 < |\beta| \leq p, \nonumber \\ 
	\hspace{-6mm}[D_y^{\beta} u]_{\mu;q,B_{R/p}(X_1)} &\leq D_t D_y^{\beta} v_0(0,0) \hspace{5mm} \text{for } 0 \leq |\beta| \leq p, 
\end{align}
and $v_j(0,0) = 0$, $D_y^{\gamma} v_j(0,0) \geq 0$,  
\begin{align} \label{holder_f_vj}
	\sup_{B_{R/p}(X_1)} |D_y^{\beta} D_j u| &\leq D_y^{\beta} v_j(0,0) \hspace{9.8mm} \text{for } 0 < |\beta| < p, \nonumber \\
	[D_y^{\beta} D_j u]_{\mu;q,B_{R/p}(X_1)} &\leq D_t D_y^{\beta} v_0(0,0) \hspace{5mm} \text{for } 0 \leq |\beta| < p, 
\end{align}
for $j = 1,\ldots,n$.  Recall that $f_{\gamma}^i$ can be expressed as in (\ref{f_gamma}) and compute that 
\begin{align} \label{expression_f}
	&[f_{\gamma}^i]_{\mu;q,B_{R/p}(X_1)} 
	\nonumber \\&\leq \sum c_{\alpha, j, \beta} \left( [(D_{(y,Z,P)}^{\alpha} A^i)(X,u,Du)]_{\mu;q,B_{R/p}(X_1)} 
		\prod_{k \leq |\alpha_Z|} \sup |D_y^{\beta_{Z,k}} u| \cdot \prod_{k \leq |\alpha_P|} \sup |D_y^{\beta_{P,k}} D_{j_k} u| 
	\right. \nonumber \\& + \sup |D_{(X,Z,P)}^{\alpha} A^i| \cdot \sum_{k \leq |\alpha_Z|} [D_y^{\beta_{Z,k}} u]_{\mu} 
		\prod_{l \neq k} \sup |D_y^{\beta_{Z,l}} u| \cdot \prod_{k \leq |\alpha_P|} \sup |D_y^{\beta_{P,k}} D_{j_k} u|
	\nonumber \\& \left. + \sup |D_{(X,Z,P)}^{\alpha} A^i| \cdot \prod_{k \leq |\alpha_Z|} \sup |D_y^{\beta_{Z,k}} u|
		\cdot \sum_{k \leq |\alpha_P|} [D_y^{\beta_{P,k}} D_{j_k} u]_{\mu} \prod_{l \neq k} \sup |D_y^{\beta_{P,l}} D_{j_l} u| \right) ,
\end{align}
where $(D_{(y,Z,P)}^{\alpha} A^i)(X,u,Du) = ((D_{(y,Z,P)}^{\alpha} A^i)(X,u_l,Du_l))_{l=1,2,\ldots,q}$, the supremums of derivatives of $A^i$ are taken over $X \in B_R(X_1)$, $|Z| \leq \sup_{B_R(X_1)} |u|$, and $|P| \leq \sup_{B_R(X_1)} |Du|$, and the supremums and the H\"{o}lder coefficients of derivatives of $u$ are taken over $B_{R/p}(X_1)$.  Moreover, 
\begin{align} \label{expression2_f}
	[(D_{(y,Z,P)}^{\alpha} A^i)(X,u,Du)]_{\mu;q,B_{R/p}(X_1)} 
	\leq & \, (2R/p)^{1-\mu} \sup_{B_{R/p}(X_1)} |(D_y D_{(y,Z,P)}^{\alpha} A^i)(X,u,Du)| \nonumber
	\\& + \sup_{B_{R/p}(X_1)} |(D_Z D_{(y,Z,P)}^{\alpha} A^i)(X,u,Du)| \, [u]_{\mu;q,B_{R/p}(X_1)}
	\\& + \sum_{k=1}^n \sup_{B_{R/p}(X_1)} |(D_{P_k} D_{(y,Z,P)}^{\alpha} A^i)(X,u,Du)| \, [D_k u]_{\mu;q,B_{R/p}(X_1)} \nonumber
\end{align}
We also have 
\begin{equation*} 
	D_y^{\gamma} (\Psi(y,t,v)) = \sum c_{\alpha, j, \beta} (D_{(y,Z,P)}^{\alpha} \Psi)(y,t,v) 
	\cdot \prod_{k \leq |\alpha_Z|} D_y^{\beta_{Z,1}} v_0 \cdot \prod_{k \leq |\alpha_P|} D_y^{\beta_{P,k}} v_{j_k}
\end{equation*}
where the sum is taken over (\ref{sum_over}) and the coefficients $c_{\alpha, j, \beta}$ are as in (\ref{f_gamma}).  Here $v = (v_0,v_1,\ldots,v_n)$ and $\Psi(y,t,v) = \Psi(y_1,\ldots,y_{n-2},t,v_0,v_1,\ldots,v_{n-2})$.  Thus 
\begin{align} \label{deriv_Psi_f}
	&D_t D_y^{\gamma} (\Psi(y,t,v(y,t))) 
	= \sum c_{\alpha, j, \beta} \left( D_t ((D_{(y,Z,P)}^{\alpha} \Psi)(y,t,v)) 
		\prod_{k \leq |\alpha_Z|} D_y^{\beta_{Z,k}} v_0 \cdot \prod_{k \leq |\alpha_P|} D_y^{\beta_{P,k}} v_{j_k} \right.
	\nonumber \\& + (D_{(y,Z,P)}^{\alpha} \Psi)(y,t,v) \cdot \sum_{k \leq |\alpha_Z|} D_t D_y^{\beta_{Z,k}} v_0 
		\prod_{l \neq k} D_y^{\beta_{Z,l}} v_0 \cdot \prod_{k \leq |\alpha_P|} D_y^{\beta_{P,k}} v_{j_k} 
	\nonumber \\& \left. + (D_{(y,Z,P)}^{\alpha} \Psi)(y,t,v) \cdot \prod_{k \leq |\alpha_Z|} D_y^{\beta_{Z,k}} v_0
		\cdot \sum_{k \leq |\alpha_P|} D_t D_y^{\beta_{P,k}} v_{j_k} \prod_{l \neq k} D_y^{\beta_{P,l}} v_{j_l} \right) 
\end{align}
where 
\begin{align} \label{expression2_Psi}
	D_t ((D_{(y,Z,P)}^{\alpha} \Psi)(y,t,v)) 
	\leq & \, (D_t D_{(y,Z,P)}^{\alpha} \Psi)(y,t,v) 
	+ (D_Z D_{(y,Z,P)}^{\alpha} \Psi)(y,t,v) D_t v_0 \nonumber
	\\& + \sum_{k=1}^n (D_{P_k} D_{(y,Z,P)}^{\alpha} \Psi)(y,t,v) D_t v_k .
\end{align}
Comparing (\ref{expression_f}) to (\ref{deriv_Psi_f}) and (\ref{expression2_f}) to (\ref{expression2_Psi}) using (\ref{holder_f_Psi}), (\ref{holder_f_v0}), and (\ref{holder_f_vj}), 
\begin{equation} \label{holder_f_compare}
	[f_{\gamma}^i]_{\mu;q,B_{R/p}(X_1)} \leq \left. D_t ((D_{(y,Z,P)}^{\alpha} \Psi)(y,t,v(t,y))) \right|_{y=0,t=0}. 
\end{equation}

To construct $v_0,v_1,\ldots,v_n$ and $\Psi$, first we simplify the setup by letting $v_0(y,t) = R v(R^{-1}(y_1+\cdots+y_{n-2}), (R/p)^{-\mu} t)$ and $v_1(y,t) = \ldots = v_n(y,t) = v(R^{-1}(y_1+\cdots+y_{n-2}), (R/p)^{-\mu} t)$ for some function $v(\xi,\tau)$ and replacing $\Psi(y_1,\ldots,y_{n-2},t,Z,P_1,\ldots,P_n)$ with $\Psi(R^{-1}(y_1+\cdots+y_{n-2}), (R/p)^{-\mu} t, R^{-1} Z+P_1+\cdots+P_n)$ for some function $\Psi(\xi,\tau,\zeta)$.  We choose 
\begin{align*}
	\Psi(\xi,\tau,\zeta) = & \, K_0 K \tau + \sum_{k=1}^2 \frac{1}{k!} (1 + K \tau) K_0 K^k (\xi+\zeta)^k 
		+ \sum_{k=3}^p \frac{1}{k(k-1)(k-2)} K_0 (1 + (k-2) K \tau) K^k (\xi+\zeta)^k \\
	v(\xi,\tau) = & \, (1 + \tau) \left( H_0 \xi + \frac{1}{2} H_0 \xi^2 + \sum_{s=3}^p \frac{1}{s(s-1)} H_0 H^{s-3} \xi^s \right) .
\end{align*}
It is easy to check (\ref{holder_f_Psi}), (\ref{holder_f_v0}), and (\ref{holder_f_vj}) using (\ref{holder_f_A}) and (\ref{holder_f_u}).  Note that to verify (\ref{holder_f_v0}) in the case $|\beta| = p$ we use (\ref{holder_f_u}) with $|\beta| = s = p-1$ and choose $H \geq 5/3$. 

For functions $f(\xi,\tau)$ and $g(\xi,\tau)$, let $f \ll_{p,1} g$ denote that $|D_{\xi}^s f(0,0)| \leq D_{\xi}^s g(0,0)$ and $|D_t D_{\xi}^s f(0,0)| \leq D_{\tau} D_{\xi}^s g(0,0)$ for $1 \leq s \leq p$.  By the (\ref{v_power}), for $k = 1,2,\ldots,p$, 
\begin{equation} \label{v_power2}
	(\xi + (1+n) v(\xi,\tau))^k \ll_{p,1} c^{k-1} \bar H_0^k (1 + k \tau) \left( \xi^k + \xi^{k+1} + \sum_{s=k+2}^{p^k} \frac{1}{(s-k)^2} H^{s-k-2} \xi^s \right) 
\end{equation}
where $c \geq 1$ is a constant independent of $k$ and $\bar H_0 = 1+(n+1)H_0$.  By the definition of $\Psi$ and (\ref{v_power2}), for $p \geq 5$,
\begin{align*}
	&\Psi(\xi,\tau,\zeta) \ll_{p,1} K_0 K \tau + \sum_{k=1}^2 \frac{1}{k!} c^{k-1} (1 + 2kK \tau) \bar H_0^k K_0 K^k 
		\left( \xi^k + \xi^{k+1} + \sum_{s=k+2}^{p^k} \frac{1}{(s-k)^2} H^{s-k-2} \xi^s \right) 
	\\& + \sum_{k=3}^p \frac{1}{k(k-1)(k-2)} c^{k-1} (1 + 2kK \tau) \bar H_0^k K_0 K^k 
		\left( \xi^k + \xi^{k+1} + \sum_{s=k+2}^{p^k} \frac{1}{(s-k)^2} H^{s-k-2} \xi^s \right) 
\end{align*}
It follows that 
\begin{align*}
	&\left. \frac{1}{p!} \frac{\partial^{p+1}}{\partial \tau \partial \xi^p} \Psi(\xi, (n+1) v(\xi)) \right|_{\xi=0, \tau=0} \leq 
	\sum_{k=1}^2 \frac{4}{k! (p-k)^2} c^{k-1} \bar H_0^k K_0 K^{k+1} H^{p-k-2} 
	\\& + \sum_{k=3}^{p-2} \frac{2}{(k-2)^2 (p-k)^2} c^{k-1} K_0 \bar H_0^k K^{k+1} H^{p-k-2} 
	+ \sum_{k=p-1}^p \frac{2}{(k-2)^2} c^{k-1} \bar H_0^k K_0 K^{k+1} 
\end{align*}
Suppose $H$ satisfies $c K \bar H_0 \leq H$ so that, using a computation similar to (\ref{thm5_trick}), we have  
\begin{equation*}
	\left. (R/p)^{\mu} \frac{1}{p!} \frac{\partial^{p+1}}{\partial \tau \partial \xi^p} \Psi(\xi, (n+1) v(\xi)) \right|_{\xi=0, \tau=0} 
	\leq \frac{C}{(p-3)^2} K_0 K^4 \bar H_0^3 H^{p-3} R^{-p}
\end{equation*}
for some constant $C \in (0,\infty)$ independent of $p$.  Thus by (\ref{holder_f_compare}), 
\begin{equation*}
	(R/p)^{\mu} [f_{\gamma}^i]_{\mu;q,B_{R/p}(X_1)} \leq C (p-2)! H_0^3 H^{p-3} R^{-p}
\end{equation*}
for some constant $C = C(n,K_0,K,H_0) \in (0,\infty)$ independent of $p$.
\end{proof}

\begin{lemma} \label{analyticity_lemma}
Let $\mu \in (0,1/q)$ and $u \in C^{1,\mu;q}(B_1(0))$ such that $\|u\|_{C^{1;q}(B_1(0))} \leq 1/2$ and $u$ is a solution to (\ref{theorem3_equation}) for some locally real analytic single-valued functions $A^i, B : B_1(0) \times (-1,1) \times B^n_1(0) \rightarrow \mathbb{R}$ such that (\ref{theorem3_ellipticity}) holds for some constant $\lambda > 0$.  Let $K_0, K, H_0 \geq 1$ and $B_{R_0}(X_0) \subset \subset B_1(0)$.  For some $H = H(n,q,\mu,\lambda,K_0,K,H_0) \geq 1$ the following holds.  Suppose that for every multi-index $\alpha = (\alpha_X,\alpha_Z,\alpha_P)$ with $|\alpha| = k$, 
\begin{equation*}
	|D_{(X,Z,P)}^{\alpha} A(X,Z,P)| + R_0 |D_{(X,Z,P)}^{\alpha} B| \leq \left\{ \begin{array}{ll} 
		K_0 K^k R_0^{-|\alpha_X|-|\alpha_Z|} &\text{if } k = 1,2,3, \\
		(k-3)! K_0 K^k R_0^{-|\alpha_X|-|\alpha_Z|} &\text{if } 4 \leq k \leq p, 
	\end{array} \right.
\end{equation*}
for $X \in B_{R_0}(X_0)$, $|Z| \leq 1/2$, and $|P| \leq 1/2$.  Further suppose that whenever $B_R(X_1) \subseteq B_{R_0}(X_0)$, for every multi-index $\beta$ with $|\beta| = s$, 
\begin{equation} \label{analyticity_lemma_eqn1}
	\frac{s}{R} \|D_y^{\beta} u\|'_{C^{1,\mu;q}(B_{R/2s}(X_1))} \leq \left\{ \begin{array}{ll} 
		H_0 R^{-s} &\text{if } s = 1,2, \\
		(s-2)! H_0 H^{s-2} R^{-s} &\text{if } 3 \leq s < p. 
	\end{array} \right.
\end{equation}
Then for every $B_R(X_1) \subseteq B_{R_0}(X_0)$ and multi-index $\gamma$ with $|\gamma| = p$, 
\begin{equation*}
	(R/p)^{-1} \|D_y^{\gamma} u\|'_{C^{1,\mu;q}(B_{R/2p}(X_1))} \leq (p-2)! H_0 H^{p-2} R^{-p}. 
\end{equation*}
\end{lemma}
\begin{proof} 
First we will apply Lemma \ref{bound_g_lemma} and Lemma \ref{holder_f_lemma} in order to bound the terms $[f_{\gamma}]_{\mu;q,B_{R/p}(X_1)}$ and $\sup_{B_{R/p}(X_1)} |g_{\gamma}|$ in (\ref{theorem3_schauder}).  We need to check (\ref{bound_g_u}) and (\ref{holder_f_u}).  For every $X \in B_{R/p}(X_1) \setminus [0,\infty) \times \{0\} \times \mathbb{R}^{n-2}$, $B_{(p-1)R/p}(X) \subseteq B_{R_0}(X_0)$, so by (\ref{analyticity_lemma_eqn1})
\begin{align} \label{analyticity_lemma_eqn2}
	&(R/s)^{-1} |D_y^{\beta} u(X)| + |DD_y^{\beta} u(X)| + (R/s)^{\mu} [DD_y^{\beta} u]_{\mu;q,B_{(p-1)R/2ps}(X)} 
	\\&\leq 2^{\mu} \|D_y^{\beta} u\|'_{C^{1,\mu;q}(B_{(p-1)R/2ps}(X))} 
	\leq 2^{\mu} (s-2)! H_0 H^{s-2} \left( \frac{p}{(p-1)R} \right)^s
	\leq 2^{\mu} e (s-2)! H_0 H^{s-2} R^{-s}. \nonumber 
\end{align}
Now let $X = (x_1,x_2,y), X' = (x'_1,x'_2,y') \in B_{R/p}(X_1)$ with $x_2 > 0$ and $x'_2 > 0$.  If $|X - X'| < (p-1)R/ps$, then $\hat X = (X+X')/2 \in B_{R/p}(X_1)$ and $X, X' \in B_{(p-1)R/p}(\hat X) \subset B_{R_0}(X_0)$, so by (\ref{analyticity_lemma_eqn2}) 
\begin{equation} \label{analyticity_lemma_eqn3}
	\left( \frac{R}{s} \right)^{\mu} \frac{|DD_y^s u_l(X) - DD_y^s u_l(X')|}{|X - X'|^{\mu}} 
	\leq \left( \frac{R}{s} \right)^{\mu} [DD_y^s u]_{\mu, B_{(p-1)R/2ps}(\hat X)} 
	\leq (s-2)! H_0 \frac{2^{\mu} e}{R^{s+\mu}} 
\end{equation}
for $l = 1,2,\ldots,q$.  If $|X - X'| \geq (p-1)R/ps$, then by using (\ref{analyticity_lemma_eqn2}) to bound $|DD_y^s u_l(X)|$ and $|DD_y^s u_l(Y)|$,
\begin{equation} \label{analyticity_lemma_eqn4}
	\left( \frac{R}{s} \right)^{\mu} \frac{|DD_y^s u_l(X) - DD_y^s u_l(X')|}{|X - X'|^{\mu}} 
	\leq 2^{1+2\mu} e (s-2)! H_0 R^{-s-\mu} 
\end{equation}
for $l = 1,2,\ldots,q$.  By the same computations, (\ref{analyticity_lemma_eqn3}) and (\ref{analyticity_lemma_eqn4}) also hold if instead $x_2 < 0$ and $x'_2 < 0$.  By (\ref{analyticity_lemma_eqn2}), (\ref{analyticity_lemma_eqn3}), and (\ref{analyticity_lemma_eqn4}), (\ref{bound_g_u}) and (\ref{holder_f_u}) hold with $2^{2+2\mu} H_0$ in place of $H_0$ when $s \geq 2$.  By a similar argument (\ref{bound_g_u}) and (\ref{holder_f_u}) hold with $2^{2+2\mu} H_0$ in place of $H_0$ when $s = 1,2$.  

Now by Lemma \ref{bound_g_lemma} and Lemma \ref{holder_f_lemma}, 
\begin{equation} \label{analyticity_lemma_eqn5}
	(R/p)^{\mu} [f_{\gamma}^i]_{\mu;q,B_{R/p}(X_1)} + (R/p) \sup_{B_{R/p}(X_1)} |g_{\gamma}| \leq C (p-2)! H_0^3 H^{p-3} R^{-p}.
\end{equation}
for some constant $C = C(n,q,\mu,K_0,K) \in (0,\infty)$.  Thus by (\ref{theorem3_schauder}) and (\ref{analyticity_lemma_eqn5}), 
\begin{eqnarray*}
	(R/p)^{-1} \|D_y^{\gamma} u\|'_{C^{1,\mu;q}(B_{R/2p}(X_1))}
	&\leq& C \left( p(p-3)! H_0 H^{p-3} R^{-p} + (p-2)! H^{p-3} R^{-p}  \right) \\
	&\leq& C (p-2)! H_0 H^{p-2} R^{-p} \left( \frac{p}{(p-2) H} + \frac{1}{H_0 H} \right) \\
	&\leq& (p-2)! H_0 H^{p-2} R^{-p}
\end{eqnarray*} 
where $C \in (0,\infty)$ denotes constants depending on $n$, $q$, $\mu$, $\lambda$, $K_0$, $K$, and $H_0$ and independent of $p$ and $H$ is large enough that Lemmas \ref{bound_g_lemma} and \ref{holder_f_lemma} hold and $H \geq \max\{4C, 2C/H_0\}$. 
\end{proof}

To complete the proof of Theorem \ref{theorem3}, let $B_R(x_0,y_0) \subset \subset B_1(0)$.  Since $A^i$ and $B$ are real analytic and by the proof of Lemma \ref{smoothness_thm}, there are constants $K, K_0, H_0 \geq 1$ such that for any multi-index $\alpha = (\alpha_X,\alpha_Z,\alpha_P)$ with $|\alpha| = k$ 
\begin{align*}
	|D_{(X,Z,P)}^{\alpha} A^i| + (R/2) |D_{(X,Z,P)}^{\alpha} B| &\leq K_0 K^k (R/2)^{-|\alpha_X|-|\alpha_Z|} &&\text{for } k = 2,3, \\
	|D_{(X,Z,P)}^{\alpha} A^i| + (R/2) |D_{(X,Z,P)}^{\alpha} B| &\leq (k-3)! K_0 K^k (R/2)^{-|\alpha_X|-|\alpha_Z|} &&\text{for } k \geq 4, 
\end{align*}
for $X \in B_R(x_0,y_0)$, $|Z| \leq 1/2$, $|P| \leq 1/2$ and for any multi-index $\beta$ with $|\beta| = s$
\begin{equation*}
	\frac{2s}{R} \|D_y^{\beta} u\|'_{C^{1,\mu;q}(B_{R/4s}(x,y))} \leq H_0 (R/2)^{-s} 
\end{equation*}
whenever $(x,y) \in B_{R/2}(x_0,y_0)$ and $s = 1,2,3,4$.  By the Lemma \ref{analyticity_lemma} and induction, for some $H$ sufficiently large depending on $n$, $q$, $K_0$, $K$, and $H_0$, 
\begin{equation*}
	\frac{2s}{R} \|D_y^{\beta} u\|'_{C^{1,\mu;q}(B_{R/4s}(x,y))} \leq (s-2)! H_0 H^{s-2} (R/2)^{-s} 
\end{equation*}
whenever $(x,y) \in B_{R/2}(x_0,y_0)$ and $s \geq 5$ and in particular (\ref{theorem3_conclusion}) holds true with $C = H_0 H$.

\section{Acknowledgment}

The work in this paper was carried out as part of my thesis at Stanford University.  I would like to thank my advisor Leon Simon for his guidance and encouragement.

\end{document}